\documentclass{amsart}
\usepackage[breaklinks,colorlinks=true]{hyperref}
\usepackage{enumerate}
\newcommand{\apref}[3]{\hyperref[#2]{#1\ref*{#2}#3}}
\input{xy}
\xyoption{all}
\newcommand{\bewend}{\qed}

\newcommand{\mc}[1]{\mathcal #1}
\newcommand{\fd}{\mc F}
\newcommand{\mf}[1]{\mathfrak #1}

\DeclareMathOperator{\Ext}{ext}
\DeclareMathOperator{\Int}{int}
\DeclareMathOperator{\height}{ht}
\DeclareMathOperator{\cl}{cl}
\DeclareMathOperator{\End}{End}
\newcommand{\vs}{\text{vs}}
\newcommand\R{\mathbb{R}}
\newcommand\C{\mathbb{C}}
\newcommand{\h}{\mathbb{H}}
\DeclareMathOperator{\id}{id}
\DeclareMathOperator{\Ima}{Im}
\DeclareMathOperator{\Rea}{Re}
\DeclareMathOperator{\GL}{GL}
\DeclareMathOperator{\PU}{PU}
\DeclareMathOperator{\diag}{diag}
\DeclareMathOperator{\PSp}{PSp}

\newcommand{\dg}{\overline D^g}
\newcommand{\sceq}{\mathrel{\mathop:}=}
\newcommand\mminus{\!\smallsetminus\!}
\newcommand\eg{\mbox{e.\,g., }}
\newcommand\wrt{\mbox{w.\,r.\,t.\@ }}
\newcommand\ie{\mbox{i.\,e., }}
\newcommand{\wt}{\widetilde}
\newcommand{\wh}{\widehat}
\newcommand{\mat}[4]{\begin{pmatrix} #1&#2\\#3&#4\end{pmatrix}}
\newcommand{\textmat}[4]{\left(\begin{smallmatrix} #1&#2 \\ #3&#4
\end{smallmatrix}\right)}

\newcommand{\eps}{\varepsilon}
\newcommand{\Res}{\text{res}}

\newcommand{\hg}{\overline H^g}
\newcommand{\bhg}{\partial_g}
\renewcommand{\bewend}{}

\usepackage{dsfont}
\usepackage{amssymb,amsthm,amsmath}

\theoremstyle{plain}
\newtheorem{proposition}{Proposition}[section]
\newtheorem{theorem}[proposition]{Theorem}
\newtheorem{corollary}[proposition]{Corollary}
\newtheorem{lemma}[proposition]{Lemma}

\theoremstyle{definition}
\newtheorem{definition}[proposition]{Definition}
\newtheorem{convention}[proposition]{Convention}
\newtheorem{defirem}[proposition]{Definition and Remark}
\newtheorem{remdef}[proposition]{Remark and Definition}

\theoremstyle{remark}
\newtheorem{remark}[proposition]{Remark}

\begin{document}

\title{Ford fundamental domains in symmetric spaces of rank one}
\author[A. Pohl]{Anke D. Pohl}
\thanks{The author was partially supported by the International Research Training Group 1133 ``Geometry and Analysis of Symmetries'', the Sonderforschungsbereich/Transregio 45 ``Periods, moduli spaces and arithmetic of algebraic varieties'', and the Max-Planck-Institut f\"ur Mathematik in Bonn.}
\address{Max-Planck-Institut f\"ur Mathematik \\  Vivatsgasse 7 \\
53111 Bonn \\ Germany}
\email{pohl@mpim-bonn.mpg.de}

\keywords{Ford fundamental domains, Isometric fundamental regions, Rank
one symmetric spaces, Isometric spheres, Cygan metric}
\subjclass[2000]{53C35, 22E40, 52C22}

\maketitle

\begin{abstract}
We show the existence of isometric (or Ford) fundamental regions for a large
class of subgroups of the isometry group of any rank one Riemannian symmetric
space of noncompact type. The proof does not use the classification of
symmetric spaces. All hitherto known existence results of isometric
fundamental regions and domains are essentially subsumed by our work.
\end{abstract}

\section{Introduction}

Let $D$ be a rank one Riemannian symmetric space of noncompact type and denote
its full group of Riemannian isometries by $G$. Suppose that $\Gamma$ is a
subgroup of $G$. A subset $Y$ of $D$ is called a \textit{fundamental region for
$\Gamma$ in $D$} if $Y$ is open, the translates of $Y$ by each two elements of
$\Gamma$ are disjoint, and $D$ is covered by the family of $\Gamma$-translates
of
the closure of $Y$. If, in addition, $Y$ is connected, then it is called a
\textit{fundamental domain} for $\Gamma$ in $D$.

In this article we show the existence of so-called isometric (or Ford)
fundamental regions for a large class of subgroups of $G$, see
Theorem~\ref{fund_region2}, Corollary~\ref{region_basic} and
Proposition~\ref{fdspecialcase}. In many situations, the fundamental regions will turn
out
to actually be fundamental domains (see Corollary~\ref{domain_basic} and
Proposition~\ref{fdspecialcase}). Moreover, in Section~\ref{sec_proj}, we show that our
results subsume all previously known existence results of isometric fundamental
regions and domains. 

Let $\Gamma$ be an admissible subgroup of $G$. The distinctive trait of an
isometric fundamental region for $\Gamma$ is that it consists of
two building blocks. One of them is a fundamental region $\fd_\infty$ for the
stabilizer group $\Gamma_\infty$ of $\infty$. The other one is the common
part
of the exteriors of all isometric spheres of $\Gamma$ (see
Section~\ref{sec_fundreg} for a definition). Then the isometric
fundamental region is the set
\[
\fd \sceq \fd_\infty \cap \bigcap_{g\in\Gamma\mminus\Gamma_\infty} \Ext I(g),
\]
where $\Ext I(g)$ denotes the exterior of the isometric sphere $I(g)$ of 
$g\in\Gamma\mminus\Gamma_\infty$. These fundamental regions are of
interest for several applications. For example, they reflect the geometry of
$\Gamma$ in a way
particularly adjusted to the needs of the construction of a symbolic dynamics
for
the geodesic flow on the orbifold $\Gamma\backslash D$ (see \cite{Pohl_diss}).

Our proof of the existence of isometric fundamental regions does not use the
classification of rank one Riemannian symmetric spaces of noncompact type.
This was made possible by the classification-free constructions of all these
spaces provided by Cowling, Dooley, Kor\'anyi and Ricci in \cite{CDKR1} and
\cite{CDKR2} resp.\@  by Kor\'anyi and Ricci in \cite{Koranyi_Ricci} and
\cite{Koranyi_Ricci_proofs}. These constructions appear to be the first ones
which are uniform both on the level of spaces and on the level of  isometry
groups.
The approach in \cite{CDKR1} and \cite{CDKR2} is intimately connected with the
restricted root space decomposition of the Lie algebra of the isometry group of
the symmetric space, and hence it is the method of choice for considerations of
algebraic nature. In contrast, the construction in \cite{Koranyi_Ricci} and
\cite{Koranyi_Ricci_proofs} reflects the geometric side of the spaces. 
Both constructions are amazingly easy to work with. Moreover, one can
effortless switch from one construction to the other and translate
insights and advantages from one model to the other. We
recall both constructions in Section~\ref{sec_spaces}.

Using their work we provide a uniform definition of
the notion of an isometric sphere and its exterior in Section~\ref{sec_fundreg}.
The
uniformity on the level of isometry groups then allows to stick to a
classification-free treatment of
the isometric spheres, which finally results in a classification-free proof of
the existence of isometric fundamental regions.

For real, complex and quaternionic hyperbolic spaces, there already exist
several (different and also non-equivalent) definitions of isometric spheres in
the literature, \eg in \cite{Ford} for the hyperbolic plane, in
\cite{Katok_fuchsian} for the upper half plane model and the disk model of
two-dimensional real hyperbolic space, in \cite{Marden} for three-dimensional
real hyperbolic space, in \cite{Apanasov} and \cite{Apanasov2} for real
hyperbolic spaces of arbitrary dimension, in \cite{Parker}, \cite{Goldman} and
\cite{Kamiya} for complex hyperbolic spaces, and in \cite{Kim_Parker} for
quaternionic hyperbolic spaces. Moreover, for certain subgroups of the isometry
group of real
and complex hyperbolic spaces, the existence of isometric fundamental regions
was
already known. More precisely, Apanasov (\cite{Apanasov} and \cite{Apanasov2})
provides the hitherto most general treatment of groups acting on real hyperbolic
spaces. In \cite{Ford}, Ford investigates the case of the hyperbolic plane.
However, his definition of fundamental region is not equivalent to our
definition. Therefore, his result cannot be compared to our one. Groups acting
on complex hyperbolic spaces are considered by Kamiya in \cite{Kamiya}. In
Section~\ref{sec_proj}, we will investigate
which
definitions of isometric spheres are subsumed by our uniform one, and we will
show that the known isometric fundamental regions are special cases of
Theorem~\ref{fund_region2}.  

Throughout we will use the following notation. If $T$ is a topological space and
$U$ a subset of $T$, then the  closure of $U$ is denoted by
$\overline U$ or $\cl(U)$ and its boundary is denoted by $\partial U$. Moreover,
we write $U^\circ$ for the interior of $U$. The complement of $U$ in $T$ is
denoted by $\complement U$ or $T\mminus U$.

For two arbitrary sets $A$ and $B$, the complement of $B$ in $A$ is written as
$A\mminus B$. If $\sim$ is an equivalence relation on $A$, then $A/_\sim$
denotes the set of equivalence classes. Likewise, if $\Gamma$ a group acting on
$A$, then we write $A/\Gamma$ for the space of right cosets. Moreover, if $p$
is a point of $A$, then $\Gamma_p$ denotes the stabilizer group $\{g\in
\Gamma\mid  gp=p\}$ of $p$ in $\Gamma$.

\section{Classification-free constructions}\label{sec_spaces}

The basic objects of the classification-free construction of rank one
Riemannian symmetric spaces of noncompact type in \cite{CDKR1} and
\cite{CDKR2} are so-called $H$-type (Heisenberg type) algebras. In contrast, in
\cite{Koranyi_Ricci} and \cite{Koranyi_Ricci_proofs}, Kor\'anyi and Ricci
construct these symmetric spaces from so-called $J^2C$-module structures. 

In this section we introduce these notions and we briefly recall the two
constructions. Mostly we duplicate, for the convenience of the reader, material
from their
work. Nevertheless, in parts it can be considered as complementary, \eg
Lemma~\ref{decompnonunique} on the (non-)uniqueness of decompositions of
$H$-type algebras,
the notion of ordered decompositions, and a refined definition of an isomorphism
between $H$-type algebras resp.\@ between $C$-module structures (which corrects
a minor inaccuracy). Moreover, we prove in detail that the Cayley transform is
an isometry, and that the group $M$  given in \cite{CDKR1} and \cite{CDKR2} and
that in \cite{Koranyi_Ricci} and \cite{Koranyi_Ricci_proofs} are
indeed the same.

All omitted proofs can be found in \cite{CDKR1} or \cite{CDKR2} for statements
in the language of $H$-type algebras, and in \cite{Koranyi_Ricci} or
\cite{Koranyi_Ricci_proofs} for those in the language of $J^2C$-modules. As long
as no confusion can arise, each inner product is denoted by
$\langle\cdot,\cdot\rangle$ and its associated norm by $|\cdot|$.
\label{def_inner}

\subsection{$H$-type algebras and the $J^2$-condition}\label{sec_Htype}

A vector space is called \textit{Euclidean} if it is a finite-dimensional real
vector space endowed with an inner product. A Lie algebra is called
\textit{Euclidean} if, in addition to being a Lie algebra, it is a Euclidean
vector space. For a vector space $\mf v$ let $\End_\vs(\mf v)$ denote the group
and vector space of endomorphisms of $\mf v$. If $\mf v$ carries additional
structures, then the elements of $\End_\vs(\mf v)$ are not required to be
compatible with these structures. In particular, if $\mf v$ is Euclidean, then
$\varphi\in\End_\vs(\mf v)$ need not be orthogonal.

\begin{definition}
Let $\mf n$ be a Euclidean Lie algebra. Then $\mf n$ is said to be an \textit{$H$-type algebra} 
if 
\begin{enumerate}[(H1)]
\item\label{H1} there are two subvector spaces $\mf v$, $\mf z$ of $\mf n$ (each of which may be trivial) such that 
\[ [\mf n, \mf z] = \{0\},\quad [\mf n,\mf n]\subseteq \mf z,\]
and $\mf n$ is the orthogonal direct sum of $\mf z$ and $\mf v$, 
\item\label{H2} for all $X\in\mf v$, all $Z\in\mf z$ we have
\[ |J(Z)X| = |Z|\cdot |X|\]
where $J\colon\mf z\to\End_{\text{vs}}(\mf v)$ is the $\R$-linear map defined by 
\[ \langle J(Z)X,Y\rangle = \langle Z, [X,Y]\rangle\] 
for all $X,Y\in\mf v$, all $Z\in\mf z$. The map $J$ is well-defined and unique by Riesz' Representation Theorem (or its finite-dimensional counterpart). 
\end{enumerate}
\end{definition}

If (\apref{H}{H1}{}) holds, then $\mf n$ is either abelian or two-step
nilpotent. In the first case we call $\mf n$ \textit{degenerate}, in the second
\textit{non-degenerate}. 

The construction of a symmetric space from an $H$-type algebra $\mf n$ depends
on the choice of $\mf z$ and $\mf v$ in the
decomposition $\mf z\oplus\mf v$ of $\mf n$. We call the pair $(\mf z,\mf v)$ an
\textit{ordered decomposition} of $\mf n$. The following lemma shows that the
ordered decomposition is unique unless $\mf n$ is non-trivial and abelian (which
precisely is the reason for calling abelian $H$-type algebras degenerate). It
will turn out that both possible ordered decompositions of a non-trivial
abelian $H$-type algebra give rise to the same symmetric space, but in different
models. Nevertheless, this non-uniqueness calls for a careful notion of
isomorphisms between $H$-type algebras, which we will discuss after the lemma.

We denote the center of a Lie algebra $\mf g$ by $Z(\mf g)$.

\begin{lemma}\label{decompnonunique}
Let $\mf n$ be an $H$-type algebra. If $\mf n$ is non-abelian or $\mf n =
\{0\}$, then the ordered decomposition $(\mf z,\mf v)$ of $\mf n$ is unique. In
this case, we have $\mf z=Z(\mf n)$ and $\mf v=Z(\mf n)^\perp$. If $\mf n$ is
abelian and $\mf n \not= \{0\}$, then there are two ordered decompositions of
$\mf n$, namely $(\mf z,\mf v) = (Z(\mf n), \{0\}) = (\mf n, \{0\})$ and $(\mf
z,\mf v)=(\{0\},\mf n)$.
\end{lemma}

\begin{proof} 
Suppose that $(\mf z,\mf v)$ is an ordered decomposition of $\mf n$. Then $\mf
v$ is uniquely determined by $\mf z$, namely $\mf v= \mf z^\perp$. Since $[\mf
z,\mf z] = \{0\}$, we know that $\mf z$ is a subvector space of $Z(\mf n)$ (even
a subalgebra). If $\mf z = \{0\}$, then $[\mf n,\mf n] \subseteq \mf z= \{0\}$.
In this case, $\mf n$ is abelian and $(\mf z,\mf v) = (\{0\},\mf n)$.

Suppose now that $\mf z\not= \{0\}$. We have to prove that $\mf z = Z(\mf n)$.
For contradiction assume that $\mf z\not= Z(\mf n)$, hence $\dim \mf z <
\dim Z(\mf n)$. Then there is a non-trivial element $X\in \mf v\cap Z(\mf n)$.
Fix some $Z\in\mf z$, $Z\not=0$. For all $Y\in\mf v$ it follows that 
\[
 \langle J(Z)X,Y\rangle = \langle Z, [X,Y]\rangle = 0.
\]
Thus $J(Z)X=0$. But then 
\[
 |J(Z)X| = 0 \not= |Z| \cdot |X|,
\]
which is a contradiction to (\apref{H}{H2}{}). Therefore, $\mf z = Z(\mf n)$. 

This shows that for non-abelian $\mf n$ or for $\mf n= \{0\}$, the pair $(\mf z,
\mf v) = (Z(\mf n), Z(\mf n)^\perp)$ is the only candidate for an ordered
decomposition of $\mf n$. Because there is at least one by hypothesis, $(Z(\mf
n), Z(\mf n)^\perp)$ is indeed an ordered decomposition of $\mf n$. For
non-trivial abelian $\mf n$ we have the two candidates $(\{0\}, \mf n)$ and
$(\mf n, \{0\})$, which both clearly satisfy (\apref{H}{H1}{}) and
(\apref{H}{H2}{}).  \bewend
\end{proof}

Let $\mf n$ be a non-trivial abelian $H$-type algebra. Then $\mf n$ admits the
two ordered decompositions $(\mf n, \{0\})$ and $(\{0\}, \mf n)$. The
isomorphism $\id_{\mf n}$ of $\mf n$ as a Euclidean Lie algebra does not respect
these decompositions. In Section~\ref{sec_fromHtoC} we will see that preserving the
decompositions is essential for the bijection between $H$-type algebras and
$C$-module structures. Therefore, from now on, we will always consider an
$H$-type algebra $\mf n$ as being equipped with a (fixed) ordered decomposition
and denote it by $\mf n=(\mf z, \mf v, J)$ or, briefly, by $(\mf z, \mf v, J)$.
\label{def_fixed} Although the map $J$ is determined by $\mf z$ and $\mf v$, we
keep it in the triple to fix a notation for it.

\begin{definition}
Let $\mf n_j = (\mf z_j, \mf v_j, J_j)$, $j=1,2$, be  $H$-type algebras. An
\textit{isomorphism} from $\mf n_1$ to $\mf n_2$ is a pair $(\varphi,\psi)$ of
isomorphisms of Euclidean vector spaces $\varphi\colon \mf z_1\to \mf z_2$ and
$\psi\colon \mf v_1\to \mf v_2$ such that the diagram
\[
\xymatrix{
\mf z_1\times\mf v_1 \ar[r]^{J_1} \ar[d]_{\varphi\times\psi} & \mf v_1 \ar[d]^{\psi}
\\
\mf z_2\times\mf v_2 \ar[r]^{J_2} & \mf v_2
}
\]
commutes.
\end{definition}

The following lemma shows that isomorphism of $H$-type algebras is a refined
notion of isomorphism of Euclidean Lie algebras.

\begin{lemma}\label{isosgleich} 
Let $\mf n_j = (\mf z_j, \mf v_j, J_j)$, $j=1,2$, be  $H$-type algebras and
suppose that the map $(\varphi,\psi)\colon \mf n_1\to\mf n_2$ is an isomorphism.
Then $\varphi\times\psi\colon \mf z_1\oplus \mf v_1 \to \mf z_2\oplus\mf v_2$ is
an isomorphism of Euclidean Lie algebras.
\end{lemma}

\begin{proof} For all $Z\in\mf z_2$ and all $X,Y\in\mf v_1$ we have
\begin{align*}
\big\langle Z, [\psi(X),\psi(Y)]\big\rangle & = \big\langle J_2(Z)(\psi(X)), \psi(Y)\big\rangle = \big\langle \psi\big(J_1(\varphi^{-1}(Z))X\big), \psi(Y)\big\rangle
\\
&  = \big\langle J_1(\varphi^{-1}(Z))X, Y\big\rangle = \big\langle \varphi^{-1}(Z), [X,Y]\big\rangle 
\\
& = \big\langle Z, \varphi([X,Y])\big\rangle.
\end{align*}
Since $\langle\cdot,\cdot\rangle\vert_{\mf z_2\times\mf z_2}$ is non-degenerate, it follows that $\varphi([X,Y])=[\psi(X),\psi(Y)]$. Define $\chi\sceq \varphi\times \psi$ and let $Z_1, Z_2\in\mf z_1$, $X_1, X_2\in\mf v_1$. Then 
\[
 [Z_1+X_1, Z_2+X_2] = [X_1,X_2]
\]
and
\begin{align*}
[\chi(Z_1+X_1), \chi(Z_2+X_2)] & = [\varphi(Z_1) + \psi(X_1), \varphi(Z_2) + \psi(X_2)]
\\ & = [\psi(X_1), \psi(X_2)] = \varphi([X_1,X_2])
\\ & = \chi([Z_1+X_1,Z_2+X_2]).
\end{align*} 
Finally, $\chi$ is clearly an isomorphism of Euclidean vector spaces. This
completes the proof. 
\bewend \end{proof}

\begin{remark}
Suppose that $\mf n_1,\mf n_2$ are non-degenerate $H$-type algebras and let the
map $\chi\colon \mf n_1 \to \mf n_2$ be an isomorphism between $\mf n_1$ and
$\mf n_2$ as Euclidean Lie algebras. Then $\chi(Z(\mf n_1)) = Z(\mf n_2)$. Hence
Lemma~\ref{decompnonunique} implies that $\chi$ is an isomorphism of $H$-type
algebras. In turn, Lemma~\ref{isosgleich} shows that the isomorphisms between
$\mf n_1$ and $\mf n_2$ as $H$-type algebras coincide with the isomorphisms
between $\mf n_1$ and $\mf n_2$ as Euclidean Lie algebras.
\end{remark}

For an $H$-type algebra $(\mf z,\mf v, J)$  we often write $J_ZX$ instead of
$J(Z)X$, and we abbreviate the set $\{ J_Z \mid Z\in\mf
z\}$ with $J_{\mf z}$.

\begin{definition}
An $H$-type algebra $\mf n=(\mf z,\mf v,J)$ is said to satisfy the
\textit{$J^2$-condition} 
if 
\begin{equation}\tag{H3}\label{H3}
\forall\, X\in\mf v\ \forall\, Z_1,Z_2\in\mf z\colon \big( \langle Z_1,Z_2\rangle = 0\ \Rightarrow\ \exists\, Z_3\in\mf z\colon J_{Z_1}J_{Z_2}X = J_{Z_3}X\big).
\end{equation}
\end{definition}
If $\mf n$ is abelian, then \eqref{H3} is trivially satisfied.

\subsection{$C$-module structures and the $J^2$-condition}\label{sec_Cmodules}

A \textit{$C$-module structure} is a triple $(C,V,J)$  consisting of two Euclidean vector spaces $C$ and $V$ and an $\R$-bilinear map $J\colon C\times V\to V$  satisfying the following properties:
\begin{enumerate}[(M1)]
\item\label{M1} there exists $e\in C\mminus\{0\}$ such that $J(e,v)=v$ for all $v\in V$,
\item\label{M2} for all $\zeta\in C$ and all $v\in V$ we have $|J(\zeta,v)| = |\zeta| |v|$.
\end{enumerate}

The cases where $V=\{0\}$ or $C=\R e$ are not excluded. We refer to these as
\textit{degenerate}. If $V\not=\{0\}$ and $C\not= \R e$, then the $C$-module
structure $(C,V,J)$ is called \textit{non-degenerate}. For brevity, a $C$-module
structure $(C,V,J)$ is sometimes called a $C$-module structure\footnote{The
``$C$'' is ``$C$-module structure'' or in ``$C$-module structure on $V$'' does
not refer to the Euclidean space $C$ in the triple $(C,V,J)$. Hence, if
$(C',V,J')$ satisfies (\apref{M}{M1}{}) and (\apref{M}{M2}{}), then it is still
called a $C$-module structure on $V$.} on $V$. One easily proves the following
lemma.

\begin{lemma}\label{eunique}
Let $(C,V,J)$ be a $C$-module structure. If $V\not=\{0\}$, then the element $e$ in (\apref{M}{M1}{}) is uniquely determined and of unit length. 
\end{lemma}

The construction of a symmetric space from a $C$-module structure (satisfying the $J^2$-condition defined below) depends on the choice of $e$ in (\apref{M}{M1}{}) and uses $|e| = 1$.
If $(C,V,J)$ is a $C$-module structure with $V=\{0\}$, then $J$ vanishes everywhere. Thus every element $a\in C\mminus\{0\}$ satisfies $J(a,\cdot)=\id_V$. In this case, we endow $C$ with a distinguished vector $e$ of unit length and fix it (sometimes) in the notation as $(C,e,V,J)$. The in\-fluence of the particular choice of $e$ on the constructed symmetric space is much weaker than that of the different ordered decompositions of degenerate $H$-type algebras. In fact, the choice of $e$ determines the orthogonal decomposition $C=\R e\oplus C'$ (see below). If $e_1, e_2$ are two choices for $e$, then there is an isomorphism between $\R e_1\oplus C'_1$ and $\R e_2 \oplus C'_2$ as Euclidean vector spaces which respects the decompositions. In turn, $(C,e_1,V,J)$ and $(C,e_2,V,J)$ are isomorphic as $C$-module structures (see below for the definition of isomorphism).

If $(C,V,J)$ is a non-degenerate $C$-module structure, then the element $e$ in (\apref{M}{M1}{}) is unique by Lemma~\ref{eunique}. For reasons of uniformity, also in this case, we will often use the notation $(C,e,V,J)$ for $(C,V,J)$.

If $(C,V,J)$ is a $C$-module structure, then we will use $J_\zeta v$ or $\zeta v$ to abbreviate $J(\zeta, v)$. Further, we set $Cv\sceq \{ \zeta v\mid \zeta \in C\}$ for $v\in V$.

A $C$-module structure $(C,V,J)$ is said to satisfy the \textit{$J^2$-condition} if
\begin{equation}\tag{M3}\label{M3}
C(Cv) = Cv \quad\text{for all $v\in V$.}
\end{equation}
In this case, $V$ is called\footnote{As with the ``$C$'' in ``$C$-module structure'', the ``$J^2$'' in ``$J^2$-condition'' and ``$J^2C$-module'' does not refer to the map $J$.} a \textit{$J^2C$-module} and $(C,V,J)$ a \textit{$J^2C$-module structure}.

Let $(C_1,e_1,V_1,J_1)$ and $(C_2,e_2,V_2,J_2)$ be $C$-module structures. An \textit{isomorphism}\footnote{In \cite{Koranyi_Ricci} and \cite{Koranyi_Ricci_proofs} the condition $\varphi(e_1)=e_2$ is omitted from the definition. However, as discussed, this condition is needed and indeed this stronger notion of isomorphism is used in their work.} 
from $(C_1,e_1,V_1,J_1)$ to $(C_2,e_2,V_2,J_2)$ is a pair $(\varphi,\psi)$ of isomorphisms of Euclidean vector spaces $\psi\colon V_1\to V_2$ and $\varphi\colon C_1\to C_2$ with $\varphi(e_1)=e_2$ such that the diagram
\[
\xymatrix{
C_1\times V_1  \ar[r]^{J_1} \ar[d]_{\varphi\times\psi} & V_1 \ar[d]^\psi
\\
C_2\times V_2 \ar[r]^{J_2} & V_2
}
\]
commutes.

The requirement that $\varphi(e_1)=e_2$ is relevant only if one of the $C$-module structures is degenerate. In fact, if $(C_1,e_1,V_1,J_1)$ and $(C_2,e_2,V_2,J_2)$ are non-degenerate $C$-module structures and $(\varphi,\psi)$ is a pair of isomorphisms of Euclidean vector spaces $\psi\colon V_1\to V_2$ and $\varphi\colon C_1\to C_2$ such that $J_2 \circ (\varphi\times\psi) = \psi\circ J_1$, then 
\[
J_2( \varphi(e_1), v) = \psi\big( J_1( e_1, \psi^{-1}(v) )\big)  = v
\]
for each $v\in V$. The uniqueness of $e_2$ shows that $\varphi(e_1)= e_2$.

For a $C$-module structure $(C,e,V,J)$ let $C' \sceq e^\perp$ denote the orthogonal complement of $\R e$ in $C$. For $\zeta = ae+z\in C$ with $a\in\R$ and $z\in C'$ we set $\Rea \zeta \sceq a$, the \textit{real part} 
of $\zeta$, and $\Ima \zeta \sceq z$,\label{def_realpart} the \textit{imaginary part\footnote{Note that, in contrast to the usual definition in complex analysis, if $C=\C$ and $\zeta=a+ib\in\C$, then one has here $\Ima\zeta = ib$.}}
of $\zeta$. Further we set $\overline\zeta \sceq ae - z$, the \textit{conjugate} of $\zeta$. We will use the identification $\R e \to \R$, $ae \mapsto a$, of Euclidean vector spaces.

\subsection{Bijection between $H$-type algebras and $C$-module structures}\label{sec_fromHtoC}

Let $\mf n=(\mf z,\mf v, J)$ be an $H$-type algebra. Endow $\R$ with the
standard inner product and consider the Euclidean direct sum $\mf c \sceq
\R\oplus\mf z$. The map $\widetilde J \colon \mf c \times \mf v \to \mf v$,
defined  by
\[
\widetilde J(t+Z,X) \sceq tX + J_ZX 
\]
for all $t+Z \in \R\oplus\mf z$, $X\in\mf v$, is $\R$-bilinear. Since $\langle 
J_ZX, X\rangle = \langle Z, [X,X]\rangle = 0$, we further find
\begin{align*}
|\widetilde J(t+Z,X)|^2 & = |tX + J_ZX|^2 = t^2 |X|^2 + |J_ZX|^2
\\
&  = t^2 |X|^2 + |Z|^2 |X|^2 = (t^2 + |Z|^2) |X|^2 
\\
& = |t+Z|^2 |X|^2,
\end{align*}
hence $|\widetilde J(t+Z,X)|= |t+Z| |X|$. Moreover, for each $X\in\mf v$ we have
\[
\wt J(1,X) = X.
\]
Therefore, $(\mf c, \mf v, \wt J)$ is a $C$-module structure with $e=1$. The
condition \eqref{H3} is easily seen to be equivalent to 
\begin{equation}\tag{H3'}\label{H3'}
\widetilde J(\mf c)\widetilde J(\mf c) X = \widetilde J(\mf c) X \quad\text{for all $X\in\mf v$.}
\end{equation}
Thus, $(\mf c,\mf v, \wt J)$ satisfies the $J^2$-condition if and only if $\mf
n$ does. This construction provides an assignment of a $C$-module structure to
each $H$-type algebra (with fixed ordered decomposition).

Vice versa, let $(C, e, V,J)$ be a $C$-module structure and let $[\cdot,
\cdot]\colon V\times V\to C'$ be the map defined by 
\begin{equation}\label{bracket}
\langle z, [x,y]\rangle = \langle J(z,x), y\rangle 
\end{equation}
for all $z\in C'$, all $x,y\in V$. Riesz' Representation Theorem (or its
finite-dimensional counterpart) shows that $[\cdot, \cdot]$ is well-defined. We
extend $[\cdot,\cdot]$ to the Euclidean direct sum $C'\oplus V$ by 
\[
[z_1 + v_1 , z_2+v_2] \sceq [v_1,v_2]
\]
for all $z_j+v_j\in C'\oplus V$. This map is $\R$-bilinear. Since $J_z$ is
skew-symmetric for each $z\in C'$ (cf.\@ Section~\ref{sec_B} below), the (extended)
map $[\cdot,\cdot]$ is anti-symmetric. Moreover, $[V,V]\subseteq C'$ and
$[V,C']=[C',C'] = \{0\}$ imply the Jacobi identity for $[\cdot,\cdot]$. Thus,
$C'\oplus V$ endowed with $[\cdot,\cdot]$ is a Euclidean Lie algebra. 
Let $J'\colon C' \to \End_\vs(V)$ denote the map $J'(z)(v) = J(z,v)$. Then
$(C', V, J')$ is an $H$-type algebra. Using the equivalence of \eqref{H3} and
\eqref{H3'}, we see that this $H$-type algebra satisfies the $J^2$-condition if
and only if $(C,V,J)$ does so.

Using the identification $e=1$ from Section~\ref{sec_Cmodules}, the construction
of a $C$-module structure from an $H$-type algebra
\begin{align*}
(\mf z, \mf v, J) &\mapsto (\R\oplus\mf z, 1, \mf v, \widetilde J)
\intertext{and that of an $H$-type algebra from a $C$-module structure}
(C,e,V,J) & \mapsto (C', V, J')
\end{align*}
are inverse to each other. Moreover, one easily sees that these constructions
are equivariant under isomorphisms of $C$-module structures resp.\@ of $H$-type
algebras. More precisely, if $(\varphi,\psi)\colon (\mf z_1,\mf v_1, J_1) \to
(\mf z_2,\mf v_2, J_2)$ is an isomorphism of $H$-type algebras, then 
\[
(\id\times\varphi, \psi)\colon (\R\oplus\mf z_1, 1, \mf v_1, \widetilde J_1) \to (\R\oplus\mf z_2,1,\mf v_2, \widetilde J_2)
\]
is an isomorphism of $C$-module structures. 

Conversely, if $(\varphi,\psi)\colon (C_1,e_1,V_1,J_1)\to (C_2,e_2,V_2,J_2)$ is
an isomorphism of $C$-module structures, then $\varphi(C'_1)=C'_2$ and the map 
\[
(\varphi\vert_{C'_1},\psi)\colon (C'_1,V^{\hphantom{'}}_1,J'_1)\to
(C'_2,V^{\hphantom{'}}_2,J'_2)
\]
is an isomorphism of $H$-type algebras.

\subsection{The model $D$}\label{sec_D}

Let $\mf n = (\mf z, \mf v, J)$ be an $H$-type algebra. Further let $\mf a$ be
a one-dimensional Euclidean Lie algebra and fix an element $H$ of unit length in
$\mf a$. Then $\mf a$ is spanned by $H$. We denote by $\mf s$ the Euclidean
direct sum Lie algebra $\mf a \oplus \mf n = \mf a \oplus \mf z \oplus \mf v$ 
endowed with the Lie bracket that is determined by requiring that 
\begin{align*}
[H,X] & = \tfrac12 X & \text{for all $X\in\mf v$}
\\
[H,Z] & = Z & \text{for all $Z\in\mf z$}
\end{align*}
and that equals the original Lie bracket on $\mf n$, when restricted to $\mf n$. 

Let $\exp(\mf s)$ be the connected, simply connected Lie group with Lie algebra
$\mf s$. We identify the tangent space to $\exp(\mf s)$ at the identity with
$\mf s$. Further we endow $\exp(\mf s)$ with the left-$\exp(\mf s)$-invariant
Riemannian metric that coincides with the inner product on $\mf s$ at the
identity of $\exp(\mf s)$. We parametrize $\exp(\mf s)$ by 
\[
\left\{
\begin{array}{ccl}
\R^+\times\mf z\times\mf v & \to  & \exp(\mf s)
\\
(t,Z,X) & \mapsto & \exp(Z+X)\exp( (\log t)H)
\end{array}
\right.
\]
and set $S\sceq \R^+\times\mf z\times \mf v$. By requiring this parametrization
to be a diffeomorphism and an isometry, $S$ inherits the structure of a
connected, simply connected Lie group with Riemannian metric. The
Campbell-Baker-Hausdorff formula for $\mf n = \mf z\oplus\mf v$ shows that the
group operations on $\exp(\mf s)$ correspond on $S$ to the group operations
\begin{align}
\label{mult_S} \big(t_1, Z_1, X_1\big) \big(t_2, Z_2, X_2\big) & =
\big(t_1t_2, Z_1+t_1Z_2 + \tfrac12 t_1^{1/2} [X_1,X_2], X_1 + t_1^{1/2} X_2\big)
\\
\big(t,Z,X\big)^{-1} & = \big(t^{-1}, -t^{-1}Z, -t^{-1/2}X\big)
\end{align}
 for all $(t_j,Z_j,X_j), (t,Z,X)\in S$. The differential structure on $S$
coincides with the differential structure on the open subset $\R^+\times\mf
z\times\mf v$ of some $\R^m$. Now let 
\[
D\sceq \left\{ (t,Z,X)\in \R\times\mf z\times \mf v \left\vert\ t > \tfrac14|X|^2\right.\right\}
\]
and consider the bijection
\begin{equation}\label{Theta}
\Theta\colon\left\{
\begin{array}{ccl}
\R\times \mf z\times \mf v & \to & \R\times\mf z\times \mf v
\\[1mm]
(t,Z,X) & \mapsto & \big(t+\tfrac14|X|^2, Z, X\big).
\end{array}
\right.
\end{equation}
Then $\Theta(S)=D$, so that we define the structure of a Riemannian manifold on
$D$ by requiring $\Theta$ to be an isometry. The differential structure on $D$
is identical to that of $D$ being an open subset of $\R\times \mf z\times \mf
v$. Moreover, $\Theta$ induces a simply transitive action of $S$ on $D$ by
defining
\[
s\cdot p \sceq \Theta\big(s\Theta^{-1}(p)\big)
\]
for $s\in S$ and $p\in D$. In coordinates $s=(t_s,Z_s,X_s)$ and $p=(t_p, Z_p, X_p)$ this action reads 
\begin{equation}\label{ANaction}
s\cdot p = \big(t_st_p + \tfrac14|X_s|^2 + \tfrac12t_s^{1/2}\langle X_s,X_p\rangle, Z_s + t_sZ_p + \tfrac12t_s^{1/2}[X_s,X_p], X_s + t_s^{1/2}X_p\big).
\end{equation}
Due to the definition of the Riemannian metric, $S$ obviously acts by isometries.
We call $o_D\sceq (1,0,0)$ the \textit{base point} of $D$. The geodesic inversion $\sigma$ of $D$ at $o_D$ is given by 
\[
\sigma (t,Z,X) = \frac{1}{t^2 + |Z|^2} \big(t, -Z, (-t+J_Z)X\big)
\]
for all $(t,Z,X)\in D$. Then $\sigma$ is an isometry, and hence $D$ a symmetric space, if and only if $\mf n$ satisfies the $J^2$-condition. In this case, $D$ has rank one and, if in addition $\mf n$ is non-trivial, then $D$ is of noncompact type.
If $\mf n = \{0\}$, then the constructed space $D$ is the rank one Euclidean
symmetric space $\R$.

Conversely, let $D$ be a rank one symmetric space of noncompact type. Suppose that $\mf g$ is the simple Lie algebra of the Lie group of Riemannian isometries of $D$. Let $\vartheta$ be a Cartan involution of $\mf g$, and let $\mf k$ and $\mf p$ be its $+1$- resp.\@ $-1$-eigenspace. Fix a maximal abelian subalgebra $\mf a$ of $\mf p$ and choose a vector $H\in \mf a$ which spans $\mf a$. Then the decomposition of $\mf g$ into restricted root spaces is 
\[
\mf g = \mf g_{-2\alpha} \oplus \mf g_{-\alpha} \oplus ( \mf a\oplus\mf m) \oplus \mf g_\alpha \oplus \mf g_{2\alpha},
\]
where 
\[
\mf g_\beta \sceq \{ X\in\mf g \mid  [H,X] = \beta(H) X \}
\]
for the linear functional $\beta\colon \mf a \to \R$ and 
\[
\mf m \sceq \{ X\in \mf k \mid [H,X] = 0 \}.
\]
We suppose that $H$ is normalized such that $\alpha(H) = \tfrac12$. If we
set $p\sceq \dim \mf g_\alpha$ and $q\sceq \dim \mf g_{2\alpha}$ and if we endow
$\mf n\sceq \mf g_{2\alpha}\oplus \mf g_\alpha$ with the inner product 
\[
 \langle X, Y\rangle \sceq -\frac{1}{p+4q} B(X,\vartheta Y)
\]
where $B$ is the Killing form of $\mf g$, then $\mf n = (\mf g_{2\alpha}, \mf g_\alpha, J)$ is an $H$-type algebra with $J^2$-condition. The symmetric space constructed from $(\mf g_{2\alpha}, \mf g_\alpha, J)$ is exactly $D$. This means that each rank one Riemannian symmetric space of noncompact type arises from the construction above.

\subsection{The ball model $B$}\label{sec_B}

Let $(C,e,V,J)$ be a $J^2C$-module structure and let $W\sceq C\oplus V$ be the Euclidean direct sum of $C$ and $V$. Consider the unit disc
\[
B \sceq \{ w\in W \mid |w| < 1\}
\]
in $W$ and endow it with the differential structure induced from $W$. In the following we will define a Riemannian metric on $B$ with respect to which $B$ is a rank one Riemannian symmetric space of noncompact type if $(C,V)\not=(\R e, \{0\})$. 

Polarization of the equation in (\apref{M}{M2}{}) shows that we have
\begin{equation}\label{polar}
 \langle \zeta u, \eta v\rangle + \langle \eta u, \zeta v\rangle = 2\langle \zeta,\eta\rangle \langle u,v\rangle
\end{equation}
for all $\zeta,\eta\in C$ and all $u,v\in V$. For  $\zeta\in C$ let $J_\zeta^*$
denote the adjoint of $J_\zeta$. Then  \eqref{polar} implies 
\begin{equation}\label{adjoint}
 J_\zeta^* = J_{\overline\zeta}.
\end{equation}
Moreover we have
\begin{equation}\label{square}
 J_{\overline\zeta} J_{\vphantom{\overline\zeta}\zeta} = |\zeta|^2 \id_V =
J_{\vphantom{\overline\zeta}\zeta} J_{\overline\zeta}.
\end{equation}
Thus, if we set $\zeta^{-1}\sceq |\zeta|^{-2} \overline\zeta$ for $\zeta\in C\mminus\{0\}$, then we have
\begin{equation}\label{inverse}
  \zeta^{-1}( \zeta v) = v = \zeta ( \zeta^{-1} v)
\end{equation}
for all $v\in V$. In Section~\ref{sec_division}  we will see that, if $V\not=\{0\}$, there is a multiplication on $C$ such that $\zeta^{-1}$ is the inverse of $\zeta$.

\begin{definition}
Let $(\zeta, v), (\eta, u)\in W\mminus\{0\}$. Then $(\zeta,v)$ is called
\textit{equivalent} to $(\eta,u)$,  if
either $\zeta = 0 = \eta$ and $u\in Cv$,  or $\zeta\not=0\not=\eta$ and
$\zeta^{-1}v = \eta^{-1}u$. In this case, we write $(\zeta, v) \sim
(\eta, u)$.
\end{definition}

One easily proves that $\sim$ is an equivalence relation on $W\mminus\{0\}$. For an element $w\in W\mminus\{0\}$ let 
\[
Cw \sceq \{ w'\in W\mminus\{0\} \mid w'\sim w\} \cup  \{0\}
\]
denote the equivalence class of $w$ together with the element $0\in W$. Since $B$ is an open subset of the real vector space $W$, we shall identify the tangent space $T_wB$\label{def_tangent} to the point $w\in B$ with $W$. The Riemannian metric $w \mapsto \langle \cdot, \cdot\rangle_{w-}$ on $B$ is defined by 
\begin{equation}\label{metric0}
 \langle X,Y\rangle_{0-} \sceq 4\langle X,Y\rangle \qquad\text{on $T_0B$}
\end{equation}
and, for $w\in B\mminus\{0\}$, by
\begin{equation}\label{B_metric}
 \langle X,Y\rangle_{w-} \sceq 
\begin{cases}
4\frac{\langle X,Y\rangle}{ (1-|w|^2)^2} & \text{if $X,Y\in Cw$}
\\
4\frac{\langle X,Y\rangle}{ 1-|w|^2} & \text{if $X,Y\in (Cw)^\perp$}
\\
0 & \text{if $X\in Cw, Y\in (Cw)^\perp$ (or vice versa).}
\end{cases}
\end{equation}

\subsection{The Cayley transform}

Let $\mf n = (\mf z,\mf v,J)$ be a non-trivial $H$-type algebra which satisfies
the $J^2$-condition. Further let $\mf a$ be a one-dimensional Euclidean Lie
algebra with fixed unit length vector $H$. Suppose that $(C,e,V,J)$ is the
$J^2C$-module structure associated to $(\mf z,\mf v, J)$ (see
Section~\ref{sec_fromHtoC}). We consider the Riemannian symmetric spaces $D$ and
$B$ which are constructed in Section~\ref{sec_D} resp.\@ \ref{sec_B}. 

The \textit{Cayley transform} (in $\R\times\mf z\times \mf v$-coordinates, see
\cite[(2.10a)]{CDKR2})
\[
\mc C\colon\left\{
\begin{array}{ccl}
B & \to & D
\\
(t,Z,X) & \mapsto & \frac{1}{(1-t)^2 + |Z|^2} \big( 1-t^2-|Z|^2, 2Z, 2(1-t+J_Z)X\big)
\end{array}
\right.
\]
is clearly a diffeomorphism from $B$ to $D$. Its inverse (see \cite[(2.10b)]{CDKR2}) is given by 
\[
\mc C^{-1}\colon\left\{
\begin{array}{ccl}
D & \to &  B
\\
(t,Z,X) & \mapsto & \frac{1}{(1+t)^2 + |Z|^2}\big( -1+t^2+|Z|^2, 2Z, (1+t-J_Z)X\big).
\end{array}
\right.
\]

\begin{proposition}\label{Cayleyisometry}
The Cayley transform is an isometry.
\end{proposition}

\begin{proof}
In \cite{CDKR2}, the pullback of the Riemannian metric of $D$ to $B$ via $\mc C$ is described as follows: Let $p\in B$. Denote the tangent space to $B$ at $p$ by $T_pB$, which we identify with $\R\times\mf z\times \mf v$. If $\| \mf X\|_p$ denotes the norm of $\mf X\in T_pB$ induced by the pullback of the Riemannian metric on $D$, then 
\[
\| \mf X\|_0 = 2|\mf X|
\]
for all $\mf X\in T_0B$. For $p\in B\mminus\{0\}$ we have
\[
\|\mf X\|^2_p = 
\begin{cases}
4\frac{|\mf X|^2}{1-|p|^2} & \text{if $\mf X\in T_p^{(1)}$,}
\\[1mm]
4\frac{|\mf X|^2}{(1-|p|^2)^2} & \text{if $\mf X\in \R p\oplus T_p^{(2)}$}
\end{cases}
\]
where 
\[
 T_pB = \R p \oplus T_p^{(2)} \oplus T_p^{(1)}
\]
is a direct sum which is orthogonal \wrt Euclidean metric and Riemannian inner product on $T_pB$. Theorem~6.8 in \cite{CDKR2} provides explicit formulas for $T_p^{(1)}$ and $\R p \oplus T_p^{(2)}$, which we will state in the following. For $X\in \mf v$ let $\mf j(X) \sceq J_{\mf z} X$ and define $\mf k(X)$ to be the orthogonal complement of $\R X$ in 
\[
\mf j(X)^\perp = \{ Y\in \mf v\mid \forall\, Z\in\mf z\colon \langle J_ZX,Y\rangle =0\}.
\]
For $p=(t,Z,X)\in B\mminus\{0\}$ we have
\begin{enumerate}[(i)]
\item $T_p^{(1)} = \mf v$ and $\R p \oplus T_p^{(2)} = \R \oplus \mf z$ if $X=0$,
\item $T_p^{(1)} = \R \oplus \mf z \oplus \mf k(X)$ and $\R p \oplus T_p^{(2)} = \mf j(X) \oplus \R X$ if $(t,Z) = (0,0)$,
\item in the remaining cases,
\begin{align*}
T_p^{(1)} & = \mf k(X) \oplus\left\{ \big(|X|^2 u, |X|^2W, - (u+J_W)(t-J_Z)W \big) \left\vert\  W\in\mf z, u\in \R        \vphantom{\big(|X|^2 u, |X|^2W, - (u+J_W)(t-J_Z)W \big) } \right.\right\}
\intertext{and}
\R p \oplus T_p^{(2)} & = \left\{ \big( (t^2+|Z|^2)u, (t^2+|Z|^2)W, (u+J_W)(t-J_Z)X\big) \left\vert\ W\in\mf z, u\in \R  \vphantom{\big( (t^2+|Z|^2)u, (t^2+|Z|^2)W, (u+J_W)(t-J_Z)X\big)}           \right.\right\}.
\end{align*}
\end{enumerate}
Note that this subsumes the degenerate cases. On $T_0B$, the Riemannian inner
product \eqref{metric0} obviously coincides with this one. For $p\in
B\mminus\{0\}$, \eqref{B_metric} implies that it suffices to show that $\R p
\oplus T_p^{(2)} = Cp$. To that end let $p=(\zeta, v) = (t,Z,X) \in
B\mminus\{0\}$ (hence $\zeta=(t,Z)\in C = \R\times \mf z$ and $v=X\in V=\mf
v$). If $v=0$, then 
\[
Cp = C\times \{0\} = C = \R\times\mf z = \R p \oplus
T_p^{(2)}.
\]
If $\zeta = 0$, then 
\[
Cp = Cv = (\R + J_{\mf z})X = \R p \oplus
T_p^{(2)}.
\]
If $\zeta \not=0$ and $v\not=0$, then $(\eta, u) \in Cp$ if and only
if $\eta \not=0$ and $\zeta^{-1} v  = \eta^{-1} u$, or $(\eta, u) = 0$. This
means that in both cases 
\[
u = J_\eta J_{\zeta^{-1}} v = |\zeta|^{-2} J_\eta J_{\overline\zeta} v = J_{|\zeta|^{-2}\eta}J_{\overline\zeta} v.
\]
Hence
\begin{align*}
Cp & = \left\{ \big( \eta, J_{|\zeta|^{-2}\eta}J_{\overline\zeta} v\big) \left\vert\ \eta\in C  \vphantom{\big( \eta, J_{|\zeta|^{-2}\eta}J_{\overline\zeta} v\big)}   \right.\right\} = \left\{ \big( |\zeta|^2\xi, J_\xi J_{\overline\zeta} v\big) \left\vert\ \xi\in C \vphantom{\big( |\zeta|^2\xi, J_\xi J_{\overline\zeta} v\big) }   \right.\right\} = \R p \oplus T_p^{(2)}.
\end{align*}
This completes the proof. 
\bewend \end{proof}

Note that for $\mf n = \{0\}$ resp.\@ $(C,V) = (\R e, \{0\})$ Proposition~\ref{Cayleyisometry} shows that the model $B$ is isometric to the Euclidean symmetric space $\R$.

\subsection{The map $\beta_2$}\label{def_beta2}

Let $(C,e,V,J)$ be a $J^2C$-module structure. In this section we introduce a map
\[
\beta_2\colon V\times V\to C,
\]
which will be shown to be $C$-hermitian, \ie $\beta_2$ is $\R$-bilinear and for all $u,v\in V$ we have $\beta_2(u,v) = \overline{\beta_2(v,u)}$. The map $\beta_2$ encodes the inner product and the Lie bracket on $V$.
We define $\beta_2\colon V\times V\to C$ by\footnote{This is $J^*$ in \cite{Koranyi_Ricci}.}
\[ \langle \beta_2(v,u),\zeta\rangle \sceq \langle J_\zeta u, v\rangle \quad\text{for all $\zeta\in C$.}\]

\begin{lemma} \label{conj_scalar}
For $\zeta,\eta\in C$ we have $\langle \eta, \overline\zeta\rangle = \langle \overline\eta, \zeta\rangle$.
\end{lemma}

\begin{proof} Let $\zeta = a + x$ and $\eta = b + y$ ($a,b\in\R$, $x,y\in C'$) be the decompositions of $\zeta$ and $\eta$ \wrt $C=\R \oplus C'$. 
Since $\R$ and $C'$ are orthogonal, we find
\begin{align*}
\langle \overline\zeta,\eta\rangle & = \langle a, b\rangle - \langle x,b\rangle + \langle a,y\rangle - \langle x,y\rangle
\\ & = \langle a,b\rangle - \langle x,y\rangle
\\ & = \langle a, b\rangle + \langle x, b\rangle - \langle a, y\rangle - \langle x,y\rangle
\\ & = \langle \zeta,\overline\eta\rangle. 
\end{align*}
This proves the lemma. 
\bewend \end{proof}

\begin{proposition} \label{beta_2_real}The map $\beta_2\colon V\times V\to C$ is
$\R$-bilinear. Further we have
\begin{enumerate}[{\rm(i)}]
\item \label{skew} $\beta_2(v,u) = \overline{\beta_2(u,v)}$ for all $u,v\in V$,
\item \label{inner} $\beta_2(v,v) = \langle v,v\rangle$ for all $v\in V$.
\end{enumerate}
\end{proposition}

\begin{proof} One easily sees that $\beta_2$ is $\R$-bilinear.  Using Lemma~\ref{conj_scalar} we find
\[ \langle \beta_2(v,u), \zeta\rangle = \langle J_\zeta u, v\rangle = \langle u, J_{\overline \zeta} v\rangle =  \langle\beta_2(u,v), \overline \zeta\rangle = \langle \overline{\beta_2(u,v)}, \zeta\rangle\]
for all $u,v\in V$ and all $\zeta\in C$. Hence $\beta_2(v,u) = \overline{\beta_2(u,v)}$, which proves \eqref{skew}. 
Finally let $v\in V$. From $\beta_2(v,v) = \overline{\beta_2(v,v)}$ it follows that $\beta_2(v,v)\in\R$. Then
\[ \beta_2(v,v) = \langle \beta_2(v,v), e\rangle = \langle ev, v\rangle = \langle v,v\rangle,\]
which shows \eqref{inner}. 
\bewend \end{proof}

\begin{lemma}\label{re_and_im} For $v_1,v_2\in V$ we have  
\[ \Rea \beta_2(v_1,v_2) = \langle v_1,v_2\rangle \quad\text{and}\quad \Ima\beta_2(v_1,v_2) = [v_2,v_1].\]
\end{lemma}

\begin{proof} Let $v_1,v_2\in V$. By Proposition~\ref{beta_2_real} we have
\begin{align*} 
|v_1|^2 + 2\langle v_1,v_2\rangle + |v_2|^2 & = |v_1+v_2|^2 = \beta_2(v_1+v_2,v_1+v_2) 
\\
& = |v_1|^2 + \beta_2(v_1,v_2) + \beta_2(v_2,v_1) + |v_2|^2
\\ 
& = |v_1|^2 + \beta_2(v_1,v_2) + \overline{\beta_2(v_1,v_2)} + |v_2|^2
\\ 
& = |v_1|^2 + 2\Rea\beta_2(v_1,v_2) + |v_2|^2.
\end{align*}
Hence $\Rea\beta_2(v_1,v_2) = \langle v_1,v_2\rangle$. To show the second claim,
note that Proposition~\ref{beta_2_real} implies that 
\[
\Ima \beta_2(v_1,v_2) = \tfrac12\beta_2(v_1,v_2) - \tfrac12\overline{\beta_2(v_1,v_2)} = \tfrac12\beta_2(v_1,v_2) - \tfrac12\beta_2(v_2,v_1).
\]
For each $\zeta\in C'$ it follows that 
\begin{align*}
\langle \zeta, \Ima \beta_2(v_1,v_2)\rangle & = \tfrac12 \langle\zeta,\beta_2(v_1,v_2)\rangle - \tfrac12\langle\zeta,\beta_2(v_2,v_1)\rangle
\\ & = \tfrac12\langle J_\zeta v_2, v_1\rangle - \tfrac12\langle J_\zeta v_1,v_2\rangle
\\ & = \tfrac12\langle J_\zeta v_2,v_1\rangle + \tfrac12\langle v_1, J_\zeta v_2\rangle
\\ & = \langle J_\zeta v_2,v_1\rangle = \langle \zeta, [v_2,v_1]\rangle.
\end{align*}
Since $\langle\cdot,\cdot\rangle\vert_{C'\times C'}$ is non-degenerate and $\Ima
\beta_2(v_1,v_2)\in C'$ and $[v_2,v_1]\in C'$, it follows that
$\Ima\beta_2(v_1,v_2) = [v_2,v_1]$. 
\bewend \end{proof}

\subsection{The isometry group}\label{sec_isomgroup}

Let $\mf n = (\mf z,\mf v, J)$ be a non-trivial $H$-type algebra which
satisfies the $J^2$-condition, and let $\mf a$ be a one-dimensional Euclidean
Lie algebra. Construct the Euclidean Lie algebra $\mf s$ and the spaces $S$ and
$D$ as in Section~\ref{sec_D}. We denote by $(C,e,V,J)$ the  $J^2C$-module
structure $(\R\oplus\mf z, 1, \mf v, \wt J)$ which is isomorphic to $\mf n$.
Let $G$ denote the full isometry group of $D$. Suppose that $N$ resp.\@ $A$ are
the connected, simply connected Lie groups with Lie algebra $\mf n$ resp.\@ $\mf
a$. Then $N$ and $A$ are subgroups of $S$, more precisely, $S$ is the semidirect
product $AN$. In the parametrization of $S$, the groups $N$ and $A$ are given by
\begin{align*}
N & = \left\{ n_{(Z,X)} \sceq (1,Z,X) \left\vert\  (Z,X) \in \mf z\times\mf v \vphantom{n_{(Z,X)}} \right.\right\}
\intertext{and}
A & = \left\{ a_t \sceq (t,0,0) \left\vert\  t\in \R^+ \vphantom{a_t} \right.\right\}.
\end{align*}
Let $K$ be the stabilizer of the base point $o_D=(1,0,0)$ in $G$ and let $M\sceq Z_K(A)$ be the centralizer of $A$ in $K$. Recall the geodesic inversion 
\[
 \sigma(t,Z,X) = \frac{1}{t^2 + |Z|^2}\big( t, -Z , (-t+J_Z)X\big)
\]
at the origin $o_D$ from Section~\ref{sec_D}.

\begin{theorem}[Theorem~6.4 in \cite{CDKR2}] 
The Lie group $G$ has the Bruhat decomposition $MAN\cup N\sigma MAN$. 
\end{theorem}

Let $\mc X\sceq \R\times\mf z\times \mf v \cup \{\infty\}$ be the one-point compactification of $\R\times\mf z\times\mf v$, where $\infty$ denotes the point at infinity. A compactification of $D$ is then given by the closure of $D$ in $\mc X$, namely
\begin{align*}
\dg & = \left\{ (t,Z,X) \in \R\times\mf z\times \mf v \left\vert\ t\geq \tfrac14|X|^2\right.\right\} \cup \{\infty\}
\\
& = \left\{ (\zeta, v) \in C\times V \left\vert\  \Rea \zeta \geq  \tfrac14|v|^2\right.\right\} \cup \{\infty\}.
\end{align*}
The space $\dg$ is precisely the geodesic compactification of $D$, see, \eg
\cite[Section~I.2]{Borel_Ji} or \cite[Proposition~1.7.6]{Eberlein}.
Let $\overline B$ be the closed unit ball in $W=\R\times\mf z\times \mf v$.
Then the Cayley transform $\mc C\colon B\to D$ extends (uniquely) to a
homeomorphism $\overline B \to \dg$. Therefore, \cite[Corollary~6.2]{CDKR2} amounts
to the following proposition.

\begin{proposition}\label{extends}
The action of $G$ extends continuously to $\dg$.
\end{proposition}

The Bruhat decomposition of $G$ implies that the stabilizer $G_\infty$ of $\infty$ in $G$ equals $MAN$.\label{def_Ginfty}

For future purposes we need explicit formulas for the action of the groups $M$, $A$ and $N$ on $\dg$. The action of $N$ and $A$ in $\R\times\mf z\times\mf v$-coordinates is already given in \eqref{ANaction}. Suppose that $(\zeta, v)\in \dg\mminus\{\infty\}$. For $a_t\in A$ we have
\[
a_t\infty = \infty \qquad\text{and}\qquad a_t(\zeta, v) = (t\zeta, t^{1/2}v).
\]
For $n_{(Z,X)}\in N$ we get $n_{(Z,X)}\infty = \infty$ and
\[
n_{(Z,X)}(\zeta, v)  = \big( \tfrac14|X|^2 + Z + \zeta + \tfrac12\beta_2(v,X), X+v\big).
\]
In $C\times V$-coordinates, the geodesic inversion $\sigma$ reads as
\[
 \sigma\infty = 0,\quad \sigma 0 = \infty,
\]
and for $(\zeta,v)\in\dg\mminus\{0,\infty\}$,
\begin{align*}
\sigma(\zeta, v) & = \frac{1}{(\Rea \zeta)^2 + |\Ima \zeta|^2} \big( \Rea\zeta, -\Ima\zeta, J_{-\Rea\zeta + \Ima\zeta}v\big)
\\
& = |\zeta|^{-2} \big( \overline\zeta, J_{-\overline\zeta}v\big) = \big(\zeta^{-1}, J_{-|\zeta|^{-2}\overline\zeta}v\big)
\\
& = \big(\zeta^{-1},-\zeta^{-1}v\big) = \zeta^{-1}(1,-v).
\end{align*}
The Cayley transform in $C\times V$-coordinates is
\begin{align*}
\mc C & \colon\left\{
\begin{array}{ccl}
B & \to & D 
\\
(\zeta,v) & \mapsto & |1-\zeta|^{-2}\big( 1+ 2\Ima\zeta - |\zeta|^2, 2(1-\overline\zeta)v \big)
\end{array}
\right.
\intertext{with}
\mc C^{-1} & \colon\left\{
\begin{array}{ccl}
D & \to & B 
\\
(\eta,u) & \mapsto & |1+\eta|^{-2}\big( -1+2\Ima\eta + |\eta|^2, 2(1+\overline\eta)u\big).
\end{array}
\right.
\end{align*}
The following proposition provides the explicit form of the group $M$.

\begin{proposition}
The group $M$ is the group of automorphisms of the $J^2C$-module structure
$(C,e,V,J)$, that is, the elements of $M$ are the pairs $(\varphi,\psi)$ of
orthogonal endomorphisms $\varphi\colon C\to C$ with $\varphi(e) = e$ and
$\psi\colon V\to V$ such that the diagram
\[
\xymatrix{
C\times V \ar[r]^J \ar[d]_{\varphi\times\psi} & V \ar[d]^{\psi}
\\
C\times V \ar[r]^J & V
}
\]
commutes. If $(\varphi,\psi)\in M$, then its action on $\dg$ is given by $(\varphi,\psi)(\infty) = \infty$ and $(\varphi,\psi)(\zeta,v) = (\varphi(\zeta),\psi(v))$ for $(\zeta,v)\in\dg\mminus\{\infty\}$.
\end{proposition}

\begin{proof}
Let $\wt G \sceq \mc C^{-1} G \mc C$. For each $g\in G$ set $\wt g\sceq \mc
C^{-1} g \mc C$, and for each subset $T$ of $G$ let $\wt T\sceq \mc C^{-1} T \mc
C$ be the corresponding subset of $\wt G$. Clearly, $\wt G$ is the full isometry
group of $B$. We will first characterize the centralizer $Z_{\wt K}(\wt A)$ of
$\wt A$ in $\wt K$ as a subgroup of $\wt G$. Let $\wt M$ denote the automorphism
group of $(C,e,V,J)$ and define the action of $(\varphi,\psi)\in \wt M$ on $B$
by $(\varphi,\psi)(\zeta,v) \sceq (\varphi(\zeta),\psi(v))$. We will show that
$\wt M = Z_{\wt K}(\wt A)$. By \cite[Proposition~4.1]{Koranyi_Ricci_proofs} we have
that $\wt M\subseteq \wt K$. Let $a_t\in A$ and $(\zeta, v)\in B$. Then one
easily calculates that 
\begin{align*}
\wt a_t(\zeta, v) & = \mc C^{-1}\circ a_t \circ \mc C (\zeta, v)
\\
& = \left| |1-\zeta|^2 + t(1+2\Ima\zeta - |\zeta|^2) \right|^{-2} \cdot
\\
& \quad \cdot \Big(-|1-\zeta|^4 + 4t\Ima\zeta + t^2 \left|1+2\Ima\zeta-|\zeta|^2\right|^2, 
\\
& \quad\qquad\qquad 4t^{1/2} |1-\zeta|^2 \big(|1-\zeta|^2 + t(1-\Ima\zeta - |\zeta|^2) \big) \big( (1-\overline\zeta)v\big) \Big).
\end{align*}
Suppose that $\wt m = (\varphi,\psi) \in \wt M$. Since $\varphi(e) = e$, we have that $\varphi(\overline\zeta) = \overline{\varphi(\zeta)}$ and $\varphi(\Ima\zeta) = \Ima \varphi(\zeta)$ for each $\zeta\in C$. Moreover, $|\zeta| = |\varphi(\zeta)|$ for each $\zeta\in C$.  Then the first component of $\wt m\circ \wt a_t(\zeta, v)$ is given by 
\begin{align*}
& \varphi\left( \frac{ -|1-\zeta|^4 + 4t\Ima\zeta + t^2 \left|1+2\Ima\zeta - |\zeta|^2\right|^2}{\big| |1-\zeta|^2 + t(1+2\Ima\zeta - |\zeta|^2) \big|^{2}}\right)=
\\
& \hspace{3cm} 
= \frac{ -|1-\zeta|^4 + 4t\Ima\varphi(\zeta) + t^2\left|1+2\Ima\zeta - |\zeta|^2\right|^2}{ \big||1-\zeta|^2 + t(1+2\Ima\zeta - |\zeta|^2) \big|^{2}}
\\
& \hspace{3cm} 
= \frac{ -|1-\varphi(\zeta)|^4 + 4t\Ima\varphi(\zeta) + t^2\left|1+2\Ima\varphi(\zeta) - |\varphi(\zeta)|^2\right|^2}{ \big||1-\varphi(\zeta)|^2 + t(1+2\Ima\varphi(\zeta) - |\varphi(\zeta)|^2) \big|^{2}}.
\end{align*}
This is the first component of $\wt a_t\circ\wt m(\zeta,v)$. The second component of $\wt m\circ \wt a_t(\zeta, v)$ reads as
\begin{align*}
&\psi\left( \frac{\big( |1-\zeta|^2 + t(1-2\Ima\zeta -|\zeta|^2) \big)\big( (1-\overline\zeta)v\big)}{\big| |1-\zeta|^2 + t(1+2\Ima\zeta - |\zeta|^2) \big|^{2}}\right) = 
\\
& \hspace{2cm}
= \frac{\varphi\big( |1-\zeta|^2 + t(1-2\Ima\zeta -|\zeta|^2) \big) \psi\big( (1-\overline\zeta)v\big) }{\big| |1-\zeta|^2 + t(1+2\Ima\zeta - |\zeta|^2) \big|^{2}}
\\
& \hspace{2cm}
= \frac{\big( |1-\zeta|^2 + t(1-2\Ima\varphi(\zeta) -|\zeta|^2) \big) \big(
\varphi(1-\overline\zeta)\psi(v)\big) }{\big| |1-\zeta|^2 + t(1+2\Ima\zeta -
|\zeta|^2) \big|^{2}}
\\
& \hspace{2cm}
= \frac{\big( |1-\varphi(\zeta)|^2 + t(1-2\Ima\varphi(\zeta) -|\varphi(\zeta)|^2) \big) \big( (1-\overline{\varphi(\zeta)})\psi(v)\big) }{\big| |1-\varphi(\zeta)|^2 + t(1+2\Ima\varphi(\zeta) - |\varphi(\zeta)|^2) \big|^{2}}.
\end{align*}
Clearly, this is the second component of $\wt a_t\circ\wt m(\zeta,v)$. Hence, $\wt m$ commutes with each $\wt a_t$. This shows that $\wt M \subseteq Z_{\wt K}(\wt A)$.

Conversely, suppose that $\wt m \in Z_{\wt K}(\wt A)$. We have to show that
$\wt m(1,0) = (1,0)$. Then \cite[Proposition~4.1]{Koranyi_Ricci_proofs} implies that
$\wt m\in\wt M$. For each $t\in \R^+$, we have 
\[
 \wt a_t(1,0) = \lim_{n\to\infty} \wt a_t\big( 1-\tfrac1n,0\big) = (1,0).
\]
Thus, $\wt a_t( \wt m(1,0)) = \wt m(1,0)$. The only points on $\partial B$ that
are invariant under each $\wt a_t$ are $(1,0)$ and $(-1,0)$. Seeking a
contradiction assume that $\wt m(1,0)=(-1,0)$. Since $\wt \sigma = -\id$, we 
get $\wt\sigma\circ\wt m(1,0) = (1,0)$. By
\cite[Proposition~4.1]{Koranyi_Ricci_proofs} we have $\wt\sigma\circ\wt m \in \wt M$.
Our previous argument then shows that $\wt\sigma\circ\wt m$ commutes with all
$\wt a_t$. Therefore
\[
\wt a_t\circ \wt\sigma\circ\wt m =\wt\sigma\circ\wt m\circ \wt a_t = \wt\sigma \circ \wt a_t \circ \wt m,
\]
which means that $\wt\sigma$ commutes with each $\wt a_t$. This is a
contradiction. Hence it follows that $\wt m(1,0) = (1,0)$, which by
\cite[Proposition~4.1]{Koranyi_Ricci_proofs} shows that $\wt m\in \wt M$. Therefore
$\wt M = Z_{\wt K}(\wt A)$.
Now let $\wt m = (\varphi,\psi) \in \wt M$ and set $m\sceq \mc C \circ \wt
m\circ \mc C^{-1}$. For $(\eta,u)\in D$ we find
\begin{align*}
m(\eta, u) & = \mc C \circ \wt m\circ \mc C^{-1}(\eta, u) 
\\
& = \mc C \circ \wt m\left( |1+\eta|^{-2}\big( 1+2\Ima\eta + |\eta|^2, 2(1+\overline\eta)u \big) \right)
\\
& = \mc C\left( |1+\varphi(\eta)|^{-2} \big( 1+2\Ima\varphi(\eta) + |\varphi(\eta)|^2, 2(1+\overline{\varphi(\eta)})\psi(u) \big) \right) 
\\
& = ( \varphi(\eta), \psi(u) ).
\end{align*}
The remaining claim follows directly from continuity or, alternatively, from the
extension of the Cayley transform to $\overline B$. 
\bewend \end{proof}

\section{Isometric fundamental regions}\label{sec_fundreg}

Throughout this section let $(C,e,V,J)$ be a $J^2C$-module structure such that
$(C,V)\not=(\R e, \{0\})$ and suppose that $(\mf z,\mf v, J)$ is the
corresponding $H$-type algebra with $J^2$-condition. Recall the model $D$ of the
rank one Riemannian symmetric space of noncompact type which is constructed from
$(C,e,V,J)$ resp.\@ $(\mf z,\mf v,J)$ in Section~\ref{sec_D}. Let $G$ denote the
full isometry group of $D$.

The purpose of this section is to prove the existence of isometric fundamental
regions for certain subgroups $\Gamma$ of $G$. For this we first have to define
the notion of the isometric sphere of $g\in G\mminus G_\infty$, which is a
sphere \wrt Cygan metric. The Cygan metric is a metric on
$\dg\mminus\{\infty\}$ which arises from a certain group norm on
$\dg\mminus\{\infty\}$. This group norm is an extension of the Heisenberg
pseudonorm.
\subsection{The Cauchy-Schwarz Theorem for $\beta_2$}\label{sec_CSU}

In our setting, the proof of the Cauchy-Schwarz Theorem as it is usually taught in linear algebra cannot be adapted to the map $\beta_2$. Therefore we provide an alternative proof for which the following two lemmas are needed.

\begin{lemma}\label{orth} 
Let $u, v\in V$. Then $\beta_2(v,u)=0$ if and only if $v\in (Cu)^\perp$.
\end{lemma}

\begin{proof} We have $\beta_2(v,u) =0$ if and only if
\[ \langle \beta_2(v,u),\zeta\rangle = 0 \quad\text{for all $\zeta\in C$.}\]
By the definition of $\beta_2$, this holds if and only if
\[ \langle \zeta u, v\rangle = 0 \quad\text{for all $\zeta\in C$,}\]
hence if and only if $v\in (Cu)^\perp$. 
\bewend \end{proof}

\begin{lemma}\label{multiplicative} Let $u\in V$ and $\lambda\in C$. Then $\beta_2(\lambda u, u) = |u|^2\lambda$.
\end{lemma}

\begin{proof} By the definition of $\beta_2$ and the polarization \eqref{polar} we have
\begin{align*}
\langle \beta_2(\lambda u, u),\zeta\rangle & = \langle \zeta u,\lambda u\rangle = \langle \lambda,\zeta\rangle |u|^2 = \langle |u|^2\lambda, \zeta\rangle
\end{align*}
for all $\zeta\in C$.  Hence $\beta_2(\lambda u,u)=|u|^2\lambda$. 
\bewend \end{proof}

\begin{proposition}\label{CSU} Let $u,v\in V$. Then 
\[ |\beta_2(u,v)| \leq |u| |v|.\]
Equality holds if and only if $v\in Cu$ or $u\in Cv$.
\end{proposition}

\begin{proof} Let $(v_1,v_2)\in Cu \times (Cu)^\perp$ be the unique pair such that $v=v_1+v_2$. Using that $\beta_2$ is $C$-hermitian (see Proposition~\ref{beta_2_real}), Lemma~\ref{orth} yields
\[ \beta_2(u,v) = \beta_2(u,v_1) + \beta_2(u,v_2) = \beta_2(u,v_1).\]
Then Lemma~\ref{multiplicative} shows
\[ |\beta_2(u,v_1)|^2 = |u|^2 |v_1|^2.\]
Clearly 
\[ |v_1|^2 \leq |v_1|^2 + |v_2|^2 = |v|^2, \]
where equality holds if and only if $v_2=0$, hence if and only if $v=v_1\in Cu$. 
Thus,
\[ |\beta_2(u,v)|^2 = |\beta_2(u,v_1)|^2 = |u|^2 |v_1|^2 \leq |u|^2 |v|^2, \]
where the inequality is an equality if and only if $u=0$ or $v\in Cu$. This
proves the claim. 
\bewend \end{proof}
\subsection{H-coordinates, Cygan metric, and isometric
spheres}\label{sec_isospheres}

Let $z=(t,Z,X)\in \dg\mminus\{\infty\}$ and recall the map $\Theta$ from
\eqref{Theta}. The \textit{horospherical coordinates} or \textit{H-coordinates}
of $z$ relative to the origin $o_D$ of $D$ are defined as
\[
 \Theta^{-1}(z) = \big( t-\tfrac14|X|^2, Z, X\big).
\]
To avoid confusion with the ordinary coordinates of $\dg$, H-coordinates will
be subscripted by $h$. Hence,  the H-coordinates of $z$ are denoted by $\big(
t-\tfrac14|X|^2, Z, X\big)_h$.

Since $\Theta$ induces a bijection between $\dg\mminus\{\infty\}$ and the
topological closure 
\[
\overline S = \R^+_0\times\mf z\times \mf v
\]
of $S$ in $\R\times\mf z\times\mf v$, there is a geometric characterization of
H-coordinates. We elaborate on this in the following remark.

\begin{remdef}
Let  $z=(t, Z,X)=(\zeta, v)\in \dg\mminus\{\infty\}$. The \textit{horosphere}
through $z$ with center $\infty$ is the $N$-orbit of $z$. We extend the group
$A$ to the set
\[ A^+ \sceq A \cup \{ a_0\}, \]
where $a_0\colon \dg \to \dg$ is defined by $a_0 \infty \sceq \infty$ and $a_0z
\sceq 0$ for all $z\in \dg\mminus\{\infty\}$.
Then there is a unique pair $(a_s, n) \in A^+ \times N$ such that 
\[ na_s(o_D) = z.\]
The \textit{height} of $z$ is defined as
\[ \height(z)\sceq t - \tfrac14 |X|^2 = \Rea\zeta - \tfrac14|v|^2. \]
Formula~\eqref{ANaction} shows that $s=\height(z)$ and $n=(1,Z,X)$. Hence the
H-coordinates of $z$ are 
\[ (\height(z), Z, X)_h = (\height(z), \Ima\zeta, v)_h.\]
This means that the H-coordinates are given by the height of the
$\infty$-centered horosphere on which $z$ lies and the coordinates of $z$ in the
canonical parametrization of this horosphere.
\end{remdef}

The Cygan metric on $D$ is a metric on $\dg\mminus\{\infty\}$ which arises from
a certain group norm. If $(\mc G,\cdot)$ is a group with neutral element $1_{\mc
G}$, then a map
$p\colon \mc G\to \R^+_0$ is called a \textit{group norm} if 
\begin{enumerate}[(GN1)]
\item\label{GN1}  $p(g) = 0$ if and only if $g=1_{\mc G}$,
\item\label{GN2}  $p(g^{-1}) = p(g)$ for all $g\in \mc G$,
\item\label{GN3}  $p(gh) \leq p(g) + p(h)$ for all $g,h\in \mc G$.
\end{enumerate}
Suppose that $p$ is a group norm on $\mc G$, then the map $d\colon \mc G\times
\mc G\to \R$, $d(g,h) \sceq p(g^{-1}h)$ is a metric on $\mc G$. It is called the
\textit{metric induced by $p$}.

A well-known example of a group norm and a pre-form of the group norm for the
Cygan metric is the Heisenberg pseudonorm on $N$. More precisely, $N$ inherits
an inner product from its canonical bijection to $\mf z\times\mf v$. Then the 
\textit{Heisenberg group norm} $q$ (which is also known as the
\textit{Heisenberg pseudonorm}) on $N$  is defined by   
\[ q(Z,X) \sceq \left| \tfrac14 |X|^2 + Z \right|^{1/2} = \left( \tfrac1{16}|X|^4 + |Z|^2 \right)^{1/4}.\]
The last equality holds because $\mf z$ and $\R \cong\mf a$ are orthogonal.
Since each height level set of $\dg\mminus\{\infty\}$ is isomorphic to $N$, the 
metric induced from the Heisenberg group norm measures the distance between two
elements in the same height level set.

To be able to also measure the distance between elements in different height
level sets, we extend the Heisenberg pseudonorm to the direct product of the
groups $(\R, +)$ (``differences between height level sets'') and $N$ by
\[ 
p\colon\left\{
\begin{array}{ccl}
\R\times N & \to & \R
\\[1mm]
(k,Z,X) &\mapsto & \left| \tfrac14 |X|^2 + |k| + Z\right|^{1/2}.
\end{array}
\right.
\]
Obviously, $p\vert_{\{0\}\times N} =q$.

\begin{proposition} 
The map $p$ is a group norm on $\R\times N$. 
\end{proposition}

\begin{proof} The neutral element of $\R\times N$ is $(0,0,0)$, and $\mf z$ is
orthogonal to $\R$. This implies that $p$ satisfies (\apref{GN}{GN1}{}). To
prove (\apref{GN}{GN2}{}) let $g=(k,Z,X)\in \R\times N$. Then  
\begin{align*} 
p( g^{-1}) & = p( -k, -Z, -X) = \left| \tfrac14 |X|^2 + |k| - Z \right|^{1/2}
 = \left( \left( \tfrac14|X|^2 + |k|\right)^2 + |Z|^2\right)^{1/4}
\\ & = \left| \tfrac14 |X|^2 + |k| + Z \right|^{1/2}
 = p(g).
\end{align*} 
The triangle equality (\apref{GN}{GN3}{}) is shown in several steps. For each $g=(k,Z,X)\in \R\times N$ we have
\begin{align*}
p(g) & = \left|\tfrac14|X|^2 + |k|+ Z \right|^{1/2} = \left( \left( \tfrac14 |X|^2 + |k|\right)^2 + |Z|^2 \right)^{1/4}
\\ & \geq \left( \left(\tfrac14 |X|^2\right)^2 \right)^{1/4} = \tfrac12 |X|.
\end{align*}
This and Proposition~\ref{CSU} yield that
\[ \tfrac14 |\beta_2(X_2,X_1)| \leq \tfrac{ |X_1|}{2} \tfrac{|X_2|}{2} \leq p(k_1,Z_1,X_1) p(k_2,Z_2,X_2) \]
for all $(k_j,Z_j,X_j)\in \R\times N$, $j=1,2$. Further, for all $(Z,X)\in N$ and $k_1,k_2\in\R$, we find
\begin{align*}
\left| \tfrac14 |X|^2 + |k_1+k_2| + Z\right|^{1/2} & = \left( \left( \tfrac14 |X|^2 + |k_1+k_2|\right)^2 + |Z|^2\right)^{1/4}
\\ & \leq \left( \left( \tfrac14 |X|^2 + |k_1| + |k_2|\right)^2 + |Z|^2\right)^{1/4}
\\ & =  \left| \tfrac14 |X|^2 + |k_1| + |k_2| + Z \right|^{1/2}.
\end{align*}
Now let $g_j=(k_j,Z_j, X_j)\in \R\times N$, $j=1,2$. Then  
\begin{align*}
p( g_1g_2) & = p\big(k_1+k_2, Z_1+Z_2 + \tfrac12 [X_1,X_2], X_1 + X_2\big)
\\ & = \left| \tfrac14 |X_1+X_2|^2 + |k_1+k_2| + Z_1+Z_2 + \tfrac12 [X_1,X_2]\right|^{1/2}
\\ & \leq \left| \tfrac14 |X_1+X_2|^2 + |k_1|+|k_2| + Z_1 + Z_2 + \tfrac12 [X_1,X_2]\right|^{1/2}
\\ & = \left| \tfrac14 |X_1|^2 + |k_1|+ Z_1 + \tfrac14 |X_2|^2 + |k_2| + Z_2 + \tfrac12\left(\langle X_1,X_2\rangle + [X_1,X_2] \right) \right|^{1/2}.
\end{align*}
Recall from Lemma~\ref{re_and_im} that 
\[
 \langle X_1, X_2\rangle + [X_1,X_2] = \beta_2(X_2,X_1).
\]
Then
\begin{align*}
p(g_1g_2) & \leq \left| \tfrac14 |X_1|^2 + |k_1|+ Z_1 + \tfrac14 |X_2|^2 + |k_2| + Z_2 + \tfrac12\beta_2(X_2,X_1)\right|^{1/2}
\\
& \leq \left[ \left|\tfrac14 |X_1|^2 + |k_1|+ Z_1 \right| + \left|\tfrac14 |X_2|^2 + |k_2| + Z_2 \right| + \tfrac12 \left|\beta_2(X_2,X_1) \right| \right]^{1/2}
\\ & =  \left[ p(g_1)^2 + p(g_2)^2 + \tfrac12|\beta_2(X_2,X_1)| \right]^{1/2}
\\ & \leq \left[ p(g_1)^2 + p(g_2)^2 + 2p(g_1)p(g_2) \right]^{1/2}
\\ & = p(g_1) + p(g_2).
\end{align*}
This completes the proof. 
\bewend \end{proof}

The definition of the Cygan metric uses the H-coordinates on $\dg\mminus\{\infty\}$. We use the  map
\[
\kappa\colon \left\{
\begin{array}{ccl}
\dg\mminus\{\infty\} & \to & \R\times N
\\
z & \mapsto & \Theta^{-1}(z)
\end{array}
\right.
\]
to assign to each element of $\dg\mminus\{\infty\}$ its H-coordinates. The
advantage of using H-coordinates is that the points of $\dg\mminus\{\infty\}$
get
embedded into the group $\R\times N$. This in turn allows to perform the group
operations of $\R\times N$ on  $\kappa(z_1), \kappa(z_2)$ for elements
$z_1,z_2\in \dg\mminus\{\infty\}$.

\begin{definition} 
The \textit{Cygan metric on $D$} is given by
\[
\varrho\colon\left\{
\begin{array}{ccl}
\dg\mminus\{\infty\} \times \dg\mminus\{\infty\} & \to & \R
\\
(g_1,g_2) & \mapsto & p\big(\kappa(g_1)^{-1}\kappa(g_2)\big).
\end{array}
\right.
\]
Since $\kappa$ is injective, the Cygan metric is in fact a metric on $\dg\mminus\{\infty\}$.
\end{definition}

The following lemma provides worked out formulas for the Cygan metric. It can be proved by a straightforward calculation.

\begin{lemma}\label{Cyganformula}
Let $z_j=(\zeta_j, v_j) = (k_j, Z_j, X_j)_h$ be elements of $\dg\mminus\{\infty\}$. Then the Cygan metric is given by
\begin{align*}
\varrho(z_1,z_2) & =  
\left| \tfrac14 |X_1-X_2|^2 + |k_1-k_2| +  Z_1-Z_2 + \tfrac12\Ima\beta_2(X_2,X_1) \right|^{1/2}
\\ & = \left| \tfrac14 |v_1|^2 + \tfrac14 |v_2|^2 + |\height(z_1) - \height(z_2)| + \Ima\zeta_1 - \Ima\zeta_2 - \tfrac12\beta_2(v_1,v_2)\right|^{1/2}.
\end{align*}
\end{lemma}

\begin{convention}
Let $g\in G\mminus G_\infty$. Whenever in the
following we write $g=n_1\sigma m a_t n_2$, it
should be understood that $n_1,n_2\in N$, $a_t\in A$ and $m\in M$.
\end{convention}

\begin{definition}\label{def_isospheres}
Let $g\in G\mminus G_\infty$. Suppose that  $g=n_1\sigma ma_tn_2$, and let
$R(g) \sceq t^{-1/4}$. Then the
set 
\[ I(g) \sceq \big\{ z\in D \ \big\vert\ \varrho\big(z,g^{-1}\infty\big) = R(g) \big\} \]
is called the \textit{isometric sphere of $g$}. Further, 
\[ \Ext I(g) \sceq \big\{ z\in D \ \big\vert\  \varrho\big(z,g^{-1}\infty\big) > R(g) \big\} \]
is called the \textit{exterior} of $I(g)$, and
\[ \Int I(g) \sceq \big\{ z\in D \ \big\vert\  \varrho\big(z,g^{-1}\infty\big) < R(g) \big\} \]
is called the \textit{interior} of $I(g)$. The value $R(g)$ is
the \textit{radius}  of $I(g)$.
\end{definition}

\begin{lemma}\label{metric_formula} Let $g=n_1\sigma ma_tn_2\in G\mminus G_\infty$ and $n_2=(1, Z_{s2}, X_{s2})$. Then 
\[ \varrho\big( (\zeta, v), g^{-1}\infty\big) = \Big| \tfrac14 |X_{s2}|^2 + Z_{s2} + \zeta + \tfrac12\beta_2(v,X_{s2})\Big|^{1/2}.\]
Further $\varrho\big( \cdot, g^{-1}\infty\big)$ is unbounded on $D$, and hence $\Ext I(g) \not=\emptyset$.
\end{lemma}

\begin{proof} We have
\[ g^{-1}\infty  = n_2^{-1}0 = \left(\tfrac14 |X_{s2}|^2 - Z_{s2}, -X_{s2} \right).\]
Thus, the expression for $\varrho( (\zeta,v), g^{-1}\infty)$ follows
immediately from Lemma~\ref{Cyganformula}. We now prove the second part of the
claim. Let $(\zeta, v)\in D$ and $t>1$. Then 
\[ \height\big( (t\zeta, v) \big) = t\Rea \zeta - \tfrac14 |v|^2 > \Rea\zeta -\tfrac14 |v|^2 > 0.\]
Hence $(t\zeta, v)\in D$. Further 
\begin{align*}
\varrho\big( (t\zeta, v), g^{-1}\infty\big)^2 & = \left| t\zeta + \tfrac14|X_{s2}|^2 + Z_{s2} + \tfrac12 \beta_2(v,X_{s2})\right| 
\\  & \geq t|\zeta| - \left|\tfrac14|X_{s2}|^2 + Z_{s2} + \tfrac12\beta_2(v,X_{s2}) \right|,
\end{align*}
which converges to $\infty$ for $t\to\infty$. Hence $\varrho\big(\cdot,
g^{-1}\infty\big)$ is unbounded. This completes the proof. 
\bewend \end{proof}

\begin{lemma} \label{propext}
Let $g\in G\mminus G_\infty$. Then
\begin{enumerate}[{\rm (i)}]
\item \label{propexti} $\Ext I(g)$ and $\Int I(g)$ are open,
\item \label{propextii} $\overline{\Ext I(g)} = \big\{ z\in D \ \big\vert\ \varrho\big(z,g^{-1}\infty\big) \geq R(g) \big\} = \complement\Int I(g)$,
\item \label{propextiii} $\overline{\Int I(g)} = \big\{ z\in D \ \big\vert\ \varrho(z,g^{-1}\infty) \leq R(g) \big\} = \complement\Ext I(g)$.
\item \label{propextiv} If $(\zeta, v) \in \overline{\Int I(g)}$, then
$(\zeta-s,v) \in \Int I(g)$ for each $s\in (0,\height(\zeta))$.
\item \label{propextv} If $(\zeta, v)\in \overline{\Ext I(g)}$, then
$(\zeta+s,v)\in \Ext I(g)$ for each $s>0$.
\end{enumerate}
\end{lemma}

\begin{proof}
Suppose that $g=n_1\sigma ma_tn_2$ with $n_2=(1,Z_2,X_2)$. Then $R(g)=t^{-1/4}$ and, by Lemma~\ref{metric_formula}, 
\[
\varrho\big( (\zeta,v), g^{-1}\infty \big) = \left| \tfrac14|X_2|^2 + Z_2 + \zeta + \tfrac12\beta_2(v,X_2) \right|^{1/2}
\]
for each $(\zeta, v) \in D$. Since $\beta_2(\cdot, X_2) \colon V\to C$ is $\R$-linear (see Proposition~\ref{beta_2_real}), the map
\[
f \colon \left\{
\begin{array}{ccl}
D & \to & \R
\\
(\zeta, v) & \mapsto & \left| \tfrac14|X_2|^2 + Z_2 + \zeta + \tfrac12\beta_2(v,X_2) \right|
\end{array}
\right.
\]
is continuous. Then 
\[
\Ext I(g)  = f^{-1}\big( (t^{-1/2},\infty) \big) \quad\text{and}\quad 
\Int I(g)  = f^{-1}\big( (-\infty, t^{-1/2}) \big).
\]
This shows that $\Ext I(g)$ and $\Int I(g)$ are open. Moreover, it follows that $\overline{\Int I(g)}$ is contained in $f^{-1}\big( (-\infty, t^{-1/2}]\big)$. For the converse inclusion relation, it suffices to show that $f^{-1}(t^{-1/2}) \subseteq \overline{\Int I(g)}$. Let $z_0=(\zeta_0,v_0) \in f^{-1}(t^{-1/2})$. Then 
\begin{align*}
t^{-1/2} & = \left| \tfrac14|X_2|^2 + Z_2 + \zeta_0 + \tfrac12\beta_2(v_0,X_2)\right|
\\
& = \left[ \left|\tfrac14|X_2+v_0|^2 + \height(\zeta_0) \right|^2 + \left| Z_2 + \Ima\zeta_0 + \tfrac12\Ima\beta_2(v_0,X_2) \right|^2 \right]^{1/2}.
\end{align*}
For each $s\in (0,\height(\zeta_0))$ it follows that
\begin{align*}
t^{-1/2} & > \left[ \left|\tfrac14|X_2+v_0|^2 + \height(\zeta_0) - s \right|^2 + \left| Z_2 +\Ima\zeta_0 + \tfrac12\Ima\beta_2(v_0,X_2)\right|^2 \right]^{1/2}
\\
& = \left| \tfrac14|X_2|^2 + Z_2 + \zeta_0-s + \tfrac12\beta_2(v_0,X_2) \right|.
\end{align*}
Thus, $(\zeta_0-s,v_0)\in\Int I(g)$ for each $s\in (0,\height(\zeta_0))$. Hence
\[
\lim_{s\searrow 0} (\zeta_0-s,v_0) = (\zeta_0,v_0) \ \in \overline{\Int I(g)}.
\]
This proves \eqref{propextiii} and \eqref{propextiv}. The proof of
\eqref{propextii} is analogous to that of \eqref{propextiii}, and the proof of
\eqref{propextv} is analogous to that of \eqref{propextiv}. 
\bewend \end{proof}

\begin{proposition}\label{clext}
We have
\[
\overline{\bigcap_{g\in\Gamma\mminus\Gamma_\infty}\Ext I(g)} = \bigcap_{g\in\Gamma\mminus\Gamma_\infty} \overline{\Ext I(g)} = D\mminus\bigcup_{g\in\Gamma\mminus\Gamma_\infty} \Int I(g).
\]
\end{proposition}

\begin{proof}
Lemma~\ref{propext}\eqref{propextiii} states that $\complement \Ext I(g) = \overline{\Int I(g)}$ for each $g\in\Gamma\mminus\Gamma_\infty$. Therefore
\[ 
\complement\Bigg( \overline{\bigcap_{g\in\Gamma\mminus\Gamma_\infty} \Ext I(g)}\Bigg)   = \Bigg( \bigcup_{g\in\Gamma\mminus\Gamma_\infty} \complement \Ext I(g) \Bigg)^\circ 
 = \Bigg(\bigcup_{g\in\Gamma\mminus\Gamma_\infty} \overline{\Int I(g)} \Bigg)^\circ.
\]
By Lemma~\ref{propext}\eqref{propexti}, the interior of each isometric sphere is open. Thus
\[\bigcup_{g\in\Gamma\mminus\Gamma_\infty} \Int I(g) = \Bigg(\bigcup_{g\in\Gamma\mminus\Gamma_\infty}\Int I(g) \Bigg)^\circ,
\]
 and hence
\[
\bigcup_{g\in\Gamma\mminus\Gamma_\infty} \Int I(g) \subseteq \Bigg(\bigcup_{g\in\Gamma\mminus\Gamma_\infty}\overline{\Int I(g)} \Bigg)^\circ.
\]
In the following we will prove the converse inclusion relation. Let $z =
(\zeta,v)$ be an element of $\big(\bigcup\{\, \overline{\Int I(g)} \mid
g\in\Gamma\mminus\Gamma_\infty\} \big)^\circ$ and pick $\eps > 0$ such that 
\[
z_\eps \sceq z + \eps = (\zeta+\eps,v) \in \bigcup_{g\in\Gamma\mminus\Gamma_\infty} \overline{\Int I(g)}.
\]
Fix some $k\in\Gamma\mminus\Gamma_\infty$ such that $z_\eps \in \overline{\Int
I(k)}$. Lemma~\ref{propext}\eqref{propextiv} yields that $z = z_\eps - \eps$ is
in $\Int I(k)$. Thus $z\in\bigcup\{ \Int I(g) \mid
g\in\Gamma\mminus\Gamma_\infty\}$. This shows that 
\[
\Bigg( \bigcup_{g\in\Gamma\mminus\Gamma_\infty} \overline{\Int I(g)} \Bigg)^\circ \subseteq \bigcup_{g\in\Gamma\mminus\Gamma_\infty} \Int I(g).
\]
In turn, 
\[
\complement \bigcup_{g\in\Gamma\mminus\Gamma_\infty} \Int I(g) = \overline{\bigcap_{g\in\Gamma\mminus\Gamma_\infty} \Ext I(g) }.
\]
Finally we have $\complement \Int I(g) = \overline{\Ext I(g)}$ for each $g\in\Gamma\mminus\Gamma_\infty$ by Lemma~\ref{propext}\eqref{propextiii}. Therefore
\[
\overline{\bigcap_{g\in\Gamma\mminus\Gamma_\infty} \Ext I(g) } = \complement
\bigcup_{g\in\Gamma\mminus\Gamma_\infty} \Int I(g) =
\bigcap_{g\in\Gamma\setminus\Gamma_\infty} \complement\Int I(g) =
\bigcap_{g\in\Gamma\mminus\Gamma_\infty} \overline{\Ext I(g)}. 
\] 
\bewend \end{proof}

Let $g\in G\mminus G_\infty$ and $(\zeta, v) \in \dg\mminus\{\infty\}$. Suppose that $g=n_1\sigma ma_tn_2$ with $n_j = (1,Z_{sj}, X_{sj})$ and $m=(\varphi, \psi)$. A lengthy but easy calculation shows that
\begin{align} \label{g_formula} 
g(\zeta , v) =  
\Big(\tfrac14 |X_{s1}|^2 + Z_{s1} + xt^{-1} - \tfrac12 t^{-1/2}\beta_2(x\psi(X_{s2} + v), X_{s1}), \hspace{1cm}
\\ \nonumber X_{s1} - xt^{-1/2}\psi(X_{s2}+v)\Big)
\end{align}
where
\[ x \sceq \left[ \varphi\left( \tfrac14 |X_{s2}|^2 + Z_{s2} + \zeta + \tfrac12 \beta_2(v, X_{s2}) \right) \right]^{-1}.\]
This explicit expression for the action of $g$ on $D$ is used in the following two lemmas. We denote the set of orthogonal endomorphisms of $C$ resp.\@ of $V$ by $O(C)$ resp.\@ $O(V)$.

\begin{lemma}\label{exttoint} 
If $g\in G\mminus G_\infty$, then $g$ maps $\Ext I(g)$ onto $\Int I(g^{-1})$, and $I(g)$ onto $I(g^{-1})$. 
\end{lemma}

\begin{proof}  Suppose that $g=n_1\sigma ma_tn_2$ with $n_j=(1,Z_{sj}, X_{sj})$
and $m=(\varphi, \psi)$. Then we have $g^{-1} =  n_2^{-1} \sigma m a_t
n_1^{-1}$. Hence it follows that
\[
R(g) = R(g^{-1}) = t^{-1/4}
\]
and
\[ g^{-1} I(g^{-1}) = \big\{ z\in D \ \big\vert\  \varrho( gz, g\infty) = R(g) \big\}.\]
In the following we will compare $\varrho(gz, g\infty)$ to $\varrho\big(z,
g^{-1}\infty\big)$. Let $(\zeta, v)$ be in $D$. Set 
\[ x \sceq \left[ \varphi\left( \tfrac14 |X_{s2}|^2 + Z_{s2} + \zeta + \tfrac12 \beta_2(v, X_{s2}) \right) \right]^{-1}.\]
Then \eqref{g_formula} and Lemma~\ref{metric_formula} yield that
\begin{align*}
\varrho\big(&g(\zeta, v), g\infty\big) = \left| \tfrac12|X_{s1}|^2 + xt^{-1} - \tfrac12\beta_2(X_{s1},X_{s1})\right|^{1/2}.
\end{align*}
By Proposition~\ref{beta_2_real}, $\beta_2(X_{s1},X_{s1})  = |X_{s1}|^2$. Thus
\[
\varrho\big(g(\zeta,v),g\infty\big) = \left|xt^{-1}\right|^{1/2}= |x|^{1/2} t^{-1/2}.
\]
Since $\varphi\in O(C)$, it follows that
\begin{align*} 
|x|^{1/2} & = \left| \varphi\left(\tfrac14 |X_{s2}|^2 + Z_{s2} + \zeta + \tfrac12 \beta_2(v, X_{s2}) \right)\right|^{-1/2} 
\\ & = \left| \tfrac14 |X_{s2}|^2 + Z_{s2} + \zeta + \tfrac12 \beta_2(v, X_{s2})\right|^{-1/2}
\\ & = \varrho\big((\zeta, v), g^{-1}\infty\big)^{-1}.
\end{align*}
Therefore 
\[ \varrho\big( (\zeta, v), g^{-1}\infty\big) \varrho\big( g(\zeta,v), g\infty\big) = t^{-1/2}.\]
Hence, $(\zeta,v)\in I(g)$ if and only if 
\[
\varrho\big( g(\zeta,v), g\infty \big) = t^{-1/4}.
\]
This is equivalent to $g(\zeta, v)\in I(g^{-1})$. In turn, $gI(g)  = I(g^{-1})$.
Further, we have $(\zeta, v)\in \Ext I(g)$ if and only if 
\[
 t^{-1/4} > \varrho\big( g(\zeta,v), g\infty\big).
\]
Therefore $g\Ext I(g) = \Int I(g^{-1})$. 
\bewend \end{proof}

\begin{lemma}\label{howch}
Let $n=(1,Z,X)\in N$, $a_s\in A$, and $m=(\varphi,\psi)\in M$. 
\begin{enumerate}[{\rm (i)}]
\item\label{howchi} If $n'\sceq \big(1,sZ, s^{1/2}X\big)$, then $a_sn=n'a_s$.
\item\label{howchii} For all $u,v\in V$ we have $\varphi\big( \beta_2(v,u) \big) = \beta_2(\psi(v),\psi(u))$.
\item\label{howchiii} If $n'\sceq \big(1,\varphi(Z),\psi(X)\big)$, then $mn=n'm$.
\end{enumerate}
\end{lemma}

\begin{proof}
Claim \eqref{howchi} is shown by direct calculation. To prove \eqref{howchii} recall that $\varphi\in O(C)$, $\psi\in O(V)$, and that $\psi\big(J(\eta, w)\big)= J\big(\varphi(\eta),\psi(w)\big)$ for each $(\eta,w)\in C\oplus V$. For each $\zeta\in C$ we have
\begin{align*}
\left\langle\varphi\big(\beta_2(v,u)\big),\zeta\right\rangle & = \left\langle \beta_2(v,u),\varphi^{-1}(\zeta)\right\rangle 
= \left\langle J\big(\varphi^{-1}(\zeta), u\big), v\right\rangle 
\\
& = \left\langle \psi\big( J\big(\varphi^{-1}(\zeta),u\big)\big), \psi(v) \right\rangle 
= \left\langle J(\zeta, \psi(u)), \psi(v)\right\rangle 
\\
& = \left\langle\beta_2(\psi(v),\psi(u)), \zeta\right\rangle.
\end{align*}
Thus $\varphi\big(\beta_2(v,u)\big) = \beta_2(\psi(v),\psi(u))$. For the proof of \eqref{howchiii} let $(\zeta, v)\in D$. Then 
\begin{align*}
m\big( n(\zeta,v) \big) & =  \left( \varphi\left(\tfrac14|X|^2 + Z + \zeta + \tfrac12\beta_2(v,X)  \right), \psi(X+v) \right)
\\
& = \left( \tfrac14|\psi(X)|^2  + \varphi(Z) + \varphi(\zeta) + \tfrac12 \varphi\big( \beta_2(v,X) \big), \psi(X)+\psi(v)\right).
\end{align*}
Now \eqref{howchii} yields 
\begin{align*}
m\big(n(\zeta,v)\big) & = \left( \tfrac14|\psi(X)|^2 + \varphi(Z) + \varphi(\zeta) + \tfrac12\beta_2\big(\psi(v),\psi(X)\big), \psi(X) + \psi(v) \right)
\\
& = n'\big( \varphi(\zeta), \psi(v) \big) = n'\big( m(\zeta,v) \big).
\end{align*}
Hence $mn=n'm$. 
\bewend \end{proof}

\begin{proposition}\label{stabmap} Let $g\in G_\infty$ and $h\in G\mminus
G_\infty$. Then 
\begin{align*} gI(h) & = I\big(ghg^{-1}\big),
\\ g\Int I(h) & = \Int I\big(ghg^{-1}\big), 
\\ g\Ext I(h) & = \Ext\big(ghg^{-1}\big).
\end{align*}
\end{proposition}

\begin{proof} It suffices to prove the claim for the three cases $g\in A$, $g \in N$, and $g \in M$. This we will do separately. Let $h=n_1\sigma m a_t n_2$ with $n_j=(1, Z_{j}, X_{j})$. 

Suppose first that $g=a_s$. Let $(\zeta, v)\in D$. For $j\in\{1,2\}$ set $n'_j
\sceq \big(1, sZ_j, s^{1/2}X_j\big)$. Lemma~\ref{howch} implies that 
\begin{align*}
ghg^{-1} & = a_s n_1\sigma ma_t n_2a_s^{-1} = n'_1\sigma m a_s^{-1} a_t a_s^{-1}n'_2 = n'_1\sigma m a_{s^{-2}t}n'_2.
\end{align*}
Therefore $R\big(ghg^{-1}\big) = s^{1/2}t^{-1/4} = s^{1/2} R(h)$. By Lemma~\ref{metric_formula} we have 
\begin{align*} 
\varrho\big( (\zeta, v), (ghg^{-1})^{-1}\infty\big) &= \left|\tfrac14 s|X_{2}|^2
+ sZ_{2} + \zeta + \tfrac12s^{1/2}\beta_2(v, X_{2})\right|^{1/2} 
\end{align*}
and
\begin{align*}
\varrho\big( a_s^{-1}(\zeta, v), h^{-1}\infty\big)  & = \left|\tfrac14 |X_{2}|^2 + Z_{2} + s^{-1}\zeta + \tfrac12 s^{-1/2}\beta_2(v,X_{2}) \right|^{1/2}
\\ & = s^{-1/2} \varrho\big( (\zeta, v), (ghg^{-1})^{-1}\infty\big).
\end{align*}
Then 
\begin{align*}
gI(h) &  = \left\{ z\in D \left\vert\  \varrho\big(g^{-1}z, h^{-1}\infty\big) = R(h) \right.\right\}
\\
& = \left\{ z\in D \left\vert\ \varrho\big(z, (ghg^{-1})^{-1}\infty\big) = s^{1/2}R(h) \right.\right\}
\\
& = \left\{ z\in D \left\vert\ \varrho\big( z, (ghg^{-1})^{-1}\infty\big) = R(ghg^{-1}) \right.\right\}
\\
& = I\big(ghg^{-1}\big).
\end{align*}
Analogously, we see that $g\Ext I(h) = \Ext I\big(ghg^{-1}\big)$ and $g\Int I(h) = \Int I\big(ghg^{-1}\big)$.

Now suppose that $g=n_3=(1, Z_{3}, X_{3})$. Then $ghg^{-1} = (n_3n_1)\sigma ma_t (n_2n_3^{-1})$ and (see Lemma~\ref{re_and_im})
\begin{align*} 
n_2n_3^{-1}  = \left(1, Z_{2} - Z_{3} - \tfrac12 \Ima\beta_2(X_{3},X_{2}), X_{2} - X_{3}\right).
\end{align*}
Therefore, by Lemmas~\ref{metric_formula} and \ref{re_and_im},
\begin{align*}
\varrho\big(& (\zeta, v) , (n_3hn_3^{-1})^{-1}\infty\big)  = 
\\ 
& =
 \Big|\tfrac14|X_{2}|^2 + \tfrac14|X_{3}|^2 -\tfrac12\beta_2(X_{3},X_{2}) + Z_{2} - Z_{3}  + \zeta + \tfrac12\beta_2(v,X_{2}-X_{3}) \Big|^{1/2}.
\end{align*}
On the other hand, we have
\begin{align*}
\varrho\big( &n_3^{-1}(\zeta, v)  , h^{-1}\infty\big)  = 
\\ & =\Big| \tfrac14|X_{2}|^2 +Z_{2} + \tfrac14|X_{3}|^2 - Z_{3} + \zeta - \tfrac12\beta_2(v,X_{3}) + \tfrac12 \beta_2(-X_{3} + v, X_{2})  \Big|^{1/2}
\\
& = \Big| \tfrac14|X_2|^2 + \tfrac14|X_3|^2 + Z_2 - Z_3 + \zeta - \tfrac12\beta_2(X_3,X_2) + \tfrac12\beta_2(v,X_2-X_3)  \Big|^{1/2}
\\ & = \varrho\big( (\zeta, v) , n_3h^{-1}n_3^{-1}\infty\big).
\end{align*}
Since $R(h)= R\big(ghg^{-1}\big)$, the claim follows for this case.

Finally suppose that $g=m_2=(\varphi, \psi)$ and set $n'_j\sceq
\big(1,\varphi(Z_j),\psi(X_j)\big)$ for $j=1,2$. Lemma~\ref{howch} shows that 
\begin{align*} 
ghg^{-1} &  = m_2n_1\sigma m  a_t  n_2 m_2^{-1}  = n'_1\sigma m_2^{-1}mm_2^{-1}a_t n'_2.
\end{align*}
From this it follows that $R(h) = R\big(m_2hm_2^{-1}\big)$. Moreover, we have
\[
\varrho\big( (\zeta, v), (ghg^{-1})^{-1}\infty \big)  = \Big| \tfrac14 |\psi(X_{2})|^2 + \varphi(Z_{2}) + \zeta + \tfrac12 \beta_2(v, \psi(X_{2})) \Big|^{1/2}. 
\]
From Lemma~\ref{howch} it follows  that 
\[
\varrho\big( (\zeta, v), (ghg^{-1})^{-1}\infty \big)  = \Big| \tfrac14 |X_{2}|^2 + \varphi(Z_{2}) + \zeta + \tfrac12 \varphi\big(\beta_2(\psi^{-1}(v),X_{2})\big)\Big|^{1/2}.
\]
On the other side we find
\begin{align*}
\varrho\big( g^{-1}(\zeta, v), h^{-1}\infty \big) & = \varrho\big((\varphi^{-1}(\zeta), \psi^{-1}(v)), h^{-1}\infty\big)
\\ & = \Big| \tfrac14 |X_{2}|^2 + Z_{2} + \varphi^{-1}(\zeta) + \tfrac12 \beta_2(\psi^{-1}(v), X_{2}) \Big|^{1/2}
\\ & = \Big| \tfrac14 |X_{2}|^2 + \varphi(Z_{2}) + \zeta + \tfrac12 \varphi\big(\beta_2(\psi^{-1}(v), X_{2})\big)\Big|^{1/2}
\\ & = \varrho\big( (\zeta, v), (ghg^{-1})^{-1}\infty\big).
\end{align*}
From this it follows that $gI(h) = I\big(ghg^{-1}\big)$, $g\Ext I(h) = \Ext
I\big(ghg^{-1}\big)$, and $g\Int I(h) = \Int I\big(ghg^{-1}\big)$.  
\bewend \end{proof}

An immediate consequence of Proposition~\ref{stabmap} is the following corollary.

\begin{corollary}\label{invariance}
Let $\Gamma$ be a subgroup of $G$ and $g\in \Gamma_\infty$. Then
\begin{align*} g\bigcap_{h\in\Gamma\mminus\Gamma_\infty} \Ext I(h) & =
\bigcap_{h\in\Gamma\mminus\Gamma_\infty} \Ext I(h) 
\intertext{and}
g\bigcup_{h\in\Gamma\mminus\Gamma_\infty} \Int I(h) & =
\bigcup_{h\in\Gamma\mminus\Gamma_\infty} \Int I(h).
\end{align*}
\end{corollary}

\begin{lemma}\label{ht_mut} Let  $z\in D$ and $g\in G\mminus G_\infty$. Then 
\[ \height(z) = \left( \frac{\varrho(z,g^{-1}\infty)}{R(g)}\right)^4\height(gz).\]
\end{lemma}

\begin{proof} Let $z=(\zeta, v)$ and $g=n_1\sigma ma_tn_2$ with $n_j=(1, Z_{j},
X_{j})$ and $m=(\varphi, \psi)$. We first evaluate $\height(gz)$. Set 
\[
 x \sceq \left[\varphi\left( \tfrac14|X_2|^2 + Z_2 + \zeta + \tfrac12\beta_2(v,X_2) \right) \right]^{-1}.
\]
Using \eqref{g_formula} we find
\begin{align*}
\height(gz) & = \Rea\left( \tfrac14|X_1|^2 + Z_1 + xt^{-1} - \tfrac12 t^{-1}\beta_2\big( x\psi(X_2+v),X_1\big) \right) 
\\
&  \hspace{6cm} - \tfrac14\left| X_1 - xt^{-1/2}\psi(X_2+v) \right|^2
\\
& = \tfrac14|X_1|^2 + t^{-1}\Rea(x) - \tfrac12 t^{-1/2} \left\langle x\psi(X_2+v),X_1 \right\rangle 
\\
&  \hspace{6cm} - \tfrac14\left| X_1 -xt^{-1/2}\psi(X_2+v) \right|^2
\\
& = t^{-1}\Rea(x) - \tfrac14 t^{-1} |x|^2 \cdot |X_2+v|^2
\\
& = t^{-1}|x|^2\left[ \Rea\big(|x|^{-2}x\big) - \tfrac14|X_2+v|^2 \right].
\end{align*}
Since $\Rea\varphi(\eta) = \Rea\eta$ for each $\eta\in C$, it follows that 
\begin{align*}
\Rea\big( |x|^{-2}x \big) & = \Rea\big( \overline x^{-1} \big) = \Rea \big(x^{-1}\big)
\\
& = \Rea \left( \tfrac14|X_2|^2 + Z_2 +\zeta + \tfrac12\beta_2(v,X_2) \right)
\\
& = \tfrac14|X_2|^2 + \Rea \zeta + \tfrac12\langle v,X_2\rangle
\\
& = \tfrac14|X_2+v|^2 + \Rea\zeta - \tfrac14|v|^2.
\end{align*}
Thus,
\[
\height(gz) = t^{-1}|x|^2 \left[ \Rea \zeta - \tfrac14|v|^2 \right] = t^{-1}|x|^2\height(z).
\]
Lemma~\ref{metric_formula} shows that
\begin{align*}
\varrho\big( (\zeta,v), g^{-1}\infty \big) & = \left| \tfrac14|X_2|^2 + Z_2 + \zeta + \tfrac12\beta_2(v,X_2) \right|^{1/2} = |x|^{-1/2}.
\end{align*}
Since $R(g) = t^{-1/4}$, we have
\[
\height(z) = t|x|^{-2} \height(gz) = \left(\frac{\varrho(z,g^{-1}\infty)}{R(g)} \right)^4 \height(gz).
\] 
\bewend \end{proof}
\subsection{Fundamental regions}\label{sec_isomfundregion}

A subgroup $\Gamma$ of $G$ is said to be \textit{of type~(O)} if 
\[ \bigcap_{g\in\Gamma\mminus\Gamma_\infty} \Ext I(g) = D \big\backslash \overline{\bigcup_{g\in\Gamma\mminus\Gamma_\infty} \Int I(g)}. \]
For a subset $S$  of $G$ and let $\langle S\rangle$ denote\label{def_bS} the subgroup of $G$ generated by $S$. Then $S$ is said to be of \textit{of type~(F)}, if for each $z\in D$ the maximum of the set 
\[ \big\{ \height(gz) \ \big\vert\  g\in \langle S\rangle\big\}\]
exists.

\begin{theorem}\label{fund_region2} Let $\Gamma$ be a subgroup of $G$ of type
(O) such that $\Gamma\mminus\Gamma_\infty$ is of type (F). Suppose that
$\fd_\infty$ is a fundamental region for $\Gamma_\infty$ in $D$ satisfying
\[ \overline \fd_\infty \cap \bigcap_{g\in\Gamma\mminus\Gamma_\infty} \overline{\Ext I(g)} = \overline{\fd_\infty\cap \bigcap_{g\in\Gamma\mminus\Gamma_\infty} \Ext I(g)}.\] Then 
\[ \fd \sceq \fd_\infty \cap \bigcap_{g\in\Gamma\mminus\Gamma_\infty} \Ext I(g) \]
is a fundamental region for $\Gamma$ in $D$. 
\end{theorem}

\begin{proof}
Since $\Gamma$ is of type (O), the set $\bigcap\{\Ext I(g)\mid
g\in\Gamma\mminus\Gamma_\infty\}$ is open. Then $\fd_\infty$ being open as a
fundamental region for $\Gamma_\infty$ in $D$ yields that $\fd$ is open. 

Now let $z\in \fd$ and $g\in\Gamma\mminus\{\id\}$. If $g\in\Gamma_\infty$, then
$gz\notin\fd_\infty$ since the $\Gamma_\infty$-translates of $\fd_\infty$ are
pairwise disjoint. If $g\in\Gamma\mminus\Gamma_\infty$, then $z\in\Ext I(g)$.
Lemma~\ref{exttoint} states that $gz\in \Int I\big(g^{-1}\big)$, and hence
$gz\notin \Ext I\big(g^{-1}\big)$.  Thus, in each case, $gz\notin \fd$. This
shows that the $\Gamma$-translates of $\fd$ are pairwise disjoint.

It remains to prove that $D \subseteq \Gamma\cdot \overline \fd$. To that end
let $z\in D$ and set 
\[
A\sceq \bigcap\big\{\, \overline{\Ext I(g)}\ \big\vert\ 
g\in\Gamma\mminus\Gamma_\infty\big\}.
\]
Since $\Gamma\mminus\Gamma_\infty$ is of type
(F), the set $\langle\Gamma\mminus\Gamma_\infty\rangle z$ contains an element of
maximal height, say $w$. We claim that $w\in A$. Seeking a contradiction assume
that $w\notin A$. Proposition~\ref{clext} yields the existence of
$h\in\Gamma\mminus\Gamma_\infty$ with $w\in \Int I(h)$. Then, by definition,
\[
 \varrho(w,h^{-1}\infty) < R(h).
\]
This and Lemma~\ref{ht_mut} show that $\height(hw) > \height(w)$, which is a contradiction to the choice of $w$. Thus, $w\in A$. 

Since $\Gamma_\infty\overline{\fd_\infty} = D$, there is $k\in\Gamma_\infty$ such that $kw\in \overline \fd_\infty$. Corollary~\ref{invariance} implies that $kw\in A$. Finally,
\[
kw\in \overline\fd_\infty \cap A = \overline{ \fd_\infty \cap \bigcap_{g\in\Gamma\mminus\Gamma_\infty}\Ext I(g) } = \overline\fd.
\]
This completes the proof. 
\bewend \end{proof}

\begin{remark}\label{specialtypeO}
Let $\Gamma$ be a subgroup of $G$ and suppose that the set 
\[
\{ \Int I(g) \mid g\in\Gamma\mminus\Gamma_\infty\}
\]
of interiors of all isometric spheres is locally finite. Then 
\[
\overline{ \bigcup_{g\in\Gamma\mminus\Gamma_\infty} \Int I(g) } = \bigcup_{g\in\Gamma\mminus\Gamma_\infty} \overline{\Int I(g)}.
\]
by \cite[Hilfssatz~7.14]{Querenburg}. Lemma~\ref{propext}\eqref{propextiii} implies that 
\[
\complement \left(\overline{ \bigcup_{g\in\Gamma\mminus\Gamma_\infty} \Int I(g) } \right) = \complement \left(\bigcup_{g\in\Gamma\mminus\Gamma_\infty} \overline{\Int I(g)}\right) = \bigcap_{g\in\Gamma\mminus\Gamma_\infty} \complement \overline{\Int I(g)} = \bigcap_{g\in\Gamma\mminus\Gamma_\infty} \Ext I(g).
\]
Hence, if the set of interiors of isometric spheres is locally finite, then $\Gamma$ is of type (O).
\end{remark}

A special case of Theorem~\ref{fund_region2}  is the following corollary.

\begin{corollary}\label{region_basic} Let $\Gamma$ be a subgroup of $G$ of type
(O) and $\Gamma_\infty=\{\id\}$. Further suppose that
$\Gamma\mminus\Gamma_\infty$ is of type (F). Then 
\[ \fd \sceq \bigcap_{g\in\Gamma\mminus\Gamma_\infty} \Ext I(g) \]
is a fundamental region for $\Gamma$ in $D$.
\end{corollary}

The following proposition serves to show that in many situations the
fundamental region in Theorem~\ref{fund_region2} is actually a fundamental
domain (see Corollary~\ref{domain_basic} below).

\begin{proposition}\label{extgeodconv}
Suppose that $\mf v= \{0\}$. Then the set $\Ext I(g)$ is geodesically convex for each $g\in G\mminus G_\infty$.
\end{proposition}

\begin{proof}
Let $z_1,z_2\in\Ext I(g)$, $z_1\not= z_2$. Then there exist $s_2>s_1>0$ and $h\in G$ such that $z_1 = ha_{s_1}\cdot o_D$ and $z_2 = ha_{s_2}\cdot o_D$. W.l.o.g.\@ we may assume that $h$ is either of the form $h= n\in N$ (if $h\in G_\infty$) or $h=n'\sigma n\in N\sigma N$ (if $h\in G\mminus G_\infty$). The geodesic segment connecting $z_1$ and $z_2$ is given by 
\[
 c\sceq \{ ha_r\cdot o_D \mid s_1\leq r\leq s_2 \}.
\]
In the following we will show that $c\subseteq \Ext I(g)$ by examining separately the two cases for $h$.

Suppose that $g=n_1\sigma m a_t n_2$ with $n_1, n_2=(1,Z_s)\in N$, $a_t\in A$ and $m\in M$ and $r\in [s_1,s_2]$. At first let $h=n=(1,Z)\in N$. Then 
\begin{align*}
\varrho( na_r\cdot o_D, g^{-1}\infty) & = \varrho( (r,Z), g^{-1}\infty)
 = | Z_s + Z + r|^{1/2}
 = \left( |Z_s+Z|^2 + r^2\right)^{1/4}
\\
& \geq \left( |Z_s+Z|^2 + s_1^2\right)^{1/4} = \varrho( na_{s_1}\cdot o_D, g^{-1}\infty) > R(g).
\end{align*}
Hence $ha_r\cdot o_D\in \Ext I(g)$.

Now let $h=n'\sigma n\in N\sigma N$ with $n'=(1,Z')$ and $n=(1,Z)$. A straightforward calculation shows that 
\[
 \varrho( ha_r\cdot o_D, g^{-1}\infty) > R(g) = t^{-1/4}
\]
is equivalent to 
\begin{equation}\label{isgeodconv}
 0< ar^4 + br^2 + c
\end{equation}
where $Z'_s\sceq Z_s +Z'$ and
\begin{align*}
a & \sceq |Z'_s|^2 - t^{-1}
\\
b & \sceq 1+ 2|Z|^2 |Z'_s|^2 - 2\langle Z, Z'_s\rangle - 2|Z|^2 t^{-1}
\\
c & \sceq |Z|^4 |Z'_s|^2 - 2|Z|^2 \langle Z, Z'_s\rangle + |Z|^2 - |Z|^4 t^{-1}.
\end{align*}
Solving \eqref{isgeodconv} separately for $a=0$, $a>0$ and $a<0$ yields that it
is satisfied for $r$. Therefore $ha_r\cdot o_D\in \Ext I(g)$. This completes the
proof. 
\bewend \end{proof}

\begin{remark}
Using the classification in \cite{Koranyi_Ricci_proofs} of $J^2C$-module
structures and the statement \cite[Proposition~2.5.1]{Chen_Greenberg}  on the
(non-)existence of totally geodesic submanifolds of codimension one in rank one
Riemannian
symmetric spaces we see that
Proposition~\ref{extgeodconv} in general becomes false if $\mf v\not=\{0\}$.
\end{remark}

Theorem~\ref{fund_region2} and Proposition~\ref{extgeodconv} immediately imply the
following statement.

\begin{corollary}\label{domain_basic}
Let $\mf v = \{0\}$ and let $\Gamma$ be a subgroup of $G$ of type~(O) such that
$\Gamma\mminus\Gamma_\infty$ is of type~(F). Suppose that $\fd_\infty$ is a
convex fundamental domain for $\Gamma_\infty$ in $D$ satisfying
\[
 \overline\fd_\infty \cap
\bigcap_{g\in\Gamma\mminus\Gamma_\infty}\overline{\Ext I(g)} =
\overline{\fd_\infty \cap \bigcap_{g\in\Gamma\mminus\Gamma_\infty} \Ext I(g)}.
\]
Then
\[
 \fd \sceq \fd_\infty \cap \bigcap_{g\in\Gamma\mminus\Gamma_\infty} \Ext I(g)
\]
is a convex fundamental domain for $\Gamma$ in $D$.
\end{corollary}

\section{Projective models}\label{sec_proj}

The purpose of this section is to show that the existing definitions of
isometric spheres and results concerning the existence of isometric fundamental
regions in literature are essentially covered by the definitions and results in
Section~\ref{sec_fundreg}. The reason for the reservation towards a confirmation to
cover all existing definitions and results is twofold: On the one hand the
author cannot guarantee to be aware of all existing results. On the other hand,
at least for the real hyperbolic plane, the literature contains non-equivalent
definitions of isometric spheres. Moreover, the existence results of isometric
fundamental regions by Ford are proved for a weaker notion of fundamental region
than the one used here. Section~\ref{sec_literature} contains a detailed discussion
of the latter issues.

Let $(C,V,J)$ be a $J^2C$-module structure. In Section~\ref{sec_division} we
introduce the structure of division algebras on $C$ following
\cite{Koranyi_Ricci} and \cite{Koranyi_Ricci_proofs}. For $C$ being an
associative division algebra, we redo, in
Sections~\ref{sec_projspace} and \ref{sec_hypspaces}, the
classical projective construction of hyperbolic spaces in terms of the
$J^2C$-module structure. A long part of Section~\ref{sec_hypspaces} is devoted to a
detailed study of the relation between the isometry group $G$ of the
symmetric space and the natural ``matrix'' group on the projective space. This
investigation will show that the matrix group is isomorphic to a certain
subgroup $G^\Res$ of $G$. In Section~\ref{sec_cocycle} we use these results to
provide a characterization of the isometric sphere of $g\in G^\Res$ via a
cocycle. In Section~\ref{sec_special} we prove that a special class of subgroups
$\Gamma$ of $G^\Res$ are of type~(O) with $\Gamma\mminus\Gamma_\infty$ being of
type~(F) and use Theorem~\ref{fund_region2} to show the existence of an
isometric fundamental domain for $\Gamma$. Finally, in
Section~\ref{sec_literature}, we bring together these investigations for a
comparison with the existing literature.

\subsection{Division algebras induced by $J^2C$-module
structures}\label{sec_division}

Let $(C,V,J)$ be a $J^2C$-module structure. Suppose that $V\not=\{0\}$ and fix
an element $v\in V\mminus\{0\}$. From~\eqref{M3} and (\apref{M}{M2}{}) it
follows that for each pair $(\zeta,\eta)\in C\times C$ there exists a unique
element $\tau\in C$ such that 
\[
J(\zeta, J(\eta, v)) = J(\tau,v).
\]
Hence, each choice $v\in V\mminus\{0\}$ equips $C$ with a multiplication
$\cdot_v\colon C\times C\to C$ via\label{def_multv}
\[
 \zeta\cdot_v\eta \sceq \tau \quad\Leftrightarrow\quad J(\zeta,J(\eta,v)) = J(\tau,v).
\]
Equation~\eqref{inverse} shows that the inverse of $\zeta\in C\mminus\{0\}$ is $\zeta^{-1} = |\zeta|^{-2}\overline\zeta$, independent of the choice of $v\in V\mminus\{0\}$.
The following properties of $(C,\cdot_v)$ are shown in
\cite{Koranyi_Ricci_proofs}.

\begin{proposition}[Proposition 1.1 and Corollary 1.5 in \cite{Koranyi_Ricci_proofs}]\mbox{
}\label{division} 
\begin{enumerate}[{\rm (i)}]
\item For each $v\in V\mminus\{0\}$, the Euclidean vector space $C$ with the multiplication $\cdot_v$ is a normed, not necessarily associative, division algebra.
\item The multiplication $\cdot_v$ on $C$ is independent of the choice of $v\in V\mminus\{0\}$ if and only if $(C,\cdot_v)$ is associative for one (and hence for all) $v$.
\end{enumerate}
\end{proposition}

We call the $J^2C$-module structure $(C,V,J)$ \textit{associative} if
$(C,\cdot_v)$ is associative
for some (and hence each) $v\in V\mminus\{0\}$ and otherwise
\textit{non-associative}. If $(C,\cdot_v)$ is associative, we omit the
subscript $v$ of the multiplication $\cdot_v$.

\begin{remark} 
Only associative $J^2C$-module structures are modules in the sense of \cite{Bourbaki_Algebra1}.
\end{remark}

\begin{remark}
Suppose that $(C,V,J)$ is an associative $J^2C$-module structure. The
classification in \cite{Koranyi_Ricci_proofs} (or \cite[Section~4]{Koranyi_Ricci})
shows that $C$ is real or complex
or quaternionic numbers.
\end{remark}

\subsection{$C$-sesquilinear hermitian forms}\label{sec_sesqui}

From now on let $(C,V,J)$ be an associative $J^2C$-module structure. Suppose
that  $M$ is a (left) $C$-module. A map $\Phi\colon M\times M\to C$  is
said to be a \textit{$C$-sesquilinear hermitian form} if $\Phi$ is $C$-hermitian
and $C$-linear in the first variable, that is, if the following properties are
satisfied:
\begin{enumerate}[(SH1)]
\item\label{SH1} $\Phi(\zeta_1 x_1 + \zeta_2 x_2, y) = \zeta_1\Phi(x_1,y) + \zeta_2 \Phi(x_2,y)$ for all $\zeta_1,\zeta_2\in C$ and $x,y\in M$,
\item\label{SH2} $\Phi(x,y) = \overline{\Phi(y,x)}$ for all $x,y\in M$.
\end{enumerate}
A $C$-sesquilinear hermitian form $\Phi$ is called \textit{non-degenerate} 
if $\Phi(m_1,\cdot) = 0$ implies that $m_1=0$. It is called \textit{indefinite}
if there exist $m_1,m_2\in M$ such that $\Phi(m_1,m_1) < 0$ and
$\Phi(m_2,m_2)>0$.

\begin{proposition}\label{beta1} The map $\beta_1\colon C\times C\to C$,
$\beta_1(x,y) \sceq x\overline y$, is $C$-sesquilinear hermitian. Further
$\Rea\beta_1(x,y) = \langle x,y\rangle$ for all $x,y\in C$. 
\end{proposition}

\begin{proof} Obviously, $\beta_1$ is $\R$-bilinear. For each $y\in C$, the
$C$-linearity of $\beta_1(\cdot,y)$ is exactly the left-sided distribution law
of the division algebra $C$. To show (\apref{SH}{SH2}{}) let $x,y\in C$. From
\eqref{adjoint} it follows that
\[ J_{\overline{y\overline{x}}}=J^*_{y\overline x} = J^*_{\overline x}  J^*_y =
J_{x} J_{\overline y} = J_{x\overline y}.\]  
Therefore
\[ \beta_1(x,y) =   x\overline{y}= \overline{ y \overline x}=  \overline{\beta_1(y,x)} .\]
Hence $\beta_1$ is $C$-sesquilinear hermitian.  Finally,  polarization of 
$\beta_1(x,x) = |x|^2$ (see \eqref{square}) over $\R$ implies the
remaining statement. 
\bewend \end{proof}

\begin{remark}\label{real}
An immediate consequence of Proposition~\ref{beta1} is that conjugation and multiplication in $C$ anticommute, \ie $\overline{x}\,\overline{y} = \overline{yx}$ for all $x,y\in C$. In particular, for a $C$-sesquilinear hermitian form $\Phi$ on the $C$-module $M$ we have
\[ \Phi(m_1,\zeta m_2) = \Phi(m_1,m_2)\overline{\zeta}\]
for all $m_1,m_2\in M$ and all $\zeta\in C$. Moreover, we have $\Phi(m,m) =
\overline{\Phi(m,m)}$ for each $m\in M$ and therefore $\Phi(m,m)\in \R$.
\end{remark}

\begin{lemma}\label{innerconj} 
For all $\zeta,\eta,\xi\in C$ we have $\langle \zeta,\xi \eta\rangle = \langle \overline{\xi}\zeta,\eta\rangle$.
\end{lemma}

\begin{proof} Let $v\in V\mminus\{0\}$. Using \eqref{polar} and \eqref{adjoint} we find
\begin{align*}
2 \langle \zeta,\xi\eta \rangle |v|^2 & = 2\langle \zeta v, \xi\eta v\rangle = 2\langle \overline{\xi}\zeta v, \eta v\rangle = 2 \langle \overline{\xi}\zeta, \eta\rangle |v|^2,
\end{align*}
hence $\langle \zeta, \xi\eta\rangle = \langle \overline{\xi}\zeta,
\eta\rangle$. 
\bewend \end{proof}

\begin{proposition}\label{beta2}
The map $\beta_2\colon V\times V\to C$ is $C$-sesquilinear hermitian. 
\end{proposition}

\begin{proof} Because of Proposition~\ref{beta_2_real} it remains to show that $\beta_2$ is $C$-linear in the first variable. Let $v,u\in V$ and $\zeta\in C$. Lemma~\ref{innerconj} and \eqref{adjoint} yield that for all $\eta\in C$ we have
\begin{align*}
\langle \zeta\beta_2(v,u),\eta\rangle & = \langle \beta_2(v,u),\overline{\zeta}\eta \rangle = \langle \overline{\zeta}\eta u,v\rangle
 = \langle \eta u,\zeta v\rangle = \langle \beta_2(\zeta v, u),\eta\rangle.
\end{align*}
Hence $\zeta\beta_2(v,u) = \beta_2(\zeta v, u)$. 
\bewend \end{proof}

A finite sequence $(v_1,\ldots, v_n)$ in $V$ is called an \textit{orthonormal
$C$-basis} of $V$ if
\begin{enumerate}[(CON1)]
\item\label{CON1} $|v_j| = 1$ for each $j\in\{1,\ldots, n\}$,
\item\label{CON2} for each pair $(i,j)\in \{1,\ldots, n\}^2$, $i\not=j$, the sets $Cv_i$ and $Cv_j$ are orthogonal,
\item\label{CON3} if we use for each $j\in\{1,\ldots, n\}$ the bijection $C\to
Cv_j$, $\zeta\mapsto \zeta v_j$, to equip $Cv_j$ with the structure of a
Euclidean vector space and a $C$-module, then $V$ is isomorphic as $C$-module
and Euclidean vector space to the direct sum $\bigoplus_{j=1}^n Cv_j$ of the
Euclidean spaces and $C$-modules $Cv_j$, $j=1,\ldots, n$.
\end{enumerate}

The following lemma yields that $V$ is a free $C$-module.

\begin{lemma}\label{Vfree}
There is an orthonormal $C$-basis of $V$.
\end{lemma}

\begin{proof}
Clearly, one finds a sequence $(v_1,\ldots, v_n)$ in $V$ which satisfies
(\apref{CON}{CON1}{}) and (\apref{CON}{CON2}{}) such that $V$ is isomorphic as
Euclidean vector space to  $\bigoplus_{j=1}^n Cv_j$. From
Propositions~\ref{beta2} and \ref{beta_2_real} it follows that 
\[
\psi\colon\left\{
\begin{array}{ccl}
V & \to & \bigoplus_{j=1}^n Cv_j
\\
v & \mapsto & \sum_{j=1}^n \beta_2(v,v_j)v_j
\end{array}
\right.
\]
is an isomorphism (of Euclidean vector spaces) from $V$ to $\bigoplus_{j=1}^n Cv_j$. Let $v\in V$, $\eta\in C$ and suppose that $v$ is isomorphic to $(\zeta_1v_1,\ldots, \zeta_nv_n)\in \bigoplus_{j=1}^n Cv_j$. Then
\[
\psi(\eta v)= \sum_{j=1}^n \beta_2(\eta v, v_j)v_j = \sum_{j=1}^n\big(\eta\beta_2(v,v_j)\big)v_j = \eta\Big(\sum_{j=1}^n\beta_2(v,v_j)v_j \Big) = \eta\psi(v).
\]
This shows that $\psi$ is indeed an isomorphism of $C$-modules. 
\bewend \end{proof}

\subsection{The $C$-projective space $P_C(E)$}\label{sec_projspace}

Let $(C,V,J)$ be an associative $J^2C$-module structure. Further let $W\sceq
C\oplus V$ be the Euclidean direct sum of $C$ and $V$, and 
$E\sceq C\oplus W = C\oplus C\oplus V$ be that of $C$ and $W$. We define a
$C$-multiplication on $E$ by
\[
\cdot\colon\left\{
\begin{array}{ccl}
C\times E & \to & E
\\
\big( \tau, (\zeta,\eta, v) \big) & \mapsto & (\tau\zeta, \tau\eta, \tau v).
\end{array}
\right.
\]
Because $(C,V,J)$ is associative and hence $\sigma(\tau v) = (\sigma \tau)v$ for
each $v\in V$ and all $\sigma, \tau\in C$ (or equivalently, since the definition
of the product in $C$ does not depend on the choice of $v\in V\mminus\{0\}$, see
Proposition~\ref{division}), $E$ becomes a $C$-module. Since $V$ is a free $C$-module
by Lemma~\ref{Vfree}, also $E$ is a free $C$-module.

Two elements $z_1, z_2$ of $E\mminus\{0\}$ are called \textit{equivalent}
($z_1\sim z_2$) if there is $\tau \in C$  such that $\tau z_1 = z_2$.
Note that $\tau$ is actually in $C\mminus\{0\}$. Then $E$ being a $C$-module
implies that $\sim$ is indeed an equivalence relation on $E\mminus\{0\}$. The
\textit{$C$-projective space} $P_C(E)$ of $E$ is defined as the set of
equivalence classes of $\sim$,
\[ P_C(E) \sceq \big(E\mminus\{0\}\big)/_\sim, \]
endowed with the induced topology and the differential structure generated by the following (standard) charts: Let $\{v_1,\ldots, v_{n-1}\}$ be an orthonormal $C$-basis of $V$. Then $E$ is isomorphic to $C^{n+1}$ both as Euclidean vector space and $C$-module. For each $j=1,\ldots, n+1$ the set
\[ U_j \sceq \big\{ [ (\zeta_1, \ldots, \zeta_{n+1})] \in P_C(E) \ \big\vert\  \zeta_j\not=0 \big\}\]
is open, and the maps $\varphi_j \colon U_j \to C^n$
\[ \varphi_j\big([ (\zeta_1, \ldots, \zeta_{n+1}) ]\big) \sceq \zeta_j^{-1}(\zeta_1, \ldots, \wh \zeta_j, \ldots, \zeta_{n+1}) \]
are pairwise compatible in the sense that they are real differentiable (they
are not $C$-differentiable unless $C$ is commutative). Here, $\wh \zeta_j$ means
that $\zeta_j$ is omitted, hence $(\zeta_1,\ldots,\wh\zeta_j,\ldots,z_{n+1}) =
(\zeta_1,\ldots,\zeta_{j-1},\zeta_{j+1},\ldots, \zeta_{n+1})\in C^n$. Obviously,
the differential structure is independent of the choice of the orthonormal
$C$-basis of $V$, and $P_C(E)$ is a real smooth manifold of dimension $n\cdot
\dim_\R C$. 

Let $\Phi$ be a $C$-sesquilinear hermitian form on $E$. Recall from
Remark~\ref{real} that $\Phi(z,z) \in \R$ for each $z\in E$. Suppose that $q$ is
its \textit{associated quadratic form}, that is
\[
q\colon\left\{
\begin{array}{ccl}
E & \to & \R
\\
z & \mapsto & \Phi(z,z).
\end{array}
\right.
\]
Then we define the following sets:\label{def_vectors}
\begin{align*} 
E_-(\Phi) & \sceq q^{-1}\big((-\infty, 0)\big) &&  \text{the set of
\textit{$\Phi$-negative vectors},}
\\
E_0(\Phi) & \sceq q^{-1}(0)\mminus\{0\} &&  \text{the set of \textit{$\Phi$-zero
vectors}, and}
\\
E_+(\Phi) & \sceq q^{-1}\big((0,\infty)\big) &&  \text{the set of
\textit{$\Phi$-positive vectors}.}
\end{align*}
For each $\tau\in C$ and $z\in E\mminus\{0\}$ we have
\begin{equation}\label{disjoint}
q(\tau z) = \Phi(\tau z, \tau z) = \tau \Phi(z,z) \overline\tau = \Phi(z,z) \tau\overline\tau = q(z) |\tau|^2.
\end{equation}

\begin{lemma}\label{spacedecomp}
Let $\Phi$ be a $C$-sesquilinear hermitian form on $E$. The set $P_C(E)$ equals the disjoint union $P_C\big(E_-(\Phi)\big) \cup P_C\big(E_0(\Phi)\big) \cup P_C\big(E_+(\Phi)\big)$.
\end{lemma}

\begin{proof}
Clearly, $P_C(E) = P_C\big(E_-(\Phi)\big) \cup P_C\big(E_0(\Phi)\big) \cup
P_C\big(E_+(\Phi)\big)$. Hence it remains to prove that this union is disjoint.
Suppose that $q$ denotes the quadratic form associated to $\Phi$.
If we assume for contradiction that 
\[
P_C\big(E_-(\Phi)\big) \cap P_C\big(E_0(\Phi)\big)\not=\emptyset,
\]
then there are equivalent $z_1,z_2\in E\mminus\{0\}$ such that $q(z_1) < 0$ and $q(z_2) = 0$. But then there is $\tau\in C\mminus\{0\}$ such that 
\[
 0=q(z_2) = q(\tau z_1) = q(z_1)|\tau|^2  < 0,
\]
which is a contraction. Hence $P_C\big(E_-(\Phi)\big)\cap
P_C\big(E_0(\Phi)\big) = \emptyset$. Analogously we see that
$P_C\big(E_-(\Phi)\big)\cap P_C\big(E_+(\Phi)\big) = \emptyset$ and
$P_C\big(E_+(\Phi)\big)\cap P_C\big(E_0(\Phi)\big) = \emptyset$. 
\bewend \end{proof}

Let $\pi \colon E\mminus \{0\} \to P_C(E)$ denote the projection on the
equivalence classes. Since $C\mminus\{0\}$ acts homeomorphically on
$E\mminus\{0\}$, the projection $\pi$ is open. Further $\pi$ is continuous by
the definition of the topology on $P_C(E)$.

\begin{proposition}\label{boundary}
Let $\Phi\colon E\times E\to C$ be a $C$-sesquilinear hermitian form, which is
indefinite and non-degenerate. Then
\[ \partial P_C\big(E_-(\Phi)\big) = P_C\big(E_0(\Phi)\big).\]
\end{proposition}

\begin{proof} Let $U\subseteq P_C\big(E_-(\Phi)\big)$. All complements,
closures, interiors, and boundaries of subsets of $E\mminus\{0\}$ are taken in
$E\mminus\{0\}$. At first we show that  
\begin{equation}\label{funny} \complement \pi \left(\complement \partial
\pi^{-1}(U)\right) = \partial U.\end{equation}
To that end let $M\sceq \pi^{-1}(U)$. Then  $\complement M =
\pi^{-1}\left(\complement U\right)$. From $\pi$ being open and continuous it
follows that $(\complement U)^\circ = \pi\big( (\complement M)^\circ \big)$. 
This yields 
\[ \pi\left(\complement\overline{M}\right) = \pi\left( \left(\complement
M\right)^\circ\right) = \left(\complement U\right)^\circ = \complement
\overline{U},\]
hence $\complement \pi(\complement\overline M) = \overline U$.
Again from $\pi$ being open and continuous we get $U^\circ =
\pi\left(M^\circ\right)$. 
Therefore
\begin{align*}
\complement \pi \left(\complement\partial \pi^{-1}(U)\right) & = \complement \pi
\left(\complement\partial M\right)
\\
& = \complement\pi \left(\complement\big( \overline M \cap \complement
M^\circ\big) \right)
\\
& = \complement \left( \pi \left(\complement \overline M\right) \cup
\pi\left(M^\circ\right) \right)
\\
& = \complement \pi\left(\complement \overline M\right) \cap
\complement\pi\left(M^\circ\right) 
\\ 
& = \overline U \cap \complement U^\circ  = \partial U.
\end{align*}
Let $q$ denote the quadratic form associated to $\Phi$. Then $q$ is smooth. 
Since $\Phi$ is non-degenerate, $0$ is a regular value for $q\vert_{
E\mminus\{0\} }$. Hence $E_0(\Phi) = q^{-1}(0)\mminus\{0\}$ is the boundary of
the bounded submanifold $q^{-1}((-\infty,0])\mminus\{0\}$ and therefore also of
$E_-(\Phi)=q^{-1}((-\infty,0))$.  Then the statement follows from \eqref{funny}
and Lemma~\ref{spacedecomp} with $U\sceq P_C\big(E_-(\Phi)\big)$. 
\bewend \end{proof}

\subsection{The $C$-hyperbolic spaces $P_C(E_-(\Psi_1))$ and
$P_C(E_-(\Psi_2))$}\label{sec_hypspaces}

Let $(C,V,J)$ be an associative $J^2C$-module structure. We define two specific
non-degenerate, indefinite $C$-sesquilinear hermitian
forms $\Psi_1$ and $\Psi_2$ and consider the manifolds $P_C\big(
E_-(\Psi_j)\big)$ defined in Section~\ref{sec_projspace}. For each space we choose
a set of representatives and a Riemannian metric on it such that
$P_C\big(E_-(\Psi_1)\big)$ is essentially the ball model from Section~\ref{sec_B}
and $P_C\big(E_-(\Psi_2)\big)$ becomes the Siegel domain model for hyperbolic
space, which is essentially the model $D$ from Section~\ref{sec_D}. 

We denote by
$\GL_C(E)$ the group of all $C$-linear invertible maps $E\to E$. Set
\[ U(\Psi_j, C) \sceq \big\{ g\in \GL_C(E) \ \big\vert\ \forall\, z_1,z_2\in
E\colon \Psi_j(gz_1, gz_2) = \Psi_j(z_1,z_2) \big\}, \]
and let $Z(\Psi_j,C)$ denote the center of $U(\Psi_j,C)$. Recall that $G$
denotes the full isometry group of $B$ resp.\@ $D$.  In this section
we establish a natural and explicit isomorphism between the quotient group
$\PU(\Psi_j,C)\sceq U(\Psi_j,C)/Z(\Psi_j,C)$ and a subgroup $G^\Res$ of $G$.
Moreover, we explicitly characterize $G^\Res$. In Section~\ref{sec_cocycle},
this isomorphism is used to show that the definition of isometric spheres in
literature is subsumed by our definition.

Let $\beta_3\colon W\times W\to C$ be the sum of $\beta_1$ and $\beta_2$,
that is,
\[ \beta_3\big((\eta_1, v_1), (\eta_2, v_2)\big) \sceq \beta_1(\eta_1,\eta_2) + \beta_2(v_1,v_2).\]
We define the maps $\Psi_j\colon E\times E\to C$ ($j=1,2$) by
\begin{align*} 
\Psi_1\big( (\zeta_1,\eta_1,v_1), (\zeta_2,\eta_2,v_2) \big) & \sceq - \beta_1(\zeta_1,\zeta_2) + \beta_1(\eta_1,\eta_2) + \beta_2(v_1,v_2)
\\ & \, = -\beta_1(\zeta_1,\zeta_2) + \beta_3\big((\eta_1,v_1),(\eta_2,v_2)\big)
\intertext{and}
\Psi_2\big( (\zeta_1,\eta_1,v_1), (\zeta_2,\eta_2,v_2) \big) &\sceq -\beta_1(\zeta_1,\eta_2) - \beta_1(\eta_1,\zeta_2) + \beta_2(v_1,v_2).
\end{align*}
Let $q_j$ denote the associated quadratic forms. For all
$(\zeta,w)=(\zeta,\eta,v)\in E$ we have
\begin{align*}
q_1\big( (\zeta,w)\big) & = q_1\big( (\zeta,\eta,v)\big) = -|\zeta|^2 + |\eta|^2 + |v|^2 = - |\zeta|^2 + |w|^2.
\end{align*}
Employing Proposition~\ref{beta1}, we see that 
\begin{align*}
q_2\big((\zeta,\eta,v)\big) &  = - 2\Rea\beta_1(\zeta,\eta) + |v|^2 =
-2\langle\zeta,\eta\rangle + |v|^2.
\end{align*}

\begin{lemma} 
For $j=1,2$, the map $\Psi_j$ is a  $C$-sesquilinear hermitian form on $E$ which
is non-degenerate and indefinite. 
\end{lemma}

\begin{proof}
Propositions~\ref{beta1} and \ref{beta2} imply that $\Psi_j$ is $C$-sesquilinear
hermitian. Suppose that there is $(\zeta,\eta, v)\in E$ such that for all
$(\sigma,\tau, u)\in E$
we have 
\[
 \Psi_1\big( (\zeta,\eta, v),(\sigma,\tau, u)\big) = 0.
\]
Propositions~\ref{beta1} and \ref{beta_2_real} show that
\[
0 = \Psi_1\big( (\zeta,\eta,v), (0,\eta,v)\big)  = |\eta|^2 + |v|^2.
\]
Hence $(\eta, v) = (0,0)$. Further
\[
0 = \Psi_1\big( (\zeta,\eta, v), (\zeta, 0,0)\big) = - |\zeta|^2,
\]
and therefore $\zeta = 0$. This shows that $\Psi_1$ is non-degenerate. Suppose
now that there is $(\zeta,\eta,v) \in E$ such that for all $(\sigma,\tau,u)\in
E$ we have
\[
 \Psi_2\big( (\zeta,\eta, v), (\sigma,\tau, u) \big) = 0.
\]
Then 
\[
0 = \Psi_2\big( (\zeta,\eta, v), (0,0,v) \big) = |v|^2
\]
and therefore $v=0$. Moreover, 
\[
0 = \Psi_2\big( (\zeta,\eta, v), (0,\zeta, 0) \big) = - |\zeta|^2.
\]
Thus $\zeta = 0$ and analogously $\eta= 0$. Hence $\Psi_2$ is non-degenerate.
Finally, for (any) $v\in V\mminus\{0\}$,
\begin{align*}
q_1\big( (1,0,0) \big) & = -1  & \text{and}\qquad\quad q_1\big((0,0,v)\big) & = |v|^2 > 0,
\\
q_2\big( (1,1,0) \big) & = -2 & \text{and}\qquad\quad q_2\big( (0,0,v) \big) & = |v|^2 > 0.
\end{align*}
Therefore, $\Psi_1$ and $\Psi_2$ are indefinite. 
\bewend \end{proof}

\subsubsection{Set of representatives for
$\overline{P_C(E_-(\Psi_1))}$}\label{sec_projmodels}

Proposition~\ref{boundary} shows that we have
\[\overline{P_C\big(E_-(\Psi_j)\big)} = P_C\big(E_-(\Psi_j)\big) \cup
P_C\big(E_0(\Psi_j)\big)\]
for $j\in\{1,2\}$. If $z=(\zeta, w) \in E_-(\Psi_1) \cup E_0(\Psi_1)$, then
$\zeta\not=0$. Therefore the element $[z] \in P_C\big(E_-(\Psi_1)\big) \cup
P_C\big(E_0(\Psi_1)\big)$ is represented by $(1,\zeta^{-1}w)$, and this is the
unique representative of the form $(1,\ast)$. If $z\in E_-(\Psi_1)$, then
\begin{align*}
0 > q_1\big(1,\zeta^{-1}w\big) & =  -1 + |\zeta^{-1}w|^2.
\end{align*}
This shows that $1 > |\zeta^{-1}w|^2$. Conversely, if $w\in W$ with $|w|^2 <
1$, then $[(1,w)]$ is an element of $P_C\big(E_-(\Psi_1)\big)$. Recall the unit
ball $B = \{w\in W \mid |w|<1\}$ in $W$. Then $P_C\big(E_-(\Psi_1)\big)$ and $B$
are in bijection via the map
\[
 [(\zeta,w)] \mapsto \zeta^{-1}w.
\]
The same argument shows that $\partial P_C\big(E_-(\Psi_1)\big) =
P_C\big(E_0(\Psi_1)\big)$ is bijective to $\partial B$ via this map. We
define $\tau_B\colon \overline{P_C\big(E_-(\Psi_1)\big)} \to \overline B$
as
\[
\tau_B\big( [(\zeta,w)] \big) \sceq \zeta^{-1}w.
\]
Its inverse is
\[
\tau_B^{-1}(w) = [(1,w)].
\]
In comparison with Section~\ref{sec_projspace}, we see that $\overline{P_C\big(E_-(\Psi_1)\big)}$ is a subset of $U_1$ and $\tau_B$ a restriction of $\varphi_1$. Therefore, $\tau_B$ is a diffeomorphism between the manifolds $P_C\big(E_-(\Psi_1)\big)$ and $B$, and also between the manifolds with boundary $\overline{P_C\big(E_-(\Psi_1)\big)}$ and $\overline B$.

\subsubsection{Riemannian metric on $B$}

Since $B$ is an open subset of the vector space $W$, we may and shall identify the tangent space at a point of $B$ with $W$. We define a Riemannian metric $\tilde\varrho$ on $B$ by\label{def_tvarrho}
\begin{align*}
\tilde\varrho(p)(X,Y) & \sceq \frac{\langle X,Y\rangle}{1-|p|^2} + \frac{\langle \beta_3(X,p),\beta_3(Y,p)\rangle}{(1-|p|^2)^2}
\\ & = \frac1{(1-|p|^2)^2} \Rea\left( (1-|p|^2) \beta_3(X,Y) + \beta_1\big(\beta_3(X,p),\beta_3(Y,p)\big) \right)
\end{align*}
for all $p\in B$ and all $X,Y\in T_pB=W$. Proposition~\ref{metricident} below shows that $\tilde\varrho$ essentially coincides with the Riemannian metric defined in \eqref{B_metric}.

\begin{lemma}\label{beta3_metric} 
Let $x,y\in W$. Then $\beta_3(x,y)=0$ if and only if $y\in (Cx)^\perp$.
\end{lemma}

\begin{proof}
Let $\zeta,\eta,\xi\in C$. Proposition~\ref{beta1} yields
\begin{align*}
\langle \zeta \overline\eta, \xi \rangle & = \Rea\beta_1\big(\zeta\overline{\eta},\xi\big) = \Rea\big(\zeta \overline{\eta}\,\overline{\xi}\big) = \Rea \big( \zeta \overline{\xi\eta}\big) = \Rea \beta_1(\zeta,\xi\eta) = \langle \zeta, \xi\eta\rangle.
\end{align*}
By Riesz' Representation Theorem we know that $\beta_3(x,y)=0$ if and only if 
\[ \langle\beta_3(x,y),\xi\rangle =0 \quad\text{for all $\xi\in C$.}\]
Supposing that $x=(\zeta,v), y=(\eta, u)\in C\oplus V$, it follows from the definitions of $\beta_3$, $\beta_1$, $\beta_2$ and the inner product on $W$ that 
\begin{align*}
\langle\beta_3(x,y),\xi\rangle & = \langle \beta_1(\zeta,\eta),\xi\rangle + \langle \beta_2(v,u),\xi\rangle
= \langle \zeta\overline\eta, \xi\rangle + \langle \xi u,v\rangle
\\ & = \langle \zeta, \xi\eta\rangle + \langle v,\xi u\rangle = \langle \overline{\xi}\zeta, \eta\rangle + \langle \overline{\xi} v, u\rangle 
\\ & = \langle \overline{\xi} x, y\rangle.
\end{align*}
Hence $\beta_3(x,y) = 0$ if and only if $y\in (Cx)^\perp$. 
\bewend \end{proof}

\begin{proposition}\label{metricident}
The map $\wt\varrho$ coincides, up to a multiplicative factor of $4$, with the
Riemannian metric given by \eqref{B_metric} on $B$.
\end{proposition}

\begin{proof}
For $p=0$ and all $X,Y\in T_0B$ we have
\[
 \wt\varrho(0)(X,Y) = \langle X,Y\rangle = \tfrac14 \langle X,Y\rangle_{0-}.
\] 
Suppose that $p=(\zeta, v)\in B\mminus\{0\}$. We claim that the equivalence class $Cp$ (see Section~\ref{sec_B}) coincides with the $C$-orbit $C\cdot p$ of $p$. For $\zeta = 0$ this is obviously true. Suppose that $\zeta\not=0$. For each $\tau \in C\mminus\{0\}$, we have $\tau p = (\tau \zeta, \tau v)$ and
\[
(\tau\zeta)^{-1}(\tau v) = \zeta^{-1}\tau^{-1}\tau v = \zeta^{-1}v.
\]
Hence $(\zeta, v) \sim (\tau\zeta, \tau v)$. If $\tau = 0$, then $\tau p = 0\in Cp$ by definition. Conversely, suppose that $(\eta, u) \in Cp\mminus\{0\}$. Then 
\[
\eta\zeta^{-1}p = \eta\zeta^{-1}(\zeta, v) = \big(\eta, \eta\zeta^{-1}v\big) = \big(\eta, \eta \eta^{-1}u\big) = (\eta, u).
\]
Thus, $(\eta, u)\in C\cdot p$. Clearly, $0\in Cp\cap C\cdot p$. This shows that $Cp=C\cdot p$.

Now let $p\in B\mminus\{0\}$ and $X,Y\in T_pB$. Suppose first that $X,Y\in (Cp)^\perp$. Then Lemma~\ref{beta3_metric} shows that 
\[
\wt \varrho(p)(X,Y) = \frac{\langle X,Y\rangle}{1-|p|^2}  = \tfrac14\langle X,Y\rangle_{p-}.
\]
Suppose now that $X\in Cp$ and $Y\in (Cp)^\perp$ (or vice versa). Then 
\[
\wt\varrho(p)(X,Y) = \frac{\langle X,Y\rangle}{1-|p|^2} = 0 = \tfrac14\langle X,Y\rangle_{p-}.
\]
Finally suppose that $X,Y\in Cp$. Then $X=\tau_1 p$ and $Y=\tau_2 p$ for some
$\tau_1,\tau_2\in C$. From Lemma~\ref{re_and_im}, Propositions~\ref{beta1} and
\ref{beta2} it follows that 
\begin{align*}
\langle X,Y\rangle & = \Rea \beta_3(X,Y) = \Rea \beta_3(\tau_1 p,\tau_2 p) = \Rea\big( \tau_1\beta_3(p,p)\overline\tau_2\big) = |p|^2 \Rea(\tau_1\overline\tau_2) 
\\
& = |p|^2\langle\tau_1,\tau_2\rangle.
\end{align*}
Then
\begin{align*}
\wt\varrho(p)(X,Y) & = \frac{\langle X,Y\rangle}{1-|p|^2} + \frac{\langle\beta_3(\tau_1p,p),\beta_3(\tau_2p,p)\rangle}{(1-|p|^2)^2}
\\
& = \frac{|p|^2\langle\tau_1,\tau_2\rangle}{1-|p|^2} + \frac{\langle\tau_1\beta_3(p,p),\tau_2\beta_3(p,p)\rangle}{(1-|p|^2)^2}
\\
& = \frac{|p|^2\langle\tau_1,\tau_2\rangle}{1-|p|^2} + \frac{|p|^4\langle\tau_1,\tau_2\rangle}{(1-|p|^2)^2}
\\
& = \frac{|p|^2\langle\tau_1,\tau_2\rangle}{(1-|p|^2)^2}
\\
& = \frac{\langle X,Y\rangle}{(1-|p|^2)^2} = \tfrac14\langle X,Y\rangle_{p-}.
\end{align*}
This completes the proof. 
\bewend \end{proof}

\subsubsection{Induced Riemannian isometries on $B$}

Let  $\pi_B\sceq \tau_B\circ \pi\colon E_-(\Psi_1)\to B$ and\label{def_piB}
suppose that $g\in U(\Psi_1,C)$. Since $g$ is $C$-linear, it induces a (unique)
map $\wt g$ on $B$ by requiring the diagram 
\[\xymatrix{
E_-(\Psi_1)\ar[r]^g \ar[d]_{\pi_B} & E_-(\Psi_1)\ar[d]^{\pi_B}
\\
B \ar@{.>}[r]^{\wt g} & B.
}\]
to commute. Our next goal is to show that each such induced map is a Riemannian
isometry on $B$. 

To simplify proofs we fix an orthonormal $C$-basis $\mathcal
B(V)\sceq \{v_1,\ldots, v_{n-1}\}$ for $V$. Then $\mathcal B(E) \sceq \{e_1,e_2,
\ldots, e_{n+1}\}$, given by 
\begin{align*}
e_1 & \sceq (1,0,0), \qquad e_2 \sceq (0,1,0), \qquad 
& e_k  \sceq (0,0,v_{k-2}) \quad \text{for $k=3,\ldots, n+1$,}
\end{align*}
is an orthonormal $C$-basis for $E$. In the following we will identify each
element in $E$, $W$ and $V$, and each map $g\in U(\Psi_1,C)$ with its
representative with respect to $\mathcal B(E)$. Since $E$ is a left $C$-module,
the representing vector of $z\in E$ is a row, and the application of a map $g\in
U(\Psi_1, C)$ to $z$ corresponds to the multiplication of the row vector of $z$
to the matrix of $g$ (and not vice versa, as usually in linear algebra). Further
we use the notation $z^\ast$ for $\overline z^\top$ (the conjugate transpose of
$z$).

\begin{lemma}\label{beta2_coord} Let $v=(\zeta_1,\ldots, \zeta_{n-1}), u=(\eta_1,\ldots, \eta_{n-1})\in V = C^{n-1}$. Then 
\[ \beta_2(v,u) = \sum_{j=1}^{n-1} \zeta_j\overline{\eta_j} = vu^\ast.\]
\end{lemma}

\begin{proof} Let $v_j,v_k\in\mathcal B(V)$. If $j=k$, then $\beta_2(v_j,v_k) = |v_j|^2 = 1$. If $j\not=k$, then $v_j$ is orthogonal to $Cv_k$, hence
\[ \langle \beta_2(v_j,v_k),\xi\rangle = \langle \xi v_k,v_j\rangle = 0\]
for each $\xi\in C$, hence $\beta_2(v_j,v_k)=0$. The claim now follows by
the $C$-sesquilinearity of $\beta_2$. 
\bewend \end{proof}

\begin{remark} Let $g\in U(\Psi_1,C)$ and suppose that
\[ g=\mat{a}{b}{c^\top}{A} \]
\wrt to $\mathcal B(E)$, where $a\in C$, $b,c\in C^n$ and $A\in C^{n\times n}$. Then the induced map $\tilde g\colon B\to B$ is given by 
\begin{equation}\label{induced}
 \wt g(p) = (a+pc^\top)^{-1}(b+pA).
\end{equation}
Further, the representing matrix of $\Psi_1$ is $\textmat{-1}{}{}{I}$, where
$I$ denotes the $n\times n$ identity matrix. Since $g$ preserves $\Psi_1$, we
find the conditions
\begin{align}
\nonumber \mat{-1}{}{}{I} &= \mat{a}{b}{c^\top}{A}\mat{-1}{}{}{I}\mat{\overline a}{\overline c}{b^\ast}{A^\ast}
\\ \label{J_coord} & = \mat{-|a|^2 + |b|^2}{-a\overline c +
bA^\ast}{-c^\top\overline a + Ab^\ast}{- c^\top\overline c + AA^\ast}
\end{align}
on the entries of $g$.
\end{remark}

The following proposition is proved by a long but straightforward calculation.
We omit this proof here.

\begin{proposition}\label{inducedisos}
Let $g\in U(\Psi_1,C)$. The induced map on $B$ is a Riemannian isometry.
\end{proposition}

In the following we will determine which elements in $U(\Psi_1,C)$ induce the same isometry on $B$. We denote the center of $C$ by $Z(C)$ and set \label{def_centerC}
\[
Z^1(C) \sceq \{ a\in Z(C) \mid |a|=1 \},
\]
which are the central elements in $C$ of unit length. Further we let $Z(\Psi_1,C)$ denote the center of $U(\Psi_1,C)$. 

\begin{lemma} We have $Z(\Psi_1,C) = \{ a\id_E \mid a\in Z^1(C) \}$.
\end{lemma}

\begin{proof} Clearly, $\{ a\id_E \mid a\in Z^1(C) \} \subseteq Z(\Psi_1,C)$. For the converse inclusion relation let  
\[ g=\mat{a}{b}{c^\top}{A}\in Z(\Psi_1, C).\]
For each $d\in C$, $|d|=1$, and each matrix $D\in C^{n\times n}$, $DD^\ast=I$, the matrix
\[ h=\mat{d}{0}{0}{D}\]
is in $U(\Psi_1, C)$. So necessarily,
\[ \mat{ad}{bD}{c^\top d}{AD}=gh=hg=\mat{da}{db}{Dc^\top}{DA}.\]
The left upper entries show that $a\in Z(C)$. Choosing different values for
$d$, but the same (invertible) $D$, we find $b=c=0$. If $D$ runs through all
permutation matrices and all rotation matrices, then we see that $A$ is a
diagonal matrix $\diag(x,\ldots, x)$ with $x\in Z(C)$. By \eqref{J_coord},
$|a|=1=|x|$. 
Then 
\[
 g= \mat{a}{0}{0}{xI}.
\]
Let $h=\mat{d}{u}{w^\top}{B}\in U(\Psi_1,C)$. Then 
\[
\mat{ad}{au}{xw^\top}{xB}=gh=hg=\mat{da}{ux}{w^\top a}{Bx}.
\]
Therefore $au=ux=xu$. For $u\not=0$ it follows that $a=x$. Hence $g=a\id_E$ for
some $a\in Z^1(C)$. 
\bewend \end{proof}

\begin{proposition} Let $g_1,g_2\in U(\Psi_1,C)$. Then $g_1$ and $g_2$ induce
the same isometry on $B$ if and only if  $g_1h=g_2$ for some $h\in Z(\Psi_1,C)$.
\end{proposition}

\begin{proof}
It suffices to show that exactly the elements in $Z(\Psi_1,C)$ induce $\id_B$.
Let $g$ be an element of $U(\Psi_1,C)$ and suppose that 
\[
g=\mat{a}{b}{c^\top}{A}
\]
\wrt $\mc B(E)$, where $a\in C$, $b,c\in C^n$ and $A\in C^{n\times n}$. Then the induced isometry on $B$ is 
\[
\wt g\colon \left\{
\begin{array}{ccl}
B & \to & B
\\
p & \mapsto & (a+pc^\top)^{-1}(b+pA).
\end{array}
\right.
\]
Suppose that $\wt g = \id_B$. Then 
\[
 0 = \wt g(0) = a^{-1}b,
\]
which yields that $b=0$ and, by \eqref{J_coord}, $|a|=1$. Hence 
\[
g = \mat{a}{0}{c^\top}{A}.
\]
Now \eqref{J_coord} shows that $0= -a\overline c$, which implies that $c=0$. Thus,
\[
\wt g(p) = a^{-1}pA
\]
for all $p\in B$. Suppose that $A=(a_{ij})_{i,j=1,\ldots,n}$. Let $j\in
\{1,\ldots, n\}$ and consider the vector $p=(p_i)_{i=1,\ldots,n}$ with
$p_j=\tfrac12$ and $p_i=0$ for $i\not=j$. Then $p\in B$ and
\[
p = \wt g(p) = \tfrac12 a^{-1}(a_{j1},\ldots, a_{jn}).
\]
Therefore $a_{ij} = 0$ for $i\not=j$ and $a_{jj}=a$. This shows that $A=aI$. 

Now let $\zeta\in C$ with $|\zeta|<1$. Then $(\zeta,0)\in B$ and therefore
\[
(\zeta,0) = \wt g(\zeta,0) = a^{-1}(\zeta,0)a.
\]
Hence $a\zeta = \zeta a$. By scaling, this equality holds for all $\zeta\in C$.
Thus $a\in Z^1(C)$, which yields that $g\in Z(\Psi_1,C)$. Conversely, each
element of $Z(\Psi_1,C)$ clearly induces $\id_B$ on $B$. 
\bewend \end{proof}
Let 
\[
\PU(\Psi_1,C) \sceq U(\Psi_1,C)/Z(\Psi_1,C)
\]
and denote the coset of $g\in U(\Psi_1,C)$ by $[g]$.  As before we use the
notation $\wt g$ for the isometry on $B$ induced by $g\in U(\Psi_1,C)$. Recall
that $G$ is the full isometry group of $B$.

\begin{corollary}\label{iso1}
The map 
\[
j_{\Psi_1}\colon\left\{
\begin{array}{ccl}
\PU(\Psi_1,C) & \to & G
\\{}
[g] & \mapsto & \wt g
\end{array}
\right.
\]
is a monomorphism of groups.
\end{corollary}

The next goal is to characterize the subgroup of $G$ to which $\PU(\Psi_1,C)$
is isomorphic. For this it is most convenient to work with the space
$\overline{P_C\big(E_-(\Psi_2)\big)}$.

\subsubsection{Set of representatives for $\overline{P_C\big(E_-(\Psi_2)\big)}$}

If $z=(\zeta,\eta,v)\in E_-(\Psi_2)$, then 
\[
 0 > q_2\big( (\zeta,\eta,v) \big) = -2\langle\zeta,\eta\rangle + |v|^2,
\]
which shows that $\zeta\not=0$. Thus, $[z]\in P_C\big(E_-(\Psi_2)\big)$ is
represented by $(1,\zeta^{-1}\eta,\zeta^{-1}v)$ and this is the unique
representative of the form $(1,*,*)$. We note that 
\[
0> q_2\big( (1,\zeta^{-1}\eta,\zeta^{-1}v) \big) = -2\langle 1, \zeta^{-1}\eta\rangle + |\zeta^{-1}v|^2 = -2\Rea (\zeta^{-1}\eta) + |\zeta^{-1}v|^2,
\]
and therefore
\[
\Rea(\zeta^{-1}\eta) > \tfrac12 |\zeta^{-1}v|^2.
\]
Hence, if we define
\[
  H\sceq \left\{ (\tau,u)\in C\oplus V \left\vert\  \Rea(\tau) > \tfrac12|u|^2 \right.\right\},
\]
then
\[
(\zeta^{-1}\eta, \zeta^{-1}v) \in H.
\]
Conversely, if $(\tau,u)\in H$, then $[(1,\tau,u)] \in P_C(E_-(\Psi_2))$. 

If $z=(\zeta,\eta,v) \in E_0(\Psi_2)$, then either $\zeta\not=0$ or
$z=(0,\eta,0)$ with $\eta\not=0$. Applying the previous argumentation we see
that the set of elements in $P_C\big(E_0(\Psi_2)\big)$ which have a
representative $(\zeta,\eta, v)\in E_0(\Psi_2)$ with $\zeta\not=0$ is bijective
to 
\[
\left\{ (\tau,u)\in C\oplus V\left\vert\ \Rea(\tau) = \tfrac12|u|^2 \right.\right\}
\]
via
\[
[(\zeta,\eta,v)] \mapsto (\zeta^{-1}\eta, \zeta^{-1}v).
\]
If $z=(0,\eta,0)$ with $\eta\not=0$, then $[z]$ is represented by $(0,1,0)$.
Let
\[
\hg \sceq \left\{ (\tau,u)\in C\oplus V\left\vert\ \Rea(\tau) \geq \tfrac12|u|^2\right.\right\} \cup \{\infty\}
\]
denote the closure of $H$ in the one-point compactification $(C\oplus V) \cup
\{\infty\}$ of $C\oplus V$. Then the map $\tau_H\colon
\overline{P_C\big(E_-(\Psi_2)\big)} \to \hg$,
\[
\tau_H\big( [(\zeta,\eta,v)] \big) \sceq 
\begin{cases}
(\zeta^{-1}\eta,\zeta^{-1}v) & \text{if $\zeta\not=0$,}
\\
\infty & \text{if $\zeta=0$,}
\end{cases}
\]
is a bijection with inverse map
\[
\tau_H^{-1}(\infty)  = [(0,1,0)] \quad\text\quad \tau_H^{-1}(\tau,u) = [(1,\tau,u)].
\]
Since $\tau_H\vert_{P_C(E_-(\Psi_2))}$ is a restriction of the chart map $\varphi_1$ from Section~\ref{sec_projspace}, $\tau_H$ is a diffeomorphism between $P_C\big(E_-(\Psi_2)\big)$ and $H$.

\subsubsection{Riemannian metric and induced isometries on $H$}

Recall the model $D$ from Section~\ref{sec_D}. The map 
\begin{align*}
\beta&\colon\left\{
\begin{array}{ccl}
\hg & \to & \dg
\\
\infty & \mapsto & \infty
\\
(\zeta, v) & \mapsto & (\zeta, \sqrt{2}v)
\end{array}
\right.
\end{align*}
is clearly a diffeomorphism between the manifolds $\hg$ and $\dg$ with
boundary, and hence a diffeomorphism between $H$ and $D$. We endow $H$ with a
Riemannian metric by requiring that $\beta$ be an isometry.

Let $\pi_H\sceq \tau_H\circ \pi\colon E_-(\Psi_2) \to H$. As
before, each $g\in U(\Psi_2,C)$ defines a (unique) map $\wt g$ on $H$ which
makes the diagram
\[
\xymatrix{
E_-(\Psi_2) \ar[r]^g \ar[d]_{\pi_H} & E_-(\Psi_2) \ar[d]^{\pi_H}
\\
H \ar@{..>}[r]^{\wt g} & H
}
\]
commutative. 

Let $Z(\Psi_2,C)$ denote the center of $U(\Psi_2,C)$ and set 
\[
\PU(\Psi_2,C)\sceq U(\Psi_2,C)/Z(\Psi_2,C).
\]
In the following we will use the results from the previous subsections to show
that $\PU(\Psi_2,C)$ is isomorphic to a subgroup of the full isometry group $G$
of $H$.
We consider the map
\[
T \colon\left\{
\begin{array}{ccl}
E & \to & E 
\\
(\zeta,\eta,v) &\mapsto & \left( \tfrac{1}{\sqrt{2}}(\zeta-\eta), \tfrac{1}{\sqrt{2}}(\zeta+\eta), v \right).
\end{array}
\right.
\]
Then $T$ is $C$-linear and invertible with inverse map $T^{-1}\colon E\to E$, 
\[
T^{-1}(\zeta,\eta, v) = \left( \tfrac{1}{\sqrt{2}}(\zeta+\eta), \tfrac{1}{\sqrt{2}}(-\zeta+\eta), v \right).
\]
The following lemma is shown by a straightforward calculation.

\begin{lemma}\label{viaT}
The map $T$ transforms $\Psi_2$ into $\Psi_1$, \ie $\Psi_2\circ (T\times T) =
\Psi_1$. Further, we have $T\big( E_-(\Psi_1) \big)
= E_-(\Psi_2)$ and $T\big( E_0(\Psi_1) \big) = E_0(\Psi_2)$.
\end{lemma}

\begin{proposition}\label{indisos2}
The map
\[
\lambda\colon\left\{
\begin{array}{ccl}
U(\Psi_1,C) & \to & U(\Psi_2,C)
\\
g & \mapsto & T\circ g \circ T^{-1}
\end{array}
\right.
\]
is an isomorphism with $\lambda\big( Z(\Psi_1,C)\big) = Z(\Psi_2,C)$.
\end{proposition}

\begin{proof}
Let $g\in U(\Psi_1,C)$. Since $T\in \GL_C(E)$, we have $\lambda(g)\in \GL_C(E)$. Lemma~\ref{viaT} shows that
\begin{align*}
\Psi_2\circ \big( TgT^{-1}\times TgT^{-1}\big) & = \Psi_2\circ\big( T\times T)\circ \big(gT^{-1}\times gT^{-1}\big)
\\
& = \Psi_1\circ\big( g\times g\big) \circ \big( T^{-1}\times T^{-1}\big) = \Psi_1\circ \big( T^{-1}\times T^{-1}\big) 
\\
& = \Psi_2.
\end{align*}
Hence $\lambda(g) \in U(\Psi_2,C)$. Clearly, $\lambda$ is an isomorphism of
groups, which shows that $\lambda\big( Z(\Psi_1,C)\big) =
Z(\Psi_2,C)$. 
\bewend \end{proof}

For the proof of the following proposition recall the Cayley transform $\mc C\colon B\to D$ from Section~\ref{sec_isomgroup}. For each $\zeta\in C\mminus\{1\}$ we have
\begin{align*}
(1-\zeta)^{-1}(1+\zeta) & = |1-\zeta|^{-2}(1-\overline\zeta)(1+\zeta) 
= |1-\zeta|^{-2}(1-\overline\zeta+\zeta -|\zeta|^2) 
\\ & = |1-\zeta|^{-2}(1 + 2\Ima\zeta - |\zeta|^2).
\end{align*}
Hence
\[
 \mc C(\zeta,v) = (1-\zeta)^{-1}(1+\zeta,2v).
\]

\begin{proposition}\label{def_j2}
Let $g\in U(\Psi_2,C)$. Then the induced map $\wt g$ on $H$ is an isometry. Further, the map
\[
j_{\Psi_2}\colon\left\{
\begin{array}{ccl}
\PU(\Psi_2,C) & \to & G
\\{}
[g] & \mapsto & \wt g
\end{array}
\right.
\]
is well-defined and a monomorphism of groups.
\end{proposition}

\begin{proof}
The map $T$ induces the map $\wh T\colon P_C\big(E_-(\Psi_1)\big)\to P_C\big(E_-(\Psi_2)\big)$ given by 
\[
\wh T\big( [(\zeta,\eta,v)] \big) = \left[ \big( \tfrac{1}{\sqrt{2}}(\zeta-\eta), \tfrac{1}{\sqrt{2}}(\zeta+\eta), v \big) \right]
\]
and the map $\wt T \colon B \to H$ given by
\begin{align*}
\wt T\big( (\eta,v) \big) & = \tau_H\big( \wh T[(1,\eta,v)] \big) = \tau_H \left( \big[\tfrac{1}{\sqrt{2}}(1-\eta), \tfrac{1}{\sqrt{2}}(1+\eta), v\big] \right) 
\\
& = \tau_H\left( \big[(1, (1-\eta)^{-1}(1+\eta), (1-\eta)^{-1}\sqrt{2}v) \big] \right)
\\
& = (1-\eta)^{-1}(1+\eta, \sqrt{2}v).
\end{align*}
Therefore, $\wt T = \beta^{-1}\circ \mc C$. From Propositions~\ref{Cayleyisometry} and
\ref{metricident} it follows that the isometry group of $B$ and that of
$H$ are identical. Then Propositions~\ref{inducedisos} and \ref{indisos2} imply
that for each $g\in U(\Psi_2,C)$, the induced map $\wt g$ is an isometry on $H$.
Moreover, Proposition~\ref{indisos2} shows that $\lambda$ factors to a map 
\[
\wt\lambda\colon \PU(\Psi_1,C) \to \PU(\Psi_2,C).
\]
Recall the map $j_{\Psi_1}$ from Corollary~\ref{iso1}. Then the diagram
\[
\xymatrix{
\PU(\Psi_1,C)\ar[dr]_{j_{\Psi_1}} \ar[rr]^{\wt\lambda} && \PU(\Psi_2,C)\ar[dl]^{j_{\Psi_2}}
\\
& G
}
\]
commutes, which completes the proof. 
\bewend \end{proof}

\subsubsection{Lifted isometries}\label{sec_liftedisos}

In this section we determine which isometries on $H$ are induced from an
element in $U(\Psi_2,C)$. To that end we need explicit formulas for the action
of an element $g\in G$ on $H$, which are provided by the following remark.

\begin{remark}
Let $g\in G$. Section~\ref{sec_isomgroup} provides explicit formulas for the action
of $g$ on $D$. Using the isometry $\beta\colon H\to D$, the action of $g$ on $H$
is given by
\[
 g^H\sceq \beta^{-1}\circ g \circ \beta \colon H\to H.
\]
Evaluating this formula, we find the following action laws. For the geodesic inversion $\sigma$ we have
\[
 \sigma^H(\zeta, v) = \zeta^{-1}(1,-v).
\]
For $a_s\in A$ we get
\[
 a_s^H(\zeta,v) = \big(s\zeta, s^{1/2}v\big).
\]
For $n=(\xi,w)\in N$ it follows
\[
n^H(\zeta,v) = \left(\zeta+\xi+\tfrac14|w|^2 + \beta_2\big( v, \tfrac{1}{\sqrt{2}}w\big), \tfrac{1}{\sqrt{2}}w+v\right).
\]
For $m=(\varphi,\psi)\in M$ we have
\[
 m^H(\zeta,v) = \big(\varphi(\zeta),\psi(v)\big).
\]
\end{remark}

\begin{proposition}\label{rep_sigma}
A representative of $\sigma^H$ in $U(\Psi_2, C)$ is
\[ g(\zeta, \eta, v) = (\eta, \zeta, -v).\]
\end{proposition}

\begin{proposition}\label{rep_a}
Let $a_s\in A$. Then
\[ g(\zeta, \eta, v) = \left( s^{-1/2}\zeta, s^{1/2}\eta, v \right) \]
is a representative of $a_s$ in $U(\Psi_2,C)$.
\end{proposition}

\begin{proof}
Obviously, $g$ is $C$-linear and induces $a_s$ on $B$. Further 
\begin{align*}
\Psi_2\big(g(\zeta_1,\eta_1,v_1), g(\zeta_2,\eta_2,v_2)\big) & = \Psi_2\big( (s^{-1/2}\zeta_1, s^{1/2}\eta_1,v_1), (s^{-1/2}\zeta_2, s^{1/2}\eta_2, v_2) \big) 
\\ & = - s^{-1/2}\zeta_1\overline{\eta_2}s^{1/2} - s^{-1/2}\eta_1\overline{\zeta_2}s^{1/2} + \beta_2(v_1,v_2)
\\ & = -\zeta_1\overline{\eta_2} - \eta_1\overline\zeta_2 + \beta_2(v_1,v_2)
\\ & = \Psi_2\big( (\zeta_1,\eta_1,v_1), (\zeta_2,\eta_2, v_2) \big). 
\end{align*}
Hence $g \in U(\Psi_2, C)$. 
\bewend \end{proof}

\begin{proposition}\label{rep_n}
Let $n=(\xi, w)\in N$. A representative of $n$ in $U(\Psi_2,C)$ is
\[ g(\zeta, \eta, v) = \left(\zeta, \zeta\left(\xi + \tfrac14|w|^2\right) + \eta + \tfrac{1}{\sqrt{2}}\beta_2(v,w), \tfrac{1}{\sqrt{2}}\zeta w +v\right).\]
\end{proposition}

\begin{proof} The map $g$ is clearly $C$-linear and induces $n$ on $B$. To see that $\Psi_2$ is $g$-invariant, we calculate
\begin{align*}
\Psi_2\big( g(\zeta_1&,\eta_1,v_1), g(\zeta_2,\eta_2,v_2) \big) 
\\ & = \Psi_2\Big[ \left( \zeta_1, \zeta_1\left( \xi + \tfrac14|w|^2\right) + \eta_1 + \tfrac{1}{\sqrt{2}}\beta_2(v_1,w), \tfrac{1}{\sqrt{2}}\zeta_1w + v_1\right),
\\ & \qquad\quad \left(\zeta_2, \zeta_2\left(\xi + \tfrac14|w|^2\right) + \eta_2 + \tfrac{1}{\sqrt{2}}\beta_2(v_2,w), \tfrac{1}{\sqrt{2}}\zeta_2 w + v_2\right) \Big]
\\ & = -\zeta_1 \left[ \left(-\xi + \tfrac14|w|^2\right)\overline\zeta_2 + \overline\eta_2 + \tfrac{1}{\sqrt{2}}\beta_2(w,v_2)\right] 
\\ & \quad - \left[ \zeta_1\left(\xi + \tfrac14|w|^2\right) + \eta_1 + \tfrac{1}{\sqrt{2}}\beta_2(v_1,w) \right]\overline\zeta_2 
+ \beta_2\left(\tfrac{1}{\sqrt{2}}\zeta_1 w + v_1, \tfrac{1}{\sqrt{2}}\zeta_2 w + v_2\right)
\\ & = \zeta_1\xi\overline\zeta_2 - \tfrac14|w|^2\zeta_1\overline\zeta_2 - \zeta_1\overline\eta_2 - \tfrac{1}{\sqrt{2}}\zeta_1\beta_2(w,v_2) - \zeta_1\xi\overline\zeta_2 - \tfrac14|w|^2\zeta_1\overline\zeta_2 
\\ & \quad - \eta_1\overline\zeta_2 - \tfrac{1}{\sqrt{2}}\beta_2(v_1,w)\overline\zeta_2 + \tfrac12|w|^2\zeta_1\overline\zeta_2 + \tfrac{1}{\sqrt{2}}\zeta_1\beta_2(w,v_2)
\\ & \quad +\tfrac{1}{\sqrt{2}}\beta_2(v_1,w)\overline\zeta_2 + \beta_2(v_1,v_2)
\\ & = \Psi_2\big( (\zeta_1,\eta_1,v_1), (\zeta_2,\eta_2, v_2)\big).
\end{align*}
This completes the proof. 
\bewend \end{proof}

The remaining part of this section is devoted to the discussion which elements
of $M$ have a representative in $U(\Psi_2,C)$. The situation for $M$ is much 
more involved than the proofs of Propositions~\ref{rep_sigma}-\ref{rep_n}. In
particular, it will turn out that in general not every element of $M$ can be
lifted to $U(\Psi_2,C)$.

\begin{lemma}\label{commutes}
Let $m=(\varphi,\psi)\in M$ and suppose that $\varphi$ is an inner automorphism of $C$. Then 
\[
 \left\{ \beta_2\big(\psi(v_1),\psi(v_2)\big) \left\vert\  v_1,v_2\in V \vphantom{\beta_2\big(\psi(v_1),\psi(v_2)\big) } \right.\right\} = C.
\]
\end{lemma}

\begin{proof} Let $a\in C\mminus\{0\}$ and suppose that $\varphi(\zeta) = a^{-1}\zeta a$  for all $\zeta \in C$. Choose $v\in V$ with $|v|=1$. For each $\zeta\in C$ we find
\begin{align*}
\beta_2\big(\psi(a\zeta a^{-1}v), \psi(v)\big) & = \beta_2\big(\varphi(a\zeta a^{-1})\psi(v),\psi(v)\big)
 = \beta_2\big(\zeta \psi(v),\psi(v)\big) 
\\
& = \zeta \beta_2(\psi(v),\psi(v)) = \zeta |\psi(v)|^2 = \zeta |v|^2 
\\
&= \zeta.
\end{align*}
Therefore 
\[
C \subseteq \left\{ \beta_2\big(\psi(v_1),\psi(v_2)\big) \left\vert\  v_1,v_2\in V \vphantom{\beta_2\big(\psi(v_1),\psi(v_2)\big)}   \right.\right\}.
\]
The converse inclusion relation clearly holds by the range of $\beta_2$. 
\bewend \end{proof}

Let $i_H\colon \hg\to E_-(\Psi_2)\cup E_0(\Psi_2)$ be any section of
\[ 
\pi_H = \tau_H\circ \pi\colon E_-(\Psi_2)\cup E_0(\Psi_2) \to \hg.
\]
Let $\wt g\in G$ and recall from Proposition~\ref{extends} that $g$ extends
continuously to $\hg$. If the element $g \in U(\Psi_2,C)$ is a representative of
$\wt g$, then the diagram
\[
\xymatrix{
E_-(\Psi_2)\cup E_0(\Psi_2) \ar[r]^g & E_-(\Psi_2)\cup E_0(\Psi_2)\ar[d]^{\pi_H}
\\
\hg \ar[u]_{i_H} \ar[r]^{\wt g} & \hg
}
\]
commutes. We will make use of this fact in the proof of Proposition~\ref{rep_m} below. For convenience we define $i_H$ by
\begin{equation}\label{def_iH}
 i_H(\infty) \sceq (0,1,0)\quad\text{and}\quad i_H(\eta, v) \sceq (1,\eta, v).
\end{equation}

\begin{proposition}\label{rep_m}
Let $m=(\varphi,\psi)\in M$. Then there exists a representative of $m$ in
$U(\Psi_2,C)$ if and only if $\varphi=\id$. In this case,
\[
 g(\zeta,\eta,v) = (\zeta,\eta,\psi(v) )
\]
is such a representative.
\end{proposition}

\begin{proof}
Suppose first that $m=(\varphi,\psi)$ with $\varphi=\id$. We will show that 
\[
g\colon\left\{
\begin{array}{ccl}
E & \to & E
\\
(\zeta,\eta,v) & \mapsto & (\zeta,\eta,\psi(v))
\end{array}
\right.
\]
is an element of $U(\Psi_2,C)$. To that end let $(\zeta_1,\eta_1,v_1),(\zeta_2,\eta_2,v_2)\in E$ and $\zeta\in C$. Then 
\begin{align*}
g\big( \zeta(\zeta_1,\eta_1,v_1) + (\zeta_2,\eta_2,v_2) \big) & = \big( \zeta\zeta_1+\zeta_2, \zeta\eta_1+\eta_2,\psi(\zeta v_1+v_2) \big)
\\
& = \big( \zeta\zeta_1, \zeta\eta_1, \psi(\zeta v_1) \big) + \big(\zeta_2 ,\eta_2 , \psi(v_2) \big)
\\
& = \big( \zeta\zeta_1 ,\zeta\eta_1 , \varphi(\zeta)\psi(v_1)\big) + g(\zeta_2,\eta_2,v_2)
\\
& = \big(\zeta\zeta_1 ,\zeta\eta_1 , \zeta \psi(v_1) \big) + g(\zeta_2,\eta_2,v_2)
\\
& = \zeta g(\zeta_1,\eta_1,v_1) + g(\zeta_2,\eta_2,v_2).
\end{align*}
This shows that $g$ is $C$-linear. The map $g$ is obviously invertible, hence 
$g\in \GL_C(E)$. Lemma~\ref{howch}\eqref{howchii} implies that $\Psi_2$ is
invariant under $g$. Hence $g\in U(\Psi_2,C)$. Clearly, $g$ induces $m$.

Suppose now that $m=(\varphi,\psi)\in M$ and that there is a representative $g$
of $m$ in $U(\Psi_2,C)$. We have to show that $\varphi=\id$. Since $m(0) = 0$,
it follows that 
\[
g\big( i_H(0) \big) = g(1,0,0) \in \pi_H^{-1}(0) = \big( C\mminus\{0\} \big) \times \{0\} \times \{0\}.
\]
Thus, there is $a\in C\mminus\{0\}$ such that $g(1,0,0) = (a,0,0)$. Further
$m(\infty) = \infty$. The same argument shows that there is $b\in C\mminus\{0\}$
such that 
\[
g\big( i_H(\infty) \big) = g( 0,1,0) = (0,b,0).
\]
Then for each $\zeta\in C$ with $\Rea \zeta \geq 0$ we have
\begin{align*}
(\varphi(\zeta),0) & = m(\zeta,0) = \pi_H\big( g(i_H(\zeta,0)) \big) =
\pi_H\big( g(1,\zeta,0) \big) = \pi_H\big( (a,\zeta b, 0)\big) 
\\ & = (a^{-1}\zeta b, 0).
\end{align*}
Thus $\varphi(\zeta) = a^{-1}\zeta b$ for all $\zeta \in C$ with $\Rea \zeta \geq 0$. Now
\[
 1 = \varphi(1) = a^{-1}b
\]
and therefore $b=a$. Hence $\varphi(\zeta) = a^{-1}\zeta a$ for all $\zeta\in C$
with $\Rea\zeta \geq 0$. Suppose now that $(\eta, v)\in H$ is arbitrary. Then 
\begin{align*}
g\big( i_H(\eta,v)\big) & = g(1,\eta, v) = g(1,\eta,0)+g(0,0,v)
\\
& = (a,\eta a, 0) + (\sigma, \tau, u) = (a+\sigma + \eta a+\tau, u)
\end{align*}
where $(\sigma, \tau, u)$ depends only on $v$. From 
\begin{align*}
 (a^{-1}\eta a, \psi(v)) & = (\varphi(\eta),\psi(v)) = \pi_H(a+\sigma, \eta a+
\tau, u) 
\\
&= \big( (a+\sigma)^{-1}(\eta a + \tau), (a+\sigma)^{-1}u\big)
\end{align*}
it follows that 
\[
 a^{-1}\eta a = (a+\sigma)^{-1}\eta a + (a+\sigma)^{-1}\tau
\]
for all $\eta$ with $\Rea(\eta)>\tfrac12|v|^2$. Pick $k\in\R$ with
$k>\tfrac12|v|^2$. Then 
\[
 k= a^{-1}ka = k(a+\sigma)^{-1}a + (a+\sigma)^{-1}\tau
\]
and
\[
 2k = 2k(a+\sigma)^{-1} a + (a+\sigma)^{-1}\tau.
\]
Solving this system for $\sigma$ and $\tau$, we find $\sigma=\tau = 0$. Hence
\[
 g(1,\eta, v) = (a, \eta a, a\psi(v)).
\]
Since $g$ has to be $C$-linear, it follows that 
\[
g(\zeta,\eta, v) = (\zeta a, \eta a, a \psi(v))
\]
for all $(\zeta,\eta, v)\in E$. We derive further properties of $a$. Let $v\in
V\mminus\{0\}$. Then $g$ being in $U(\Psi_2,C)$ yields that 
\begin{align*}
|v|^2 & = q_2\big( (0,0,v) \big) = q_2\big( g(0,0,v) \big) = |a\psi(v)|^2 = |a|^2 |\psi(v)|^2 = |a|^2 |\psi(v)|^2,
\end{align*}
hence $|a|^2 = 1$. Again using that $g\in U(\Psi_2,C)$ we find that for all $v_1,v_2\in V$ 
\begin{align*}
\beta_2\big(\psi(v_1),\psi(v_2)\big) & = \Psi_2\big( (0,0,v_1),(0,0,v_2) \big)
\\
& = \Psi_2\big( g(0,0,v_1), g(0,0,v_2) \big) = \Psi_2\big( (0,0,a\psi(v_1)), (0,0,a\psi(v_2)) \big)
\\
& = \beta_2\big(a\psi(v_1),a\psi(v_2) \big) = a\beta_2\big(\psi(v_1),\psi(v_2)\big)\overline a
\\
& = a\beta_2\big(\psi(v_1),\psi(v_2)\big) a^{-1}.
\end{align*}
Before we can apply Lemma~\ref{commutes} we have to show that $\varphi(\zeta) =
a^{-1}\zeta a$ for all $\zeta\in C$. Let $\zeta\in C$ with $\Rea \zeta < 0$ and
consider the decomposition $\zeta = \zeta_1 + \zeta_2$ with $\zeta_1\in \R$ and
$\zeta_2\in C'$. Then the $\R$-linearity of $\varphi$ yields
\begin{align*}
\varphi(\zeta) & = \varphi(\zeta_1+\zeta_2) = \varphi(\zeta_1) + \varphi(\zeta_2) = -\varphi(-\zeta_1) + a^{-1}\zeta_2 a
\\
& = -a^{-1}(-\zeta_1)a + a^{-1}\zeta_2 a = a^{-1}\zeta_1 a + a^{-1}\zeta_2 a = a^{-1}(\zeta_1+\zeta_2) a = a^{-1}\zeta a.
\end{align*}
Then Lemma~\ref{commutes} implies that $a\in Z(C)$. Therefore $\varphi=\id$. 
\bewend \end{proof}

We set 
\begin{align*}
M^\Res & \sceq \{ (\varphi,\psi)\in M\mid \varphi=\id\} \quad\text{and}\quad 
G^\Res  \sceq N\sigma M^\Res A N \cup M^\Res A N.
\end{align*}
Further we define a map $\varphi_H\colon G^\Res \to \PU(\Psi_2,C)$ as follows: For $\wt g=\sigma$, $\wt g\in A$, $\wt g\in N$ or $\wt g\in M^\Res$ we set $\varphi_H(\wt g) \sceq [g]$, where $g$ is the lift of $\wt g$ as in Propositions~\ref{rep_sigma}, \ref{rep_a}, \ref{rep_n} or \ref{rep_m}, resp. For $\wt g = n_2\sigma m a_s n_1 \in N\sigma M^\Res A N$ we define
\begin{equation}\label{def_varphiH}
\varphi_H(\wt g) \sceq \varphi_H(n_2) \varphi_H(\sigma) \varphi_H(m)
\varphi_H(a_s) \varphi_H(n_1),
\end{equation}
and likewise for $\wt g = ma_sn\in M^\Res A N$ we set
\[
 \varphi_H(\wt g) \sceq \varphi_H(m)\varphi_H(a_s)\varphi_H(n).
\]
In other words, we extend $\varphi_H$ to a group homomorphism. Since the Bruhat
decomposition of an element $\wt g\in G^\Res$ is unique and the Bruhat
decomposition of $G^\Res$ can be directly transfered to $\PU(\Psi_2,C)$, the map
$\varphi_H$ is indeed well-defined. By our previous considerations, $\varphi_H$
is
even a group isomorphism.

The following remark shows that $M^\Res$ is not necessarily all of $M$.

\begin{remark} The classification \cite[Theorem~3.1]{Koranyi_Ricci_proofs} of
division algebras arising from associative $J^2C$-module
structures $(C,V,J)$ shows that $C$ is either real or complex or quaternionic
numbers. In the following we show that $M=M^\Res$ for $C=\R$, but $M\not=
M^\Res$ for $C=\C$ or $C=\h$ (quaternions).
\begin{enumerate}[(i)]
\item \label{Mi} Let $C=\R$ and suppose that $m=(\varphi,\psi)\in M$. We claim
that $\varphi=\id$. Since $\varphi\colon \R\to\R$ is a norm-preserving
endomorphism of $\R$, the map $\varphi$ is either $\id$ or $-\id$. Assume for
contradiction that $\varphi=-\id$. Let $(\zeta,v) \in C\times V$ such that
$\zeta \not= 0$ and $v\not=0$. Then 
\[
 \varphi(\zeta)\psi(v) = -\zeta \psi(v) = -\psi(\zeta v) \not= \psi(\zeta v).
\]
Hence $m\notin M$. This is a contradiction and therefore $\varphi= \id$.
\item Let $C=\C = V$ and suppose that $\varphi=\psi$ are complex conjugation.
For all elements $(\zeta,v)\in C\oplus V = \C^2$ we have $\varphi(\zeta)\psi(v)
=
\overline\zeta\, \overline v = \overline{\zeta v} = \psi(\zeta v)$. Clearly,
$\varphi,\psi$ are $\R$-linear endomorphisms of the Euclidean vector space $\C$.
Therefore $m=(\varphi,\psi) \in M$, but $m\notin M^\Res$.
\item Let $C=\h = V$. Define $\varphi\colon C\to C$ and $\psi\colon V\to V$ by
\begin{align*}
\varphi(a+ib+jc+kd) & \sceq a-ib-jc+kd
\\
\psi(a+ib+jc+kd) & \sceq -a + ib + jc - kd
\end{align*}
for $a+ib+jc+kd\in \h$. Clearly, $\varphi,\psi$ are $\R$-linear endomorphisms of the Euclidean vector space $\h$. We claim that $J\circ (\varphi\times\psi) = \psi\circ J$. To that end let $\zeta = a_1+ib_1 + jc_1 + kd_1\in C$ and $v = a_2+ib_2 + jc_2 + kd_2 \in V$. Then 
\begin{align*}
\zeta v & = a_1a_2 - b_1b_2 - c_1c_2 - d_1d_2 + i( a_1b_2+ b_1a_2 + c_1d_2 - d_1c_2) 
\\
& \quad  + j(a_1c_2 - b_1d_2 + c_1a_2 +d_1b_2) + k(a_1d_2 + b_1c_2 - c_1b_2 + d_1a_2).
\end{align*}
Therefore
\begin{align*}
\psi(\zeta v) & = -a_1a_2 + b_1b_2 + c_1c_2 + d_1d_2 + i(a_1b_2 + b_1a_2 + c_1d_2 - d_1c_2) 
\\
& \quad + j(a_1c_2 - b_1d_2 + c_1a_2 + d_1b_2) + k(-a_1d_2 - b_1c_2 + c_1b_2 - d_1a_2).
\end{align*}
On the other side we get
\begin{align*}
\varphi(\zeta) \psi(v) & = (a_1 - ib_1 - jc_1 + kd_1) (-a_2 + ib_2 + jc_2 - kd_2) 
\\
& = -a_1a_2 + b_1b_2 + c_1c_2 + d_1d_2 + i(a_1b_2 + b_1a_2 + c_1d_2 - d_1c_2) 
\\
& \quad + j(a_1c_2 - b_1d_2 + c_1a_2 + d_1b_2) + k(-a_1d_2 - b_1c_2 + c_1b_2 - d_1a_2)
\\
& = \psi(\zeta v).
\end{align*}
Therefore, $m=(\varphi,\psi) \in M$, but $m\notin M^\Res$.
\end{enumerate}
\end{remark}

\subsection{Isometric spheres via cocycles}\label{sec_cocycle}

For an element $g\in G^\Res\mminus G_\infty$, we give a characterization of the
isometric sphere and its radius via a cocycle. Using the isometry $\beta\colon
H\to D$ we find for the height function on $H$
the formula 
\[ \height^H(\zeta, v) = \Rea\zeta - \tfrac12|v|^2,\]
and for the Cygan metric ($z_j=(\zeta_j,v_j)$)
\begin{align}\label{cyganH}
\varrho^H(z_1,z_2) = \left| \tfrac12|v_1|^2 + \tfrac12|v_2|^2 + \left| \height^H(z_1) - \height^H(z_2)\right| + \Ima\zeta_1 - \Ima\zeta_2 - \beta_2(v_1,v_2) \right|^{1/2}
\end{align}
The basis point in $H$ is $o^H = (1,0)$. The horospherical coordinates of
$z\in \hg\mminus\{\infty\}$, $z=(\zeta, v)$, are
\[ \left(\height^H(z), \Ima\zeta, \tfrac{1}{\sqrt{2}}v \right)_h.\]
In the following we derive a formula for isometric spheres which uses a cocycle.

Let $Z^1(C)$ act diagonally on $E$ and suppose that $\pi^{(1)}\colon E\to
E/Z^1(C)$ is the canonical projection. Further set 
\[
\pi_H^{(1)}\sceq \tau_H\circ\pi\circ \big(\pi^{(1)}\big)^{-1}\colon \big(E_-(\Psi_2)\cup E_0(\Psi_2)\big)/Z^1(C) \to \hg.
\]
Recall the section $i_H$ of $\pi_H$ from \eqref{def_iH} and set 
\[
i_H^{(1)}\sceq \pi^{(1)}\circ i_H \colon \hg \to E/Z^1(C).
\]
Further recall the isomorphism $\varphi_H\colon G^\Res\to\PU(\Psi_2,C)$ from
\eqref{def_varphiH}. For each $g\in G^\Res$ the map $\varphi_H(g)\colon E\to E$
induces canonically a map $E/Z^1(C)\to E/Z^1(C)$, which we denote by
$\varphi_H(g)$ as well. For all $g\in G^\Res$ the diagram
\[
\xymatrix{
E/Z^1(C) \ar[r]^{\varphi_H(g)} & E/Z^1(C) \ar[d]^{\pi_H^{(1)}}
\\
\hg \ar[u]^{i_H^{(1)}} \ar[r]^g & \hg
}
\]
commutes, but the diagram
\[
\xymatrix{
E/Z^1(C) \ar[r]^{\varphi_H(g)} & E/Z^1(C) 
\\
\hg \ar[u]^{i_H^{(1)}} \ar[r]^g & \hg \ar[u]_{i_H^{(1)}}
}
\]
in general not. The second diagram gives rise to the cocycle \label{def_cocycle}
\[
j\colon G^\Res\times \hg \to C^\times/Z^1(C)
\]
defined by
\[ \varphi_H(g)\big(i_H^{(1)}(z)\big) = j(g,z) i_H^{(1)}(gz)\quad \forall\, g\in G^\Res\ \forall\, z\in\hg.\]

\begin{lemma}\label{cocycle_R}
 Let $g=n_2\sigma ma_tn_1\in G^{\text{res}}\mminus G_\infty$. Then 
\[ j(g^{-1},\infty) = t^{1/2} \mod Z^1(C).\]
Further, $R(g) = \big|j(g^{-1},\infty)\big|^{-1/2}$.
\end{lemma}

\begin{proof} Suppose that $n_1=(\xi_1, w_1)$. From $n_1^{-1}=( -\xi_1, -w_1)$ 
and $g^{-1} = n_1^{-1}\sigma a_tmn_2^{-1}$ it follows that
\[ g^{-1}\infty = n_1^{-1}\sigma\infty = n_1^{-1}0 =
\left(-\xi_1+\tfrac14|w_1|^2, -\tfrac{1}{\sqrt{2}}w_1 \right).\]
Further
\begin{align*}
\varphi_H\big(g^{-1}\big)\big( i_H(\infty) \big) & = \left( t^{1/2},
t^{1/2}\left(-\xi_1 + \tfrac14|w_1|^2\right), -\tfrac{1}{\sqrt{2}}t^{1/2} w_1
\right) \mod Z^1(C).
\end{align*}
Hence $j(g^{-1},\infty) = t^{1/2} \mod Z^1(C)$. Then $R(g) = t^{-1/4} =
\big|j(g^{-1},\infty)\big|^{-1/2}$. 
\bewend \end{proof}

\begin{proposition}\label{cygan_form}
Let $z_1,z_2\in \hg\mminus\{\infty\}$, $z_j=(\zeta_j,v_j)$. Then 
\[ \varrho^H(z_1,z_2) = \big| \Psi_2\big(i_H(z_1),i_H(z_2)\big) + 2\min\big(\height^H(z_1),\height^H(z_2)\big) \big|^{1/2}.\]
\end{proposition}

\begin{proof}
For all $k_1,k_2\in \R$ we have
\[ k_1 + k_2 - |k_1-k_2| = 2\min(k_1,k_2).\]
This and the definition of $\Psi_2$ show that
\begin{align*}
\left|\vphantom{\height^H(z_2)} \right.&\Psi_2\big(i_H(z_1),i_H(z_2)\big) \left. + 2\min\big(\height^H(z_1),\height^H(z_2) \big) \right| 
\\ & = \Big| \Psi_2\big( (1,\zeta_1,v_1), (1,\zeta_2,v_2)\big) + \height^H(z_1) + \height^H(z_2) - \left|\height^H(z_1) - \height^H(z_2) \right| \Big|
\\
& = \Big| -\beta_1(1,\zeta_2) - \beta_1(\zeta_1,1) + \beta_2(v_1,v_2) + \height^H(z_1) + \height^H(z_2) - \left|\height^H(z_1) - \height^H(z_2) \right|\Big|
\\ & = \Big|\zeta_1 - \height^H(z_1) + \overline\zeta_2 - \height^H(z_2) - \beta_2(v_1,v_2) + \left|\height^H(z_1) - \height^H(z_2) \right| \Big|.
\end{align*}
Since
\begin{align*}
\zeta_1 - \height^H(z_1) & = \zeta_1 - \Rea\zeta_1 + \tfrac12|v_1|^2 = \Ima\zeta_1 + \tfrac12|v_1|^2
\intertext{and}
\overline\zeta_2 - \height^H(z_2) & = \overline\zeta_2 - \Rea\zeta_2 + \tfrac12|v_2|^2 = -\Ima\zeta_2 + \tfrac12|v_2|^2,
\end{align*}
it follows that
\begin{align*}
\left|\vphantom{\height^H(z_2)}\right. &\Psi_2\big(i_H(z_1),i_H(z_2)\big)  + \left. 2\min\left(\height^H(z_1),\height^H(z_2) \right) \right| 
\\ & = \Big| \tfrac12|v_1|^2 + \Ima\zeta_1 + \tfrac12|v_2|^2 - \Ima\zeta_2 - \beta_2(v_1,v_2) + \left|\height^H(z_1) - \height^H(z_2) \right|\Big|
\\ & = \varrho^H(z_1,z_2)^2.
\end{align*}
Since $\varrho^H(z_1,z_2)$ is nonnegative, the claim follows. 
\bewend \end{proof}

For the proof of the following proposition we note that the map
\[
\Psi_2^{(1)}\colon\left\{
\begin{array}{ccl}
E/Z^1(C) \times E/Z^1(C) & \to & C/Z^1(C)
\\
\big( [z_1], [z_2] \big) & \mapsto & \big[ \Psi_2(z_1,z_2) \big]
\end{array}
\right.
\]
is well-defined. In particular, we have
\[
 \big|\Psi_2^{(1)}\big([z_1],[z_2]\big) \big| = \big| \Psi_2(z_1,z_2)\big|
\]
for all $z_1,z_2\in E$. Further, let 
\[
 \bhg H \sceq \big\{(\zeta,v)\in C\oplus V \ \big\vert\ \Rea(\zeta) =
\tfrac12|v|^2 \big\} \cup \{\infty\}
\]
denote the boundary of $H$ in $\hg$.

\begin{proposition}\label{cygan_cocycle}
Let $g\in G^{\text{res}}\mminus G_\infty$ and $z\in\hg\mminus\{\infty, g^{-1}\infty\}$. Then
\[ |j(g,z)|^{1/2} = |j(g^{-1},\infty)|^{1/2}\varrho^H(z,g^{-1}\infty) = \frac{\varrho^H(z,g^{-1}\infty)}{R(g)}.\]
\end{proposition}

\begin{proof}
The second equality is proved by Lemma~\ref{cocycle_R}. For the proof of the
first equality recall from Proposition~\ref{cygan_form} that 
\[
\varrho^H(z,g^{-1}\infty) = \left| \Psi_2\big(i_H(z),i_H(g^{-1}\infty)\big) + 2\min\big(\height^H(z),\height^H(g^{-1}\infty)\big)\right|^{1/2}.
\]
From $g^{-1}\infty \in\bhg H\mminus\{\infty\}$ it follows that
$\height(g^{-1}\infty) = 0$. Hence
\[ \varrho(z,g^{-1}\infty)^2 = \left| \Psi_2\big(i_H(z), i_H(g^{-1}\infty)\big)\right|.\]
Suppose that $gz=(\zeta, v)$. Then 
\begin{align*}
 \Psi_2\big(i_H(gz),i_H(\infty)\big) & = \Psi_2\big((1,gz), (0,1,0)\big)
 \\ & = - \beta_1(1,1) - \beta_1(\zeta,0) + \beta_2(v,0) = -1.
\end{align*}
It follows that
\begin{align*}
\varrho^H(z,g^{-1}\infty)^2 & = 
\left| \Psi_2\big(i_H(z), i_H(g^{-1}\infty)\big)\right|  = \left|\Psi_2^{(1)}\big(i_H^{(1)}(z),i_H^{(1)}(g^{-1}\infty)\big)\right|
\\
& = \left| \Psi_2^{(1)}\big(i_H^{(1)}(z), j(g^{-1},\infty)\varphi_H(g^{-1})i_H^{(1)}(\infty)\big)\right|
\\
& = |j(g^{-1},\infty)|^{-1} \left| \Psi_2^{(1)}\big(i_H^{(1)}(z), \varphi_H(g^{-1})i_H^{(1)}(\infty)\big)\right|
\\ 
& = |j(g^{-1},\infty)|^{-1} \left|\Psi_2^{(1)}\big(\varphi_H(g)i_H^{(1)}(z),i_H^{(1)}(\infty)\big) \right|
\\ 
& = |j(g^{-1},\infty)|^{-1} |j(g,z)| \left|\Psi_2\big(i_H(gz),i_H(\infty)\big) \right|
\\ 
& = |j(g^{-1},\infty)|^{-1} |j(g,z)|.
\end{align*}
This proves the claim. 
\bewend \end{proof}

\begin{proposition}\label{sphereco}
Let $g\in G^{\text{res}}\mminus G_\infty$. Then 
\begin{align*} 
I(g) & = \{ z\in H \mid |j(g,z)|=1 \},
\\  
\Ext I(g) & = \{ z\in H \mid |j(g,z)|>1 \},
\\
\Int I(g) & = \{ z\in H \mid |j(g,z)|<1 \}.
\end{align*}
\end{proposition}

\begin{proof}
This claim follows immediately from comparing Definition~\ref{def_isospheres} and 
Proposition~\ref{cygan_cocycle}. 
\bewend \end{proof}

\subsection{A special instance of Theorem~\ref{fund_region2}}\label{sec_special}

The most frequent appearance of isometric fundamental regions in the literature
is for properly discontinuous subgroups $\Gamma$ of $G^\Res$ for which $\infty$
is an ordinary point and the stabilizer $\Gamma_\infty$ is trivial. We show that
under these conditions $\Gamma$ is of type~(O) and $\Gamma\mminus\Gamma_\infty$
of type~(F). Theorem~\ref{fund_region2} implies the existence of an isometric
fundamental region for $\Gamma$, which will be seen to actually be a
fundamental
domain.

\begin{definition}
Let $U$ be a subset of $\hg$ and $\Gamma$ a subgroup of $G$. Then $\Gamma$ is 
said to \textit{act properly discontinuously} on $U$ if for each compact subset
$K$ of $U$, the set $K\cap gK$ is nonempty for only finitely many $g$ in
$\Gamma$. The group $\Gamma$ is said to be \textit{properly discontinous} 
if $\Gamma$ acts properly discontinuously on $H$.
\end{definition}

\begin{defirem}
Let $\Gamma$ be a properly discontinuous subgroup of $G$ and $z\in H$. The
\textit{limit set} $L(\Gamma)$ of $\Gamma$ is the set of accumulation points of
the orbit $\Gamma z$. Since $\Gamma$ is properly discontinuous, $L(\Gamma)$ is a
subset of $\bhg H$. The combination of Propositions~2.9 and 1.4 in
\cite{Eberlein_ONeill} shows that $L(\Gamma)$ is independent of the choice of
$z$. The \textit{ordinary set} or \textit{discontinuity set} $\Omega(\Gamma)$ of
$\Gamma$ is the complement of $L(\Gamma)$ in $\bhg H$, hence $\Omega(\Gamma)
\sceq \bhg H\mminus L(\Gamma)$.
\end{defirem}

\begin{remark}
Let $g,h\in G^\Res$ and $z\in \hg$. Since $\varphi_H$ is a group isomorphism and $\varphi_H(g)$ is $C$-linear, it follows that
\begin{align*}
\varphi_H(gh)\big( i_H(z) \big) & = \varphi_H(g)\left(\varphi_H(h)\big( i_H(z) \big) \right) 
\\
& = \varphi_H(g)\big( j(h,z) i_H(hz)\big) 
\\
& = j(h,z)\varphi_H(g)\big( i_H(hz) \big) 
\\
& = j(h,z) j(g,hz) i_H(hgz).
\end{align*}
Further,
\[
\varphi_H(gh)\big( i_H(z)\big) = j(gh,z) i_H(ghz).
\]
Thus,
\[
 j(gh,z) = j(h,z)j(g,hz).
\]
\end{remark}

The proof of the following lemma proceeds along the lines of \cite[Section~17]{Ford}.

\begin{lemma}\label{isomradius}
Let $\Gamma$ be a properly discontinuous subgroup of $G^\Res$ such that
$\Gamma_\infty=\{\id\}$ and $\infty\in \Omega(\Gamma)$. Then
\begin{enumerate}[{\rm (i)}]
\item \label{isomradiusi} the set of radii of the isometric spheres of
$\Gamma$ is  bounded from above.
\item \label{isomradiusii} the number of isometric spheres with radius exceeding
a given positive quantity is finite.
\end{enumerate}
\end{lemma}

\begin{proof}
We start by proving some relations between radii and distances of centers of isometric spheres. Let $g,h\in\Gamma\mminus\Gamma_\infty$ such that $g\not= h^{-1}$. The cocycle relation shows that 
\[
 j\big((gh)^{-1},\infty\big) = j(g^{-1},\infty)j(h^{-1},g^{-1}\infty).
\]
Proposition~\ref{cygan_cocycle} yields
\begin{align*}
\big|j\big((gh)^{-1},\infty\big)\big|^{-1/2} & = |j(g^{-1},\infty)|^{-1/2} |j(h^{-1},g^{-1}\infty)|^{-1/2}
\\
& = |j(g^{-1},\infty)|^{-1/2} |j(h,\infty)|^{-1/2} \varrho^H(g^{-1}\infty, h\infty)^{-1}.
\end{align*}
Note that $gh\not=\id$ and therefore $gh\notin\Gamma_\infty$. Then Lemma~\ref{cocycle_R} shows that 
\begin{equation}\label{radius1}
R(gh) = \frac{R(g)R(h^{-1})}{\varrho^H(g^{-1}\infty,h\infty)}.
\end{equation}
Using the same arguments, we find
\[
j(g^{-1},\infty) = j\big(h(gh)^{-1},\infty\big) = j\big( (gh)^{-1},\infty\big) j\big(h,(gh)^{-1}\infty\big)
\]
and therefore
\begin{align}
R(g) & \nonumber = |j(g^{-1},\infty)|^{-1/2} = \big|j\big((gh)^{-1},\infty\big)\big|^{-1/2} \big|j\big(h,(gh)^{-1}\infty\big)\big|^{-1/2} 
\\ \label{radius2}
 &= R(gh) \big|j\big(h,(gh)^{-1}\infty\big)\big|^{-1/2}.
\end{align}
Proposition~\ref{cygan_cocycle} shows that identity
\begin{equation}\label{radius3}
\big|j\big(h,(gh)^{-1}\infty\big)\big|^{1/2} = R(h)^{-1} \varrho^H\big((gh)^{-1}\infty, h^{-1}\infty\big).
\end{equation}
Because $R(h) = R(h^{-1})$, it follows from \eqref{radius1}-\eqref{radius3} that 
\begin{align}
\nonumber \varrho^H\big((gh)^{-1}\infty, h^{-1}\infty\big) & = \big|j\big(h,(gh)^{-1}\infty\big)\big|^{1/2} R(h) = \frac{R(gh)R(h)}{R(g)} 
\\ \label{radius4}
& = \frac{R(h)^2}{\varrho^H(g^{-1}\infty,h\infty)}.
\end{align}
Since $\Gamma$ is properly discontinuous and $\infty\in \Omega(\Gamma)$,
\cite[Proposition~8.5]{Eberlein_ONeill} shows
that there exists an open neighborhood $U$ of $\infty$ in $\hg$ such that
$(\Gamma\mminus\Gamma_\infty)\subseteq \hg\mminus U$. Since $\hg\mminus U$ is
compact, we find $m>0$ such that 
\[
\varrho^H(a\infty, b\infty) < m
\]
for all $a,b\in\Gamma\mminus\Gamma_\infty$. Then \eqref{radius4} shows that 
\begin{equation}\label{radius5}
R(h)^2 = \varrho^H\big( (gh)^{-1}\infty, h^{-1}\infty\big) \varrho^H(g^{-1}\infty,h\infty) < m^2.
\end{equation}
Thus, for each $a\in\Gamma\mminus\Gamma_\infty$ we have $R(a)<m$, which proves \eqref{isomradiusi}.

Now let $k>0$ and suppose that there are $a,b\in\Gamma\mminus\Gamma_\infty$ such
that $I(a)\not= I(b^{-1})$ and $R(a),R(b)>k$. Since $a\not= b^{-1}$,
\eqref{radius1} shows in combination with \eqref{radius5} that 
\[
 \varrho^H(a^{-1}\infty, b\infty) = \frac{R(a)R(b)}{R(ab)} > \frac{k^2}{m}.
\]
This means that the distance between the centers of isometric spheres whose
radii exceed $k$ is bounded from below by $k^2/m$. The centers of all these
isometric spheres are contained in the compact  set $\bhg H\mminus U$, which is
bounded in $C\times V$ and hence also \wrt $\varrho^H$. It follows that there
are only finitely many spheres with radius exceeding $k$. This proves
\eqref{isomradiusii}. 
\bewend \end{proof}

\begin{proposition}\label{fdspecialcase}
Suppose that  $\Gamma$ is a properly discontinuous subgroup of $G^\Res$ such
that $\Gamma_\infty=\{\id\}$ and $\infty\in \Omega(\Gamma)$. Then $\Gamma$ is of
type (O) and $\Gamma\mminus\Gamma_\infty$ of type (F). Moreover,
\[
\fd \sceq \bigcap_{g\in\Gamma\mminus\Gamma_\infty} \Ext I(g)
\]
is a fundamental domain for $\Gamma$ in $H$.
\end{proposition}

\begin{proof}
For each $z\in H$, the map $\varrho^H(\cdot, z)\colon H\to \R$ is continuous.
Therefore each $\varrho^H$-ball is open in $H$.
Lemma~\ref{isomradius}\eqref{isomradiusii} implies that the set $\{ \Int I(g)
\mid g\in\Gamma\mminus\Gamma_\infty\}$ is locally finite. Then
Remark~\ref{specialtypeO} shows that $\Gamma$ is of type (O). Note that here the
subgroup $\langle\Gamma\mminus\Gamma_\infty\rangle$ of $\Gamma$ which is
generated by $\Gamma\mminus\Gamma_\infty$ is exactly $\Gamma$. Let $z\in H$.
Lemma~5 before Theorem~5.3.4 in \cite{Ratcliffe} states that $\Gamma z$ is a closed
subset of $H$. Since $\infty\in\Omega(\Gamma)$, Proposition~8.5 in 
\cite{Eberlein_ONeill} shows that we find an open neighborhood $U$ of
$\infty$ in $\hg$ such that $\Gamma z \subseteq \hg\mminus U$. Now $\hg\mminus
U$ is compact and therefore $\Gamma z$ is so. The height function is continuous,
which shows that the maximum of $\{ \height(gz) \mid g\in \Gamma\}$ exists.
Thus, $\Gamma\mminus\Gamma_\infty$ is of type (F). By
Theorem~\ref{fund_region2}, $\fd$ is a fundamental region for $\Gamma$ in $H$.
By Lemma~\ref{isomradius}\eqref{isomradiusi} the radii of the isometric spheres
of $\Gamma$ are uniformly bounded from above. Hence there is $R>0$ such that the
arc-connected set $\{z\in  H\mid \height^H(z) = R\}$ is contained in $\fd$. Let
$w_1,w_2\in\fd$. Then $w_1+[0,R-\height^H(w_1)]$ and $w_2+[0,R-\height^H(w_2)]$
are contained in $\fd$ by Lemma~\ref{propext}\eqref{propextv}. Hence, there is
an arc in $\fd$ from $w_1$ to $w_2$. This shows that $\fd$ is arc-connected,
and thus connected. 
\bewend \end{proof}

\subsection{Isometric spheres and isometric fundamental regions in the
literature}\label{sec_literature}

For real hyperbolic spaces, definitions of isometric spheres are given at
several places, \eg the original definition of Ford in \cite{Ford} for the
plane, in \cite{Katok_fuchsian} for the upper half plane model and the disk
(ball) model of two-dimensional real hyperbolic space, in \cite{Marden} for
three-dimensional space and in \cite{Apanasov2} (or,
earlier, in \cite{Apanasov}) for arbitrary dimensions. The definition for the
upper half plane model is not equivalent to that in the disk model (see
\cite{Marden}). Ford's definition is that for the upper half plane model. His
definition  has been directly generalized to higher dimensions. Ford
and Apanasov show the existence of isometric fundamental regions for a huge
class of groups.

For complex hyperbolic spaces, the (to the knowledge of the author) only
existing definitions of isometric spheres are given by Parker in \cite{Parker},
Goldman in
\cite{Goldman}, and Kamiya in \cite{Kamiya}. Kamiya also discusses the existence
of isometric fundamental regions for certain groups.

For quaternionic hyperbolic spaces, an investigation of isometric fundamental
regions does not seem to exist. A definition of isometric spheres is provided by
\cite{Kim_Parker}.

In this final section we discuss the relation of the existing definitions
of isometric spheres (for real hyperbolic space only exemplarily) to our
definition of isometric spheres, and we compare the
existing statements on the existence of isometric fundamental regions  to
Theorem~\ref{fund_region2}.

\subsection{Real hyperbolic spaces}\label{sec_real}

Ford \cite{Ford} shows the existence of isometric fundamental regions for
groups of isometries acting on real hyperbolic plane (see \cite[Theorems~15 and 22
in Ch.~III]{Ford}). Ford's definition of fundamental region is not equivalent to
our definition. In fact, each fundamental region in sense of our definition is a
fundamental region in the sense of Ford, but not the other way round. In
particular, the translates of a fundamental region in sense of Ford are not
required to cover the whole space. For this reason the hypothesis of
\cite[Theorems~15, 22]{Ford} are weaker than that of Theorem~\ref{fund_region2}.
However, the following discussion shows that the definition of isometric spheres
in \cite{Ford} is subsumed by our definition.

In \cite{Apanasov2}, the model 
\[
D'=\{ (t,Z)\in \R\times \mf z \mid t > 0\}
\]
of real hyperbolic space is used, and the isometric sphere for an element $g\in
G^\Res$, acting on $D'$, is defined as
\[
I(g) \sceq \{ z\in D' \mid |g'(z)| = 1\}.
\]
More precisely, Apanasov (as all other references) uses the coordinates in the
order $\mf z\times\R$. In particular, his model space for two-dimensional real
hyperbolic space is the upper half plane, whereas $D'$ is the right half plane.
Clearly, this difference has no affect on his definition of isometric sphere.
Lemma~\ref{isoreal} below will show that this definition of isometric spheres is
subsumed by our definition. 

The space $D'$ is the symmetric space in Section~\ref{sec_D} constructed from
the abelian $H$-type algebra $\mf n = (\mf n, \{0\}) = (\mf z,\{0\})$. In
Sections~\ref{sec_proj}-\ref{sec_special} we had to work with the ordered
decomposition $(\{0\},\mf v) = (\{0\},\mf n)$ of $\mf n$. Hence we used the
symmetric space
\[
 D = \big\{ (u,X) \in\R\times \mf v \ \big\vert\  u > \tfrac14|X|^2 \big\},
\]
which is isometric to $D'$. According to \cite{CDKR2}, the isometry from $D'$ to
$D$ is 
\[
\nu\colon\left\{
\begin{array}{ccl}
D' & \to & D
\\
(t,Z) & \mapsto & (t^2 + |Z|^2, 2Z).
\end{array}
\right.
\]
Then the action of an isometry $g\in G^\Res$ on $D'$ is given by $\nu_G(g)
\sceq \nu^{-1}\circ g \circ \nu$. Recall the isometry $\beta\colon H\to D$ from
Section~\ref{sec_projmodels}.

\begin{lemma}\label{isoreal}
Let $z\in D'$ and $g\in G^\Res\mminus G^\Res_\infty$. Then 
\[
\big|j\big(g,\beta^{-1}\circ\nu(z)\big)\big|^{-1} = \big|\nu_G(g)'(z)\big|.
\]
\end{lemma}

\begin{proof}
Let $z=(t,Z)\in D'$ and $g=n_2\sigma m a_s n_1 \in G^\Res\mminus G^\Res_\infty$
with $n_j=(1,w_j)$, $j=1,2$, and $m=(\id,\psi)$. At first we calculate the value
of $\big| j\big( g,\beta^{-1}\circ\nu(z)\big)\big|$. We have
\begin{align*}
\varphi_H(g)i_H\big(\beta^{-1}\circ\nu(z)\big)  & = \varphi_H(g)\big(1, t^2+|Z|^2, \sqrt{2}Z\big)
\\
& = \varphi_H(n_2)\varphi_H(\sigma)\varphi_H(m)\varphi_H(a_s)\varphi_H(n_1)\big(1,t^2+|Z|^2, \sqrt{2}Z\big)
\\
& = \big( s^{1/2}\big( \tfrac14|w_1|^2 + t^2 + |Z|^2 + \beta_2(Z,w_1) \big), *, *\big).
\end{align*}
Thus
\[
\big|j\big(g,\beta^{-1}\circ\nu(z)\big)\big| = s^{1/2}\big| \tfrac14|w_1|^2 + t^2 + |Z|^2 + \beta_2(Z,w_1) \big|.
\]
Note that $Z$ and $\beta_2(Z,w_1)$ have to be seen as element of $D$ (not of $D'$). Hence $\beta_2(Z,w_1)\in \R$. Therefore $\beta_2(Z,w_1) = \langle Z,w_1\rangle$ and further
\[
\big|j\big(g,\beta^{-1}\circ\nu(z)\big)\big| = s^{1/2} \big| t^2 + \big|\tfrac12w_1 + Z\big|^2  \big| = s^{1/2} \big|\big(t, \tfrac12w_1+Z\big) \big|^2.
\]
Let $(u,W)\in \R\times \mf z$. For the derivative $\nu_G(g)'(z)$ we find
\begin{align*}
\nu_G&(g)'(z)(u,W)  =  -s^{-1/2}\big(t,\psi(\tfrac12w_1+Z)\big)^{-1}\big(u,\psi(W)\big)\big(t,\psi(\tfrac12w_1+Z)\big)^{-1}.
\end{align*}
Then 
\begin{align*}
\big|\nu_G(g)'(z)(u,W)\big|& = s^{-1/2}\big|\big(t,\psi\big(\tfrac12w_1+Z\big)\big)\big|^{-1}\cdot \big|\big(u,\psi(W)\big)\big| \cdot \big|\big(t,\psi\big(\tfrac12w_1+Z\big)\big) \big|^{-1}
\\
& = s^{-1/2}\big|\big(t,\tfrac12w_1+Z\big) \big|^{-2} \big|(u,W)\big|.
\end{align*}
Thus, 
\begin{align*}
\big|\nu_G(g)'(z)\big|  = s^{-1/2}\big|\big(t,\tfrac12w_1+Z\big) \big|^{-2} = \big|j\big(g,\beta^{-1}\circ\nu(z)\big)\big|^{-1}.
\end{align*}
This completes the proof. 
\bewend \end{proof}

Lemma~\ref{isoreal} and Proposition~\ref{sphereco} immediately imply the
following characterization of exteriors of isometric spheres.

\begin{proposition}
Let $g\in G^\Res\mminus G^\Res_\infty$. Then 
\[
 \Ext I(g) = \big\{ z\in D' \ \big\vert\ |\nu_G(g)'(z)|<1 \big\}.
\]
\end{proposition}

Hence \cite[Theorem~2.30]{Apanasov2} is  a special case of
Proposition~\ref{fdspecialcase}. Lemma~2.31 in \cite{Apanasov2} states an extension of
Theorem~2.30 for subgroups $\Gamma$ of $G^\Res$ with $\Gamma_\infty\not=\{\id\}$.
Unfortunately, the hypotheses of \cite[Lemma~2.31]{Apanasov2} are not completely
stated, for which reason we cannot compare this lemma with
Theorem~\ref{fund_region2}.

\subsection{Complex hyperbolic spaces}\label{Kamiya}

The isometric spheres for isometries of complex hyperbolic spaces in
\cite{Kamiya} are identical to those in \cite{Goldman}. Kamiya uses the model
$H$ and defines the Cygan metric by formula \eqref{cyganH}.  Recall the map
$\varphi_H$ from Section~\ref{sec_liftedisos}. Let $f\in
\PU(1,n;\C) = \PU(\Psi_2,\C)$ such that $\varphi_H^{-1}(f)$ does not fix
$\infty$. Further suppose
that $(a_{ij})_{i,j=1,\ldots,n+1}$ is a matrix representative of $f$.
Then Kamiya defines the isometric sphere of $f$ to be the set
\[
I(f) \sceq \big\{ z\in H \ \big\vert\ \varrho(z,\varphi_H^{-1}(f^{-1})\infty) =
R_f \big\}
\]
where $R_f \sceq |a_{12}|^{-1/2}$. One easily proves that this definition
does not depend on the choice of the matrix representative. The following lemma
shows that our definition of isometric spheres covers this one. 

\begin{lemma}\label{isocomplex}
Suppose that $f=(a_{ij})_{i,j=1,\ldots,n+1} \in \PU(1,n;\C)$ with
$\varphi_H^{-1}(f)(\infty) \not=
\infty$. Then $R_f= R\big(\varphi_H^{-1}(f)\big)$.
\end{lemma}

\begin{proof} Set $g\sceq \varphi_H^{-1}(f)$.
We have
\[
f\begin{pmatrix} 0 \\ 1 \\ 0 \end{pmatrix} = \begin{pmatrix} a_{12} \\ \vdots \\ a_{n+1,2} \end{pmatrix}.
\]
Therefore 
\[
|j(g,\infty)| = |a_{12}|.
\]
Since $|j(g^{-1},\infty)| = |j(g,\infty)|$,  Lemma~\ref{cocycle_R} shows that
$R(g) = |a_{12}|^{-1/2}$. 
\bewend \end{proof}

\cite[Theorem~3.1]{Kamiya} states the existence of isometric fundamental domains
for discrete subgroups $\Gamma$ of $G^\Res$ for which, after possible
conjugation of $\Gamma$, we have $\infty\in\Omega(\Gamma)$ and
$\Gamma_\infty=\{\id\}$.  By \cite[Theorem~5.3.5]{Ratcliffe}, $\Gamma$ is
discrete if and only if $\Gamma$ is properly discontinuous. Therefore, Kamiya's
Theorem is a special case of Proposition~\ref{fdspecialcase}.

In \cite{Parker}, Parker uses a section of the projection map from
$E\mminus\{0\}$ to horospherical coordinates which is reminiscent of the ball
model. Therefore, we expect that, as in the real case, this definition is not
equivalent to the definition from \cite{Kamiya}.

\subsection{Quaternionic hyperbolic spaces}

In \cite{Kim_Parker}, Kim and Parker propose a definition of isometric spheres
for isometries in $G^\Res\mminus G^\Res_\infty$ of quaternionic hyperbolic
space. They use the model $H$ and horospherical coordinates for the definition.
 
With the bijection
\[
\left\{
\begin{array}{ccl}
\mf z \times \mf v & \to & \mf z\times \mf v
\\[1mm]
(Z,X) & \mapsto & \big(Z,\tfrac12 X\big),
\end{array}
\right.
\]
our Heisenberg group $N$ and our Cygan metric is transfered into their one.
After shuffling coordinates, they characterize (see
\cite[Proposition~4.3]{Kim_Parker}) the isometric sphere $I(g)$ of an element
$g=(a_{ij})_{i,j=1,\ldots, n+1}$ in $\PSp(n,1;\h)$ with
$\varphi_H^{-1}(g)(\infty)\not=\infty$ as
\[
I(g) = \big\{ z\in H\ \big\vert\ \varrho(z, \varphi_H^{-1}(g^{-1})\infty) =
\sqrt{2}\cdot |a_{12}|^{1/2} \big\}.
\]
They use a slightly different indefinite form on $E$ for the definition of the
hyperbolic space. Despite this difference we can apply the calculation in
Section~\ref{Kamiya} we see that our definition of isometric sphere provides
$|a_{12}|^{1/2}$ as radius. The factor $\sqrt{2}$ in \cite{Kim_Parker} is due to
the factor $\tfrac12$ in their choice of the section of the projection from
$E_-(\Psi_2)$ to $H$.


\end{document}